\documentclass[oneside,final]{amsart}

\usepackage[T1]{fontenc}
\usepackage{geometry}

\usepackage[utf8]{inputenc}
\usepackage{hyperref}
\usepackage[english]{babel}
\usepackage{enumerate}
\usepackage{geometry}
\usepackage{aliascnt}
\usepackage{amssymb}
\usepackage{amsmath}
\usepackage{amsthm}
\usepackage{verbatim}
\usepackage[normalem]{ulem}
\usepackage{mathtools}
\usepackage{esint} 
\usepackage{nth}
\usepackage{ragged2e}

\usepackage[noabbrev,capitalize]{cleveref}

\usepackage[right]{showlabels}

\usepackage{tikz}
\usepackage{tikzscale}
\usepackage{graphicx}
\usepackage[labelformat=simple]{subcaption}
\usepackage{mathscinet} 
\usepackage[title]{appendix}

\usepackage{color}

\numberwithin{equation}{section}

\usepackage{dsfont} 
\usepackage{hyperref}

\usepackage{autonum} 
\allowdisplaybreaks[4]

\theoremstyle{plain}

\newtheorem{proposition}{Proposition}[section]

\newaliascnt{lemma}{proposition} 

\newtheorem{lemma}[lemma]{Lemma}

\aliascntresetthe{lemma} 
\Crefname{lemma}{Lemma}{Lemmas}

\newaliascnt{theorem}{proposition} 
\newtheorem{theorem}[theorem]{Theorem}

\aliascntresetthe{theorem} 

\newaliascnt{corollary}{proposition} 
\newtheorem{corollary}[corollary]{Corollary}

\aliascntresetthe{corollary}

\newaliascnt{hypothesis}{proposition} 
\newtheorem{hypothesis}[lemma]{Hypothesis}

\aliascntresetthe{hypothesis} 

\theoremstyle{definition}

\newaliascnt{definition}{proposition} 
\newtheorem{definition}[definition]{Definition}

\aliascntresetthe{definition} 
\Crefname{definition}{Definition}{Definitions}

\newaliascnt{problem}{proposition} 

\aliascntresetthe{problem} 

\newaliascnt{example}{proposition} 
\newtheorem{example}[example]{Example}

\aliascntresetthe{example} 

\theoremstyle{remark}

\newaliascnt{remark}{proposition} 
\newtheorem{remark}[remark]{Remark}

\aliascntresetthe{remark} 

\newtheorem{claim}{Claim}

\makeatletter
\@addtoreset{claim}{proposition}
\makeatother

\def\equationautorefname~#1\null{%
	(#1)\null
}

\newcommand{\R}{\mathbb{R}}
\newcommand{\C}{\mathbb{C}}
\newcommand{\Z}{\mathbb{Z}}
\newcommand{\N}{\mathbb{N}}
\newcommand{\E}{\mathcal{E}}

\newcommand{\curv}{\vec{\kappa}}
\newcommand{\scurv}{\kappa}
\renewcommand{\S}{\mathbb{S}}

\newcommand{\A}{\mathcal{A}}

\newcommand{\T}{\mathbb{T}^2}
\renewcommand{\H}{\mathbb{H}}
\newcommand{\W}{\mathcal{W}}

\newcommand{\Ll}{\mathcal{L}}

\newcommand{\diam}{\mathrm{diam}}
\newcommand{\measurerestr}{%
  \,\raisebox{-.127ex}{\reflectbox{\rotatebox[origin=br]{-90}{$\lnot$}}}\,%
}

\newcommand*{\dd}{\mathop{}\!\mathrm{d}}

\def\nicefrac#1#2{%
    \raise.5ex\hbox{$#1$}%
    \kern-.15em/\kern-.05em%
    \lower.25ex\hbox{$#2$}}

\title[\textsc{Dimension reduction for Willmore flows of tori}]{\Large Dimension reduction for Willmore flows of tori: \\ fixed conformal class and analysis of singularities}

\author[A.~Dall'Acqua]{Anna Dall'Acqua}
\address[A.~Dall'Acqua]{Institute for Applied Analysis, Ulm University, Helmholtzstraße 18, 89081 Ulm, Germany.}
\email{anna.dallacqua@uni-ulm.de}
\author[M.~Müller]{Marius Müller}
\address[M.~Müller]{University of Augsburg, Institute for Mathematics, 86135 Augsburg, Germany.}
\email{marius1.mueller@uni-a.de}
\author[F.~Rupp]{Fabian Rupp}
\address[F.~Rupp]{Faculty of Mathematics, University of Vienna, Oskar-Morgenstern-Platz 1, 1090 Vienna, Austria.}
\email{fabian.rupp@univie.ac.at}
\author[M.~Schlierf]{Manuel Schlierf}
\address[M.~Schlierf]{Institute for Applied Analysis, Ulm University, Helmholtzstraße 18, 89081 Ulm, Germany.}
\email{manuel.schlierf@uni-ulm.de}

\begin{document}

\begin{abstract}
       This work studies Willmore flows of tori and their singularities via a dimension reduction approach. We introduce a Willmore flow that preserves the degenerate constraint of prescribed conformal class and, for rotationally symmetric initial data, we establish a strong relation with the length-preserving elastic flow in the hyperbolic plane. We provide a necessary condition for singularities and a criterion for the initial datum that allows to exclude them. Our results allow for initial data with arbitrarily large energy, in particular exceeding the usual Li--Yau threshold of $8\pi$. As an application, we obtain existence of a new class of conformally constrained Willmore tori. Moreover, we investigate singularities of the classical Willmore flow. For a class of tori, we identify a non-smooth object, the inverted catenoid, as the limit shape and we show that the flow can be restarted at this singular surface and converges to a round sphere.
\end{abstract}

\begingroup
\def\uppercasenonmath#1{\scshape} 
\maketitle
\endgroup



\section{Introduction}

The \emph{Willmore energy} of a closed immersed surface $f\colon\Sigma\to\R^3$ is given by
\begin{equation}\label{eq:def_Willmore_energy}
    \W(f) \vcentcolon= \int_{\Sigma} H^2\dd\mu
\end{equation}
where $\mu$ is the Riemannian measure on $\Sigma$ induced by the pull-back $f^*\langle\cdot,\cdot\rangle$ of the Euclidean metric
and $H=\frac{\kappa_1+\kappa_2}{2}$ is the mean curvature. 
The energy $\mathcal{W}$ was already 
studied by Blaschke and Thomsen \cite{blaschke} who observed that, for $\Sigma$ closed, i.e., compact and without boundary, the Willmore energy is invariant with respect to smooth \emph{Möbius transformations} of the ambient space, making it an important concept in conformal geometry. It became more popular after the work of Willmore \cite{MR0202066} who proved $\mathcal{W}(f)\geq 4\pi$ with equality only for round spheres.

Another crucial energy value is the threshold of $8\pi$ which, by an inequality of Li and Yau \cite{liyau1982}, guarantees embeddedness.
This multiplicity control makes $\mathcal{W}(f)<8\pi$ an ubiquitous assumption to rule out \emph{bubbling singularities} that are typical for functionals of critical growth, see \cite{simon1993,kuwertschaetzle2004,kuwertli2012,MR3276154}. 

\subsection{Conformally constrained minimization and Willmore flow}
The minimization of the Willmore energy among surfaces in a fixed \emph{conformal class} is a particularly natural geometric framework. In \cite{kuwertschaetzle2013}, existence of a smooth minimizer was shown, provided the infimum is strictly below $8\pi$, ruling out \emph{branch point singularities}, see also \cite{MR3276154} and \cite{kuwertli2012}. For  tori, the conformal class of $f\colon \T\to\R^3$ can be identified with some $\vec\omega\in \R^2$, see \Cref{sec:conform_constr}. The Clifford torus, the global minimizer of $\mathcal
W$ among tori \cite{MR3152944}, has conformal class $\vec\omega=(0,1)$. Together with \cite{kuwertschaetzle2013}, the construction in \cite{MR3210758} implies existence of Willmore minimizing tori in all \emph{rectangular} conformal classes, i.e., where $\vec\omega\in \R\cdot(0,1)$,
see also \cite[Theorem 4.5]{dallacquamullerschatzlespener2020}.
In a suitable neighborhood of  $\vec\omega =(0,1)$, conformally constrained Willmore minimizers have been identified in \cite{MR3403429,MR4270047}.
The Euler--Lagrange equation for the conformally constrained Willmore functional is of the form
\begin{align}\label{eq:intro:conf_Willmore_eq}
\Delta H + |A^0|^2 H = \langle q,A^0\rangle.
\end{align}
Here $A^0$ denotes the trace-free part of the (scalar) second fundamental form of $f$ and $q$ is a transverse traceless symmetric $2$-covariant tensor with respect to $f^*\langle\cdot,\cdot\rangle$. Solutions to \eqref{eq:intro:conf_Willmore_eq} are called \emph{conformally constrained Willmore immersions,} see \cite{bohlepeterspinkall2008,MR3199732,MR3276154}.

A dynamic minimization approach for the Willmore energy is given by the \emph{Willmore flow}, the $L^2(\dd \mu)$-gradient flow of $\mathcal{W}$, i.e., the geometric evolution equation
\begin{align}\label{eq:intro:WF}
\partial_tf = - \nabla \mathcal{W}(f) \vcentcolon= - (\Delta H + |A^0|^2H)N,
\end{align}
where $N$ is a local unit normal field. This is a quasilinear parabolic PDE of fourth order, thus no maximum principles are available.
The mathematical analysis of the Willmore flow started with the work of Simonett \cite{simonett2001} and the seminal papers by Kuwert--Schätzle \cite{kuwertschaetzle2002,kuwertschaetzle2001}, who established a \emph{curvature-concentration}-based blow-up procedure, eventually leading to a global existence and convergence result for topological spheres with initial energy below $8\pi$ \cite{kuwertschaetzle2004}. This threshold is sharp 
\cite{blatt2009},  but whether singularities occur in finite or infinite time is still open. Similar global existence and convergence results have been obtained for Willmore flows constrained to \emph{level set manifolds} of geometric functionals \cite{rupp2023,rupp2021} and the \emph{Möbius invariant Willmore flow} \cite{jakob2024singularitiesconvergencemobiusinvariantwillmore}.
In \cite{dallacquamullerschatzlespener2020}, the first global existence and convergence result in the higher genus case has been obtained for tori, provided the initial datum has a \emph{rotational symmetry}. Nevertheless, in all these works, the initial energy threshold of $8\pi$ persists and it is sharp for spheres \cite{blatt2009} and tori \cite{dallacquamullerschatzlespener2020}, even assuming rotational symmetry.

Another approach to the Willmore flow has been introduced in \cite{MR4385140}, based on the \emph{parametric framework} of weak immersions in \cite{kuwertli2012,MR3276154}.
This flow yields immersions $f(t)\colon \Sigma\to\R^3$ which, at any given time $t$, are conformal with respect to the canonical \emph{Poincar\'e metric} \cite[Theorem 4.4.1]{jost2006}
$g_{\mathrm{Poin}}(t)$ on $\Sigma$ if $\operatorname{genus}\Sigma\geq 1$. 
However, in general, the conformal class varies along this flow.

\subsection{The Willmore flow with fixed conformal class}
One contribution of this paper is a new notion of Willmore flow for tori which preserves the conformal class. Denoting by $\vec\omega \in \R^2$ the conformal class of an immersed torus $f\colon\T\to\R^3$, we find that \eqref{eq:intro:conf_Willmore_eq} is equivalent to 
\begin{align}
\Pi^{\overline{\mathfrak{C}_f(\vec\omega)}}\nabla \mathcal{W}(f) = 0.
\end{align}
Here $\Pi^{\overline{\mathfrak{C}_f(\vec\omega)}}$ denotes the $L^2(\dd \mu)$-orthogonal projection onto the space of infinitesimal variations preserving the conformal class, see \Cref{sec:conform_constr} for a detailed discussion. We thus introduce the \emph{conformally constrained Willmore flow} as the geometric evolution equation
\begin{align}\label{eq:intro:conf_WF}
\partial_t f = -\Pi^{\overline{\mathfrak{C}_f(\vec\omega)}}\nabla \mathcal{W}(f).
\end{align}
While, as the projected Willmore-type flows in \cite{rupp2023,rupp2021}, \eqref{eq:intro:conf_WF} formally is a quasilinear nonlocal parabolic system of fourth order, its analysis is significantly more complicated by the intricate nature of the constraint of fixed conformal class and its potential degeneracy, a major issue also in the static problem \cite{kuwertschaetzle2013,MR3276154,MR3199732}. Indeed, this constraint degenerates at all \emph{isothermal immersions}, including, in particular, constant mean curvature immersions or surfaces of revolution, see \cite{bohlepeterspinkall2008}.

Despite this degeneracy, our first main result provides a characterization of the long-time behavior of \eqref{eq:intro:conf_WF} for rotationally symmetric initial data.

\begin{theorem}\label{intro:thm:main-conf-constr-wf}
    Let $f_0\colon\T\to\R^3$ be a rotationally symmetric torus. Then there exists a maximal solution $f\colon[0,T)\times\T\to\R^3$ to the conformally constrained Willmore flow such that at least one of the following is true.
    \begin{enumerate}[(i)]
        \item\label{item:intro:thm_conformal_case_1} 
        There exists $t_j\nearrow T$ along which
        the hyperbolic total curvature of the profile curves converges to zero.
        \item\label{item:intro:thm_conformal_case_2} We have $T=\infty$ and, after conformal reparametrization, the flow smoothly converges to a conformally constrained Willmore torus.
    \end{enumerate}
    Moreover, if $\mathcal{W}(f_0)\leq 8\pi$ or if the initial profile curve has hyperbolic length $L$ and turning number $m\in\Z$ such that 
    $L< 2\pi|m|$, then case \eqref{item:intro:thm_conformal_case_2} occurs.
\end{theorem}

In contrast to its Euclidean counterpart, the total hyperbolic curvature arising in \Cref{intro:thm:main-conf-constr-wf} is not a regular homotopy invariant and thus not necessarily preserved under the flow.

We highlight that \Cref{intro:thm:main-conf-constr-wf}\eqref{item:intro:thm_conformal_case_1} gives a necessary criterion for singularities of the conformally constrained Willmore flow \eqref{eq:intro:conf_WF}. However, this criterion is not sufficient, see \Cref{rem:lambdafigure8}.  Moreover, the last part of the statement not only enables us to go beyond the usual energy threshold of $8\pi$ for the initial datum, but it in fact allows us to find a class of globally existent and convergent solutions to \eqref{eq:intro:conf_WF} with arbitrarily large initial energy, see \Cref{rem:circular-elastica}. 

We emphasize that the methods used in the non-constrained Willmore flow of tori \cite{dallacquamullerschatzlespener2020} are not directly applicable to prove \Cref{intro:thm:main-conf-constr-wf}. 
Indeed, these techniques 
are based on the concentration-compactness alternative of Kuwert--Schätzle \cite{kuwertschaetzle2002,kuwertschaetzle2001,kuwertschaetzle2004} which is not available in this setup.

Instead, our method exploits the inherent one-dimensional structure of the problem due to its rotational symmetry to reduce the dimension. Since the flow preserves symmetry of the initial datum, we may describe the conformally constrained Willmore flow entirely by the dynamics of the associated \emph{profile curves}. The Willmore energy of the surface of revolution $f_{\gamma}$ is related to the  \emph{(hyperbolic) elastic energy} of the profile curve $\gamma\colon\S^1\to\mathbb{H}^2$ given by
\begin{align}\label{def:elastenergyhyp}
    \E(\gamma) \vcentcolon= \int_{\S^1} |\vec\kappa|^2_g \dd s.
\end{align}
Here $\vec\kappa$ is the curvature of $\gamma$ in the hyperbolic plane $(\H^2,g)$. As a consequence of the Bryant--Griffiths-formula \cite{bryantgriffiths1986}, i.e.,
\begin{align}\label{eq:BG}
    \frac{2}{\pi} \mathcal{W}(f_\gamma) = \E (\gamma),
\end{align}
it is not too surprising that \eqref{eq:intro:conf_WF} is somehow related to the \emph{elastic flow} of curves in the hyperbolic plane \cite{dallacquaspener2017}, see \eqref{eq:elasflowhyp} below. However, this correspondence does not yield the $L^2(\dd s)$-gradient flow, but a gradient flow with respect to a metric that degenerates for curves approaching the axis of revolution.

The key observation for our analysis is the following. Fixing the conformal class is, after our dimension reduction, equivalent to fixing the hyperbolic length. Due to this fact the constrained gradient flow can be obtained with a \emph{scalar} time-dependent Lagrange multiplier.
Nevertheless, due to the degenerate metric,
we cannot directly rely on classical energy estimates and interpolation inequalities that are commonly used for elastic flows \cite{dziukkuwertschaetzle2002,dallacquaspener2017}. We overcome this issue by deriving suitable \emph{weighted energy estimates} to compensate the degeneracy. If the total hyperbolic curvature does not vanish, cf.\ case \eqref{item:intro:thm_conformal_case_1} in \Cref{intro:thm:main-conf-constr-wf}, we may reduce the order of the scale-invariant and thus critical Lagrange multiplier by appropriately testing the evolution. The proof of \Cref{intro:thm:main-conf-constr-wf} is then based on a \emph{lifespan bound} for the flow in terms of the distance to the rotation axis and a new \emph{iterative blow-up scheme} which yields control over all derivatives of the curvature and thus sub-convergence as $t\to\infty$. Lastly, since the fixed hyperbolic length constraint induces a Hilbert manifold structure we may employ a \emph{constrained {\L}ojasiewicz--Simon gradient inequality} \cite{rupp2020} to conclude full convergence.

As a byproduct of our analysis, we obtain a new proof of the convergence result in \cite{dallacquamullerschatzlespener2020} for the (non-constrained) Willmore flow of tori of revolution in a purely one-dimensional way, leading to further insights also about the non-constrained Willmore flow. In particular, we can  explicitly describe the reparametrizations needed to obtain full convergence as $t\to\infty$ in \cite[Theorems 1.2 and 1.3]{dallacquamullerschatzlespener2020}, see \Cref{cor:better-sub-convergence} and \Cref{sec:conv-unconstr-wf}.

\Cref{intro:thm:main-conf-constr-wf} can be used to prove existence of infinitely many new critical points with energy above the $8\pi$-threshold assumed in \cite{kuwertschaetzle2013} to avoid singularities.
Recall that the tori $f_r\colon \R^2/2\pi\Z^2\to\S^3, f_r(u,v)\vcentcolon= (r e^{i u}, \sqrt{1-r^2}e^{i v})\in \S^3\subset \C^2= \R^4$, $r\in (0,1)$, are conformally constrained Willmore tori, which, for $r$ close to $1/\sqrt{2}$, minimize the Willmore functional in their conformal classes \cite{MR3403429} or are at least stable \cite{MR3055802}. 
\begin{corollary}\label{cor:intro_new_conf_willmore_tori}
    Let $m\in\N, m\geq 3$. Then there exists a conformally constrained Willmore torus $f\colon \T\to\S^3$ regularly homotopic to $f_{m,r}\colon \R^2/2\pi\Z^2\to\S^3, f_{m,r}(u,v)=f_r(u,mv)$, for some $r\in (0,1)$, and in the same conformal class with $8\pi \leq \mathcal{W}(f)<\mathcal{W}(f_{m,r})$.
\end{corollary}

This is a consequence of the observation that after stereographic projection into $\R^3$, the tori $f_{m,r}$ yield surfaces of revolution, corresponding to an $m$-fold circle of appropriate radius. For $m\geq 3$, these tori are unstable for $\mathcal{E}$ as a consequence of the computations in \cite{langersinger1984}.
Starting the conformally constrained Willmore flow with a suitable perturbation of these $m$-fold circles, we may verify the assumptions in \Cref{intro:thm:main-conf-constr-wf} and deduce \Cref{cor:intro_new_conf_willmore_tori}. The range of admissible $m,r$ in the proof of \Cref{cor:intro_new_conf_willmore_tori} and the resulting rectangular conformal classes
can be made explicit, see \Cref{rem:admissible_conf_class}.

\subsection{Analysis of singularities}
As our second main contribution, we further investigate singularities of the non-constrained Willmore flow of tori of revolution. Our preceding dimension reduction allows us to identify the \emph{inverted catenoid} as the limit of a singular Willmore flow for a class of figure-eight-shaped profile curves 
assuming global existence and non-degeneracy of the area,
see \Cref{prop:sing}. While we cannot verify these assumptions analytically, they are supported by numerical experiments \cite{barrettgarckenuernberg2021}, see also \Cref{rem:numerics} below. This suggests that a \emph{loss of topology} occurs in the Willmore flow, since the inverted catenoid is no longer an immersed torus, but rather a $C^{1,\alpha}$-immersed sphere with Willmore energy exactly $8\pi$. As the inverted catenoid is not smooth enough to apply Kuwert--Schätzle's results \cite{kuwertschaetzle2002,kuwertschaetzle2004}, we construct a suitable smooth approximation and use a recent \emph{local $C^{1,\alpha}$-well-posedness result} by LeCrone--Shao--Simonett \cite{lecroneshaosimonett2020} to find a spherical Willmore flow that starts at the inverted catenoid and converges to a round sphere. 

\begin{theorem}\label{thm:intro:wf-starting-in-ic}
    There exists a parametrization $f_0$ of an inverted catenoid with $f_0\in C^{1,\alpha}(\mathbb{S}^2,\mathbb{R}^3)$ for all $\alpha\in(0,1)$ and a family of immersions $f\in C^{\infty}((0,\infty)\times\S^2,\R^3)$ with
    \begin{align}
        \|f(t)-f_0\|_{C^{1,\alpha}(\S^2)}\to0&\quad\text{ for $t\searrow 0$, and}\\
        \W(f(t))<8\pi \text{ for $t>0$},&\quad \W(f(t))\nearrow 8\pi \quad\text{ for $t\searrow 0$},\label{eq:wf-starting-in-ic-en-cont}
    \end{align}
    satisfying 
    \begin{equation}
        \begin{cases}
            \partial_t^{\bot} f = -\nabla\W(f)&\text{in $(0,\infty)\times\S^2$},\\
            f(0)=f_0&\text{in $\S^2$}
        \end{cases}
    \end{equation}
    such that $f(t)$ smoothly converges to a round sphere for $t\to\infty$ up to reparametrization.
\end{theorem}

Combined with the numerical experiments in \cite{barrettgarckenuernberg2021}, our result suggests the existence of a Willmore-energy-reducing evolution of surfaces, starting at a rotationally symmetric smooth torus with figure-eight-shaped profile curve, \emph{changing topology} by passing through the inverted catenoid, and continuing to flow towards a round sphere. We believe that this provides an essential example to keep in mind for further investigations of singularities of the Willmore flow, particularly in the context of developing \emph{surgery theory} or concepts of \emph{weak solutions}.

\subsection{Structure of this article}
In \Cref{sec:preliminaries}, we review some geometric connections between tori of revolution and their profile curves. In particular, we give a precise definition of the conformally constrained Willmore flow of tori \eqref{eq:intro:conf_WF} in \Cref{sec:conform_constr}. In \Cref{sec:2}, we present the essential ingredient of our one-dimensional analysis, the weighted curvature estimates, in a concise unified approach treating both the Willmore flow and the conformally constrained Willmore flow simultaneously. Special attention is given in \Cref{sec:a-priori-lambda-control} to reduce the order of the non-local Lagrange multiplier in the case of the conformally constrained Willmore flow. In \Cref{sec:convergence}, we then use these tools to set up an iterative blow-up scheme, see \Cref{sec:iterative_blow_up}, which allows us to prove \Cref{intro:thm:main-conf-constr-wf} and \Cref{cor:intro_new_conf_willmore_tori}. Lastly, in \Cref{sec:catenoid}, we examine the singularities of the (non-constrained) Willmore flow of tori of revolution and prove \Cref{thm:intro:wf-starting-in-ic}. 
The paper is complemented by an appendix that includes the local $C^{1,\alpha}$-well-posedness for the Willmore flow (\Cref{sec:app:well-posedness}).

\section{Geometric preliminaries}\label{sec:preliminaries}

We first recall the geometric quantities induced by a general immersion $f\colon \Sigma\to\R^3$ of an oriented closed surface $\Sigma$. Unless specified otherwise, we assume all geometric objects to be smooth, i.e., of class $C^\infty$.
We endow $\Sigma$ with the pull-back $g_f \vcentcolon= f^*\langle\cdot,\cdot\rangle$ of the Euclidean metric. In local coordinates with respect to the orientation, the \emph{unit normal} along $f$ is given by
\begin{align}\label{eq:defNormal}
N=N_f \vcentcolon= \frac{\partial_1 f\times \partial_2f}{|\partial_1 f\times \partial_2 f|}.
\end{align}
Here and in the following, $| \cdot|$ denotes the Euclidean norm in any $\R^n$, $n \geq 1$.
The \emph{(scalar) second fundamental form} is $A_{ij} \vcentcolon= \langle \partial_{ij}^2 f,N\rangle$, $i,j=1,2$. With $g_f^{ij}$ denoting the components of the inverse of the metric tensor $g_f$, the \emph{mean curvature} can be computed as $H\vcentcolon=\frac{1}{2}g_f^{ij}A_{ij}$ and the tensor $A^0 \vcentcolon= A - Hg_f$ is called \emph{trace-free second fundamental form.} We denote by $\mu$ the \emph{area measure} induced by $g_f$. 
For a smooth family of immersions $f\colon[0,T)\times \Sigma\to\R^3$, with the Laplace--Beltrami operator $\Delta$ with respect to $g$, the first variation of the Willmore energy \eqref{eq:def_Willmore_energy} is
\begin{equation}
    \frac{\dd}{\dd t} \W(f) = \int_{\Sigma} \langle \nabla\W(f),\partial_t f\rangle\dd\mu,
\end{equation}
see \cite[pp. 9 -- 12]{kuwertschaetzle2012}.

\subsection{Tori of revolution and hyperbolic geometry}
A particular focus in this article are \emph{tori of revolution} in $\R^3$. Let $\T=\S^1\times\S^1$ be the standard torus with orientation $\{\frac{\partial}{\partial x_1},\frac{\partial}{\partial x_2}\}$, where we usually identify $\S^1=\R /(2\pi\Z)$ in the sequel. Given an immersion $\gamma\colon\S^1\to\H^2$ into the \emph{hyperbolic plane} $\H^2=\R\times(0,\infty)$, we denote by $f_{\gamma}\colon \T\to\R^3$ the associated surface of revolution given by
\begin{align}
    f_{\gamma}(x_1,x_2) \vcentcolon=(\gamma^{(1)}(x_1),\gamma^{(2)}(x_1)\cos(x_2),\gamma^{(2)}(x_1)\sin(x_2)).\label{eq:def-surf-rev}
\end{align}
More generally, for $\varphi\colon \S^1\to\R^2$, we define $\mathcal R\varphi\colon\T\to\R^3$,
\begin{equation}\label{eq:def-R}
    (\mathcal R\varphi)(x_1,x_2)\vcentcolon=(\varphi^{(1)}(x_1),\varphi^{(2)}(x_1)\cos(x_2),\varphi^{(2)}(x_1)\sin(x_2)),
\end{equation}
so that $f_\gamma=\mathcal R\gamma$.
For a surface of revolution parametrized as in \eqref{eq:def-surf-rev}, we have
\begin{align}\label{eq:metric_torus_revolution}
g_{f_\gamma}(x_1,x_2) = \begin{pmatrix}
|\gamma'(x_1)|^2 & 0 \\
0 & (\gamma^{(2)}(x_1))^2
\end{pmatrix},
\end{align}
so that the area measure is given by
\begin{equation}\label{eq:dmu-axisymm}
    \dd\mu_{f_\gamma} = |\partial_x\gamma(x_1)|\gamma^{(2)}(x_1)\cdot  \dd x_1\otimes \dd x_2.
\end{equation}

Many geometric properties of surfaces of revolution as in \eqref{eq:def-surf-rev} may be described by their profile curves in the hyperbolic plane $\H^2=\{(y^{(1)}, y^{(2)}) \in \R^2 : y^{(2)}  > 0\}$  endowed 
with 
the Riemannian metric $g_{(y^{(1)},y^{(2)})} \vcentcolon= \frac{1}{(y^{(2)})^2} \langle \cdot,\cdot\rangle$.
It is well-known that $(\H^2,g)$ has constant sectional curvature equal to $-1 $ and  the covariant derivative of a vector field $X=X(x)$ along a curve $\gamma\colon \S^1\to\H^2, \gamma=\gamma(x)$, is given by
\begin{equation}
 \label{eq:NablaLocalH}
 \nabla _{\partial_x\gamma} X  = \begin{pmatrix}
\partial_x X^{(1)} - \frac{1}{\gamma^{(2)}} (X^{(1)} \partial_x \gamma^{(2)}  + X^{(2)} \partial_x \gamma^{(1)})\\
\partial_x X^{(2)} + \frac{1}{\gamma^{(2)}} (X^{(1)} \partial_x \gamma^{(1)}  - X^{(2)} \partial_x \gamma^{(2)}) \end{pmatrix}.
\end{equation}
With $\partial_s \gamma \vcentcolon= \frac{1}{|\partial_x \gamma|_g} \partial_x \gamma$, the curvature of an immersion $\gamma \colon \mathbb{S}^1 \to \H^2$ is given by
\begin{equation}
 \label{eq:curvature}
 \vec{\kappa} \vcentcolon= \nabla_{\partial_s \gamma} \partial_s \gamma= \begin{pmatrix}
\partial_s^2 \gamma^{(1)} - \frac{2}{\gamma^{(2)}} \partial_s \gamma^{(1)} \partial_s \gamma^{(2)}\\
\partial_s^2 \gamma^{(2)} + \frac{1}{\gamma^{(2)}} ( (\partial_s \gamma^{(1)})^2  -(\partial_s \gamma^{(2)})^2) \end{pmatrix},
\end{equation}
see also \cite[(12)]{dallacquaspener2017}. Writing $\dd s\vcentcolon=|\partial_x\gamma|_g\dd x$ for the arc-length element, the \emph{(hyperbolic) length} is $\mathcal{L}(\gamma)\vcentcolon=\int_{\S^1}\dd s$ and the (hyperbolic) elastic energy is defined as in \eqref{def:elastenergyhyp}.
We sometimes also work with the \emph{Euclidean} length of $\gamma$ given by $\Ll_{\R^2}(\gamma)\vcentcolon=\int_{\S^1}|\partial_x\gamma|\dd x$.
\begin{remark}\label{rem:min-e}
    Going back to \cite{langersinger1984}, we have $\E(\gamma)\geq 4\pi$ for all $\gamma\in C^{\infty}_{\mathrm{imm}}(\S^1,\H^2)$ with equality if and only if $\gamma$ is the profile curve of a so-called \emph{Clifford torus}.
\end{remark}
For the sake of simplicity, we write $\nabla_s$ resp.\ $\nabla_x$ for $\nabla_{\partial_s \gamma}$ resp.\ $\nabla_{\partial_x\gamma}$ in the following.
We also write $\nabla_s$ instead of $\partial_s$ when differentiating functions. 
For a vector field $V$ along $\gamma$, $V^\perp$ denotes the projection onto the subspace orthogonal to $\partial_s \gamma$. In particular,
\begin{equation}
 \label{eq:defNabla}
  \quad \nabla_{s}^\bot \vcentcolon= \nabla_{s } - \ \langle \nabla_{s} \ \cdot \ , \partial_s \gamma \rangle_g \partial_s \gamma.
\end{equation}

\subsection{Evolution equations and the Willmore flow}
If $\gamma \colon [0,T) \times \mathbb{S}^1 \to \H^2$ is a smooth map for some $T>0$, we endow $[0,T)$ with the coordinate $t$ and set
\begin{equation}
    \nabla_{t} \vcentcolon= \nabla_{\partial_t \gamma},\quad \nabla_{t}^\bot  \vcentcolon= \nabla_{t} - \ \langle \nabla_{t}\ \cdot \ , \partial_s \gamma \rangle_g \partial_s \gamma .
\end{equation}

The following can be computed as in \cite[Lemma 2.4]{dallacquaspener2017}.

\begin{lemma}\label{lemma:evo}
Let $T>0$ and $\gamma \colon  [0,T) \times \mathbb{S}^1  \to \H^2$ be a smooth family of immersions satisfying $\partial_t \gamma = V$ where $V=V^{\perp}$. Then we have
\begin{align}
 \partial_t (\dd s) = \partial_t (|\partial_x \gamma|_g \dd x)&= - \langle V,\vec{\kappa} \rangle_g \dd s, \label{eq:evolutionoflineelement2}\\
 \nabla_{t}\partial_s \gamma - \nabla_{s}\partial_t \gamma &= \langle V,\vec{\kappa} \rangle_g \partial_s \gamma
 \label{eq:commutator1}.
\end{align}
For any   normal vector field $\Phi \colon  [0,T) \times I \to \R^2$ we have
\begin{align}
\nabla_{s} \Phi &= \nabla_{s}^\bot \Phi - \langle \Phi , \vec{\kappa}\rangle _g \partial_s \gamma,\label{eq:NormalDerOfNormal}\\
\nabla_{t}^\bot \nabla_{s}^\bot \Phi - \nabla_{s}^\bot \nabla_{t}^\bot \Phi
& =
\langle V,\vec{\kappa}\rangle_g  \nabla_{s}^\bot \Phi\, .\label{eq:Commutator_Normal_Normal_S}
\end{align}
Moreover,
\begin{align}
\nabla_{t}^\bot \vec{\kappa} &= (\nabla_{s}^\bot)^2 V   + \langle V,\vec{\kappa}\rangle_g \vec{\kappa} - V.
\label{eq:EvoKappa}
\end{align}
\end{lemma}
With \Cref{lemma:evo}, geometric invariance implies the following.
\begin{lemma}
    For a smooth family of immersions $\gamma\colon[0,T)\times\S^1\to\H^2$, we have
    \begin{align}
    \label{eq:length-dt}
        \partial_t\Ll(\gamma(t)) &= - \int_{\S^1}\langle \curv,\partial_t\gamma\rangle_g\dd s, \\
        \label{eq:1-var-elen}
         \partial_t\E(\gamma(t)) &= \int_{\S^1} \bigr\langle \nabla\E(\gamma), \partial_t\gamma \bigl\rangle_g\dd s,
    \end{align}
    where 
    \begin{equation}\label{eq:nabla-L2}
    \nabla\E (\gamma)\vcentcolon= 2(\nabla_s^{\bot})^2\curv + |\curv|_g^2\curv - 2 \curv.
    \end{equation}
\end{lemma}

The relation \eqref{eq:BG} between Willmore energy and elastic energy is also available on the level of the $L^2$-gradients. Indeed, for a surface of revolution $f_\gamma$ as in \eqref{eq:def-surf-rev}, in the notation of \eqref{eq:def-R}, the computations in
\cite[Theorem 4.1]{dallacquaspener2018} yield
\begin{equation}\label{eq:will-grad-for-fgamma}
    \nabla\W(f_{\gamma}) =  \mathcal R \big(\frac{1}{4(\gamma^{(2)})^4}\nabla\E(\gamma)\big).
\end{equation}
This immediately yields a one-to-one correspondence between Willmore critical tori of revolution and \emph{elastica,} i.e., critical points of the elastic energy. 

\begin{definition}\label{rem:constrelastica}
Critical points of the $\lambda$-penalized elastic energy $\E(\gamma)+ \lambda \Ll(\gamma)$, $\lambda \in \R$, satisfy
\begin{equation}\label{eq:constrelast}
\nabla \E (\gamma) - \lambda \vec{\kappa}=0.
\end{equation}
We call such curves
\emph{$\lambda$-constrained elastica} and \emph{(free) elastica} if $\lambda=0$. 
\end{definition}
Solutions to \eqref{eq:constrelast} have been classified in \cite{langersinger1984}, see also \cite{muellerspener2020}.
The next proposition characterizes surfaces of revolution evolving by Willmore flow.
\begin{proposition}\label{prop:Schl1}
    Let $\gamma\colon [0,T)\times \S^1\to\H^2$ be a smooth family of curves. Then $f_\gamma\colon [0,T)\times \T\to\R^3$ solves the Willmore flow equation $\partial_tf_{\gamma} = - \nabla\W(f_{\gamma})$
     on $[0,T)\times\T$ if and only if
    \begin{equation}\label{eq:wf-ev-eq}
        \partial_t \gamma = -\frac{1}{2(\gamma^{(2)})^4} \bigl( (\nabla_s^{\bot})^2\vec{\kappa} + \frac{1}{2}|\vec{\kappa}|_g^2\vec{\kappa} -\vec{\kappa} \bigr) = -\frac{1}{4(\gamma^{(2)})^4}\nabla\E(\gamma).
    \end{equation}
\end{proposition}
\begin{proof}
    By a direct computation, $\partial_tf_{\gamma}=\mathcal R\partial_t\gamma$. 
    The claim then follows from \eqref{eq:will-grad-for-fgamma}.
\end{proof}

This observation motivates the following.
\begin{definition}\label{def:will-flow}
We call a smooth family of immersions $\gamma\colon[0,T)\times\S^1\to\H^2$ solving
\begin{equation}\label{eq:wf-eq}
    \begin{cases}
        \partial_t\gamma = -\frac{1}{4(\gamma^{(2)})^4} \nabla\E(\gamma)&\text{on $[0,T)\times\S^1$},\\
        \gamma(0)=\gamma_0&\text{on $\S^1$}
    \end{cases}
\end{equation}
a \emph{Willmore flow} starting at $\gamma_0$.
\end{definition}

By \eqref{eq:1-var-elen}, \eqref{eq:nabla-L2}, and \eqref{eq:BG} the PDE \eqref{eq:wf-eq} has the following gradient flow structure.
\begin{lemma}\label{rem:endecr}
    If $\gamma \colon[0,T)\times\S^1\to\H^2$ is a Willmore flow, then
    \begin{align}
        \frac{\dd}{\dd t} \E(\gamma) & = \int_{\S^1} \bigr\langle \nabla \E(\gamma), \partial_t \gamma \bigl\rangle_g\dd s  = - \int_{\S^1}  \frac{1}{4 (\gamma^{(2)})^4}| \nabla \E(\gamma)|^2_g\dd s \leq 0\quad\text{and}\\
        \frac{\dd}{\dd t} \W(f_{\gamma}) &= \int_{\T} \langle\nabla\W(f_{\gamma}),\partial_tf_{\gamma}\rangle\dd\mu = - \int_{\T} |\nabla\W(f_{\gamma})|^2\dd\mu \leq 0.
    \end{align}
\end{lemma}

The elastic flow in the hyperbolic plane considered in \cite{dallacquaspener2017,muellerspener2020} is   
\begin{equation}\label{eq:elasflowhyp}
    \begin{cases}
        \partial_t\gamma = - \nabla\E(\gamma)&\text{on $[0,T)\times\S^1$},\\
        \gamma(0)=\gamma_0&\text{on $\S^1$} 
    \end{cases}
\end{equation}
   and differs from  the evolution equation \eqref{eq:wf-eq}
   by the factor $\frac{1}{4(\gamma^{(2)})^4}$. Controlling this factor is a major challenge throughout this article.
   A key ingredient is the following estimate, see \cite[Lemma~2.3]{dallacquamullerschatzlespener2020}. 
   
   \begin{lemma}\label{rem:hyplengthestimatesrange(optimal)}
   For an immersion $\gamma\colon\S^1\to\H^2$ and $a,b\in\S^1$, we have
    \begin{equation}\label{eq:est-gamma2-length}
        \gamma^{(2)}(b)\leq \gamma^{(2)}(a) e^{\frac12\Ll(\gamma)}.
    \end{equation}  
\end{lemma}
\begin{proof}
Note that $\Ll(\gamma|_{[a,b]})\leq\frac12\Ll(\gamma)$ or $\Ll(\gamma|_{\S^1\setminus (a,b)})\leq\frac12\Ll(\gamma)$. In the first case, we have
    \begin{equation}
        \tfrac{1}{2}\mathcal{L}(\gamma) \geq \int_a^b \tfrac{|\partial_x \gamma|}{\gamma^{(2)}} \; \mathrm{d}x \geq  \int_a^b \tfrac{|\partial_x \gamma^{(2)}|}{\gamma^{(2)}} \; \dd x = \int_a^b |\partial_x \log(\gamma^{(2)}) | \; \dd x  \geq \log(\gamma^{(2)}(b)) - \log(\gamma^{(2)}(a)). 
    \end{equation}
    Taking exponentials, the desired inequality follows. The other case is analogous.    
\end{proof}

\subsection{The conformally constrained Willmore gradient and flow}\label{sec:conform_constr}

\begin{definition}
    Let $(M_i,g_i)$, $i=1,2$ be Riemannian manifolds. A smooth immersion $f\colon M_1\to M_2$ is called \emph{conformal} if there exists $u\in C^\infty(M_1,\R)$ with
    \begin{equation}
        f^*g_2 = e^{2u} g_1.
    \end{equation}
    Moreover, $u$ is called the \emph{conformal factor} of $f$.
\end{definition}
Letting
\begin{align}
    \mathcal M&\vcentcolon=\{\vec\omega=(\omega_1,\omega_2)\in \R^2\mid 0\leq \omega_1\leq \frac12,\ |\vec\omega|\geq 1\text{ and } \omega_2>0\}, \\
    \T_{\vec\omega}&\vcentcolon=\R^2 /(\Z \cdot (1,0)+\Z\cdot \vec\omega) \quad\text{ for }\vec\omega \in \mathcal M,
\end{align}
it is well-known that, for any torus 
$(\mathbb{T}^2,g)$, there exists a unique $\vec\omega\in\mathcal M$ such that
a conformal diffeomorphism $\phi\colon (\T_{\vec\omega},\delta_{ij})\to (\T,g)$ exists, 
see \cite[Theorem~2.7.1 and Theorem~4.4.1, additionally dividing out the orientations]{jost2006}. We write $\mathfrak c(g)\vcentcolon=\vec\omega$, the \emph{conformal class} of $g$. 

For an immersion $f\colon\T\to\R^3$, with a slight abuse of notation, write $\mathfrak c(f)\vcentcolon=\mathfrak c(f^*\langle \cdot,\cdot\rangle)$. Note that the conformal class is invariant with respect to reparametrizations. With $\vec\omega=\mathfrak c(f)$, we define the space of $C^\infty$-variations respecting the conformal class
\begin{align}
    \mathfrak C_f(\vec\omega) \vcentcolon= \{\partial_{t,0}F(t,\cdot)\mid \ & F\in C^{1,\infty}((-\varepsilon,\varepsilon)\times\T,\R^3)\text{ family of immersions with $F(0)=f$} \\
    & \text{and $\mathfrak c(F(t,\cdot))=\vec\omega$ for all $-\varepsilon<t<\varepsilon$} \}.
\end{align}
Here, as in \cite{MR3199732} for manifolds $M,N$, we write $F\in C^{1,\infty}((-\varepsilon,\varepsilon)\times M,N)$ if 
\begin{equation}
    \partial_{x}^s F,\ \partial_{x}^s\partial_t F\quad\text{ are continuous on $(-\varepsilon,\varepsilon)\times M$ for all $s\in\N_0$}
\end{equation}
where $\partial_{x}^s$ denotes any spatial derivative of order $s$.
Moreover, with a slight abuse of notation, we denote by $\overline{\mathfrak C_f(\vec\omega)}$ the \emph{closure of the linear span} of $\mathfrak C_f(\vec\omega)$ in $L^2(\dd \mu_f)$.

\begin{definition}
    For $\vec\omega\in\mathcal M$ and a smooth immersion $f\colon\T\to\R^3$ with $\mathfrak c(f)=\vec\omega$, we define its \emph{conformally constrained Willmore gradient} to be
    \begin{equation}
        \nabla^{\vec\omega}\W(f) \vcentcolon= \Pi^{\overline{\mathfrak C_f(\vec\omega)}} \nabla\W(f)
    \end{equation}
    where $\Pi^H$ denotes the $L^2(\dd\mu_{f})$-orthogonal projection onto a closed subspace $H\subseteq L^2(\dd\mu_{f})$. Furthermore, $f$ is said to be a \emph{conformally constrained Willmore torus} if $\nabla^{\vec\omega}\W(f)=0$.
\end{definition}

This definition is consistent with the literature \cite{kuwertschaetzle2013,MR3199732} in the following way.
\begin{proposition}
    Let $f\colon\T\to\R^3$ be a smooth immersion with $\mathfrak c(f)=\vec\omega$. Then $\nabla^{\vec\omega}\W(f)=0$ if and only if $f$ solves 
    \begin{equation}
        \Delta H + |A^0|^2 H = \langle q,A^0\rangle
    \end{equation}
    for some transverse traceless symmetric $2$-covariant tensor $q$ with respect to $g=f^*\langle\cdot,\cdot\rangle$.
\end{proposition}
\begin{proof}
    A standard argument yields that $\nabla^{\vec\omega}\W(f)=0$ if and only if 
    $\W'(f)V = 0$  for all $V\in\mathfrak C_f(\vec\omega)$.
    The claim then follows from \cite[Theorem~2.4]{MR3199732}.
\end{proof}

For $\vec\omega\in\mathcal M$ and an immersion $f_0\colon\T\to\R^3$ with $\mathfrak c(f_0)=\vec\omega$, a smooth family of immersions $f\colon[0,T)\times\T\to\R^3$ satisfying $\mathfrak c(f(t,\cdot))=\vec\omega$ for all $0\leq t<T$ and
\begin{equation}\label{eq:2d-conf-constr-wf-eq}
    \begin{cases}
        \partial_t f = -\nabla^{\vec\omega} \W(f)&\text{on $[0,T)\times\T$},\\
        f(0,\cdot)=f_0&\text{on $\T$}
    \end{cases}
\end{equation}
is called a $\emph{conformally constrained Willmore flow}$ starting at $f_0$. 
By the properties of the $L^2(\dd\mu)$ inner product
we have
\begin{equation}\label{eq:dtW-for-conf-constr-WF}
    \partial_t\W(f(t)) = \int_{\T} \langle\nabla\W(f(t)),-\Pi^{\overline{\mathfrak C_f(\vec\omega)}}\nabla\W(f(t))\rangle\dd\mu = - \int_{\T} |\partial_tf|^2\dd\mu,
\end{equation}
that is, the energy identity
\begin{equation}
    \W(f(s))-\W(f(t)) = \int_s^t \int_{\T} |\partial_tf(\tau,x)|^2\dd\mu_f(x) \dd\tau
\end{equation}
holds for all $0\leq s<t<T$. In the case of tori of revolution, the conformal Willmore gradient can be explicitly computed and we get the following analog to \eqref{eq:will-grad-for-fgamma}. 

\begin{proposition}[Conformal Willmore gradient for tori of revolution]\label{prop:conf-will-grad-axisym}
    Let $f_{\gamma}$ be a torus of revolution induced by an immersion $\gamma\colon\S^1\to\H^2$ via \eqref{eq:def-surf-rev} with $\vec\omega=\mathfrak c(f_\gamma)$. Then, using the notation \eqref{eq:def-R}, 
    \begin{equation}\label{eq:conf-constr-wf-eqI}
        \nabla^{\vec\omega}\W(f_{\gamma}) = \mathcal{R}\Big(\frac{1}{4(\gamma^{(2)})^4}(\nabla\E(\gamma)-\lambda(\gamma)\curv)\Big)
    \end{equation}
    where
    \begin{equation}\label{eq:lambda-conf-constr-will}
        \lambda(\gamma)\vcentcolon=\frac{\int_{\S^1} \frac{1}{(\gamma^{(2)})^4} \langle \nabla\E(\gamma), \curv\rangle_g \dd s}{\int_{\S^1} \frac{1}{(\gamma^{(2)})^4} |\curv|_g^2 \dd s}.
    \end{equation}
    In particular, $f_{\gamma}$ is a conformally constrained Willmore torus if and only if $\gamma$ is a ($\lambda$-constrained) elastica with $\lambda=\lambda(\gamma)$.
\end{proposition}

This correspondence allows us to discuss the sufficiency of \Cref{intro:thm:main-conf-constr-wf}\eqref{item:intro:thm_conformal_case_1}.
\begin{remark}\label{rem:lambdafigure8}
By \cite[Proposition 6.2]{muellerspener2020}, there exists a $\lambda$-constrained elastica $\gamma_0$ with vanishing total hyperbolic curvature, a so-called \emph{$\lambda$-figure eight,} see \Cref{fig:lambda_fig_8}. By the last part of \Cref{prop:conf-will-grad-axisym}, the surface of revolution $f_{\gamma_0}$ gives a stationary solution to the conformally constrained Willmore flow \eqref{eq:2d-conf-constr-wf-eq}. In particular, we have global existence and convergence, hence condition \eqref{item:intro:thm_conformal_case_1} in \Cref{intro:thm:main-conf-constr-wf} is not sufficient for the existence of a singularity.
\end{remark}

\begin{figure}[htb]
    \centering
    \begin{subfigure}{4cm}
        \includegraphics[width=\linewidth]{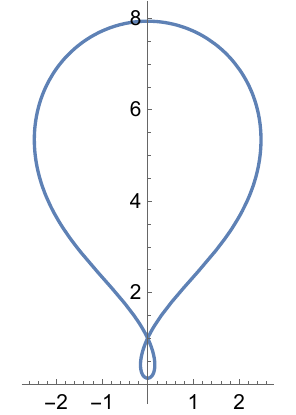}
        \caption*{$\lambda=0.5$, $\E_\lambda\approx 17.4365$}
    \end{subfigure}
    \qquad
    \begin{subfigure}{4cm}
        \includegraphics[width=\linewidth]{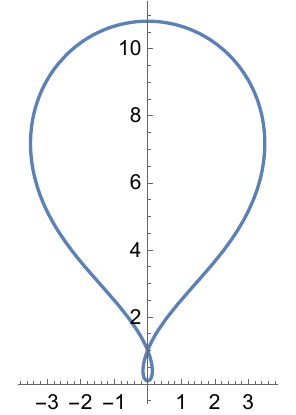}
        \caption*{$\lambda=0.3$, $\E_\lambda\approx 16.7833$}
    \end{subfigure}
    \qquad
    \begin{subfigure}{4cm}
        \includegraphics[width=\linewidth]{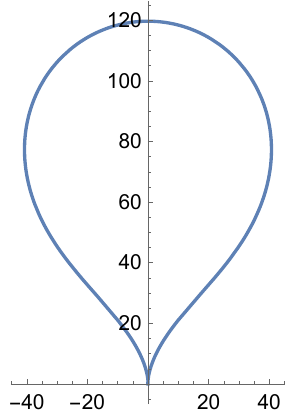}
        \caption*{$\lambda=0.005$, $\E_\lambda\approx 16.0112$}
    \end{subfigure}
    \caption{Plots of $\lambda$-figure eights for $\lambda\in\{0.5,0.3,0.005\}$ and associated elastic energy $\E_{\lambda}$ close to $16$, all having a self-intersection at $(0,1)$.}
    \label{fig:lambda_fig_8}
\end{figure}

The proof of \Cref{prop:conf-will-grad-axisym} relies on the key fact that $(\gamma^{(2)})^{-4}\mathcal R\curv_\gamma$ is orthogonal to $\overline{\mathfrak C_{f_\gamma}(\vec\omega)}$ in $L^2(\dd\mu_{f_\gamma})$. This is proven in the following two lemmas.

\begin{lemma}
    Consider a family of conformal immersions $F\in C^{1,\infty}( (-\varepsilon,\varepsilon)\times \T_{\vec\omega},\R^3)$, that is
    \begin{equation}
        g_{F(t,\cdot)} = e^{2U(t,\cdot)} (\delta_{ij})_{i,j=1,2} \quad\text{for all $t\in(-\varepsilon,\varepsilon)$}.
    \end{equation}
    Writing $V=\partial_{t,0}F(t,\cdot)$, $u=U(0,\cdot)$ and $f=F(0,\cdot)$, we have
    \begin{equation}\label{eq:conf-l2-char}
        0 = \int_{\T_{\vec\omega}} e^{-4u} \langle  \partial_{1,1}^2f-2\partial_1 u\partial_1f-\partial_{2,2}^2f+2\partial_2u\partial_2f,V\rangle \dd\mu_f.
    \end{equation}
\end{lemma}
\begin{proof}
    From the conformality, we especially have
    \begin{align}
        0 &=\partial_{t,0} \frac12 \big( |\partial_1 F|^2-|\partial_2 F|^2 \big) = \langle \partial_1 f,\partial_1V\rangle - \langle \partial_2 f,\partial_2V\rangle. \label{eq:conf-comp-1}
    \end{align}
    With $x=(x^1,x^2)$, using $\dd\mu_f = e^{2u}\dd x$ and integrating by parts, we have for $i=1,2$
    \begin{align}
        \int_{\T_{\vec\omega}} \langle \partial_i f,\partial_iV\rangle e^{-2u} \dd x &= - \int_{\T_{\vec\omega}} \langle \partial_{i,i}^2f - 2\partial_iu\partial_if,V\rangle e^{-2u}\dd x =  - \int_{\T_{\vec\omega}} e^{-4u} \langle \partial_{i,i}^2f - 2\partial_iu\partial_if,V\rangle \dd\mu_f
    \end{align}
    and thus, the claim follows after multiplying \eqref{eq:conf-comp-1} by $e^{-2u}$ and integrating with respect to $\dd x$.
\end{proof}

In the previous lemma we work with a $C^{1,\infty}$ family of conformal immersions. This is a particular choice for the variations considerd in $\mathfrak C_f(\vec\omega)$. The following observation by Schätzle \cite{MR3199732} however guarantees that we can always choose such a family. \begin{proposition}\label{prop:conf-var-ex-of-rep}
    For every family of immersions $F\in C^{1,\infty}((-\varepsilon,\varepsilon)\times\T,\R^3)$ with $F(0)=f$ and $\mathfrak c(F(t,\cdot))=\vec\omega$ for all $-\varepsilon<t<\varepsilon$, there exists a family $\phi\in C^{1,\infty}((-\varepsilon,\varepsilon)\times\T_{\vec\omega},\T)$ of diffeomorphisms $\T_{\vec\omega}\to\T$ such that $F(t,\phi(t,\cdot))\colon (\T_{\vec{\omega}},\delta_{ij}) \to (\R^3,\langle\cdot,\cdot\rangle) $ are conformal for all $-\varepsilon<t<\varepsilon$.
\end{proposition}
\begin{proof}
    This is a direct consequence of \cite[Proposition~2.1]{MR3199732}.
\end{proof}

\begin{lemma}\label{lem:Rcurv-orth-Cfomega}
    If $f=f_{\gamma}\colon\T\to\R^3$ is a smooth torus of revolution as in \eqref{eq:def-surf-rev}, then 
    \begin{equation} \label{eq:conf-class-of-torus-of-rev}
        \vec\omega=\mathfrak c(f)=\max\{\Ll(\gamma)/2\pi,2\pi/\Ll(\gamma)\}\cdot (0,1)
    \end{equation}
    and if $V\in \mathfrak C_{f}(\vec\omega)$, then
    $\int_{\T} \langle \mathcal (\gamma^{(2)})^{-4} \mathcal R\curv_{\gamma},V\rangle \dd\mu_f=0.$
\end{lemma}
\begin{proof}
    Without loss of generality, we may assume that $f_\gamma\colon \R^2 / (\Ll(\gamma)\Z \cdot (1,0) + 2\pi\Z\cdot (0,1)) \to\R^3$ is conformal by arranging that the profile curve $\gamma\colon \R/(\Ll(\gamma) \Z)\to\H^2$ satisfies $|\partial_x\gamma|_g \equiv 1$, i.e., is parametrized by arc-length in $\H^2$. 
    In fact, in this situation, \eqref{eq:metric_torus_revolution} implies that $f_\gamma$ is conformal and $u(x_1,x_2)=\log \gamma^{(2)}(x_1)$
    is the conformal factor. 
    In particular, \eqref{eq:conf-class-of-torus-of-rev} follows as in \cite[Proposition~4.2]{dallacquamullerschatzlespener2020}. By \eqref{eq:def-surf-rev} and \eqref{eq:def-R},
    \begin{align}
         \partial_{1,1}^2f-2&\partial_1 u\partial_1f-\partial_{2,2}^2f+2\partial_2u\partial_2f = \mathcal R \partial_x^2\gamma - 2\frac{\partial_x\gamma^{(2)}}{\gamma^{(2)}} \mathcal R\partial_x\gamma + \mathcal R\begin{pmatrix} 0\\\gamma^{(2)} \end{pmatrix} \\
         &= \mathcal R\Big( \partial_x^2\gamma + \frac{1}{\gamma^{(2)}} \begin{pmatrix} -2\partial_x\gamma^{(1)}\partial_x\gamma^{(2)} \\ (\gamma^{(2)})^2-2(\partial_x\gamma^{(2)})^2 \end{pmatrix} \Big)\\
         &= \mathcal R\Big( \partial_x^2\gamma + \frac{1}{\gamma^{(2)}} \begin{pmatrix} -2\partial_x\gamma^{(1)}\partial_x\gamma^{(2)} \\ (\partial_x\gamma^{(1)})^2-(\partial_x\gamma^{(2)})^2 \end{pmatrix} \Big) = \mathcal R \curv_{\gamma}\label{eq:magic}
    \end{align}
    using $|\partial_x\gamma|^2 = (\gamma^{(2)})^2$ in the second to last step and $\partial_x=\partial_s$ as well as \eqref{eq:curvature} in the last step. 
    
    Consider any variation $F\colon(-\varepsilon,\varepsilon)\times \R^2 / (\Ll(\gamma)\Z \cdot (1,0) + 2\pi\Z\cdot (0,1)) \to\R^3$ with $F(0)=f$ as in the definition of $\mathfrak C_f(\vec\omega)$ and write $V\vcentcolon=\partial_{t,0}F(t,\cdot)$. Using \Cref{prop:conf-var-ex-of-rep} and composing with a conformal diffeomorphism $\T_{\vec\omega}\to \R^2 / (\Ll(\gamma)\Z \cdot (1,0) + 2\pi\Z\cdot (0,1))$, there exists a $C^{1,\infty}$ family of diffeomorphisms $\phi\colon (-\varepsilon,\varepsilon)\times \R^2 / (\Ll(\gamma)\Z \cdot (1,0) + 2\pi\Z\cdot (0,1)) \to \R^2 / (\Ll(\gamma)\Z \cdot (1,0) + 2\pi\Z\cdot (0,1))$ such that $F(t,\phi(t,\cdot))$ is conformal for all $t$.

    Writing $\phi_0\vcentcolon=\phi(0,\cdot)$ and using $g_{f\circ \phi_0}=\phi_0^*g_f$, as $F(0)=f$ is conformal, $\phi_0$ is necessarily a conformal diffeomorphism of $\R^2 / (\Ll(\gamma)\Z \cdot (1,0) + 2\pi\Z\cdot (0,1))$. In particular, also the family $t\mapsto G(t,\cdot)\vcentcolon=F(t,\phi(t,\cdot)\circ\phi_0^{-1})\colon \R^2 / (\Ll(\gamma)\Z \cdot (1,0) + 2\pi\Z\cdot (0,1))\to\R^3$ is a family of conformal immersions with $G(0)=f$. Moreover,
    \begin{equation}
        \partial_{t,0}G(t,\cdot)= \partial_{t,0} F(t,\cdot)  + Df\cdot \big((\partial_{t,0}\phi(t,\cdot))\circ \phi_0^{-1}\big) = V + Df\cdot \big((\partial_{t,0}\phi(t,\cdot))\circ \phi_0^{-1}\big).
    \end{equation}
    By \eqref{eq:magic} and \eqref{eq:conf-l2-char}, we have
    \begin{align}
        0 &= \int_{\T}  \langle (\gamma^{(2)})^{-4}\mathcal R\curv_\gamma, \partial_{t,0} G(t,\cdot) \rangle \dd\mu_{f} \\
        &= \int_{\T}  \langle \mathcal (\gamma^{(2)})^{-4}\mathcal R\curv_\gamma, V + Df\cdot \big((\partial_{t,0}\phi(t,\cdot))\circ \phi_0^{-1}\big) \rangle \dd\mu_{f} = \int_{\T}  \langle  (\gamma^{(2)})^{-4}\mathcal R\curv_\gamma, V \rangle \dd\mu_f,
    \end{align}
    using that $\mathcal{R}\curv_\gamma(x_1,x_2)$ is orthogonal to $\mathrm{Im} \ Df_\gamma(x_1,x_2)$ in $\R^3$.
\end{proof}

\begin{proof}[Proof of \Cref{prop:conf-will-grad-axisym}]
    Consider the $L^2(\dd\mu_{f_\gamma})$-closed subspace
    $
        H\vcentcolon=\{\mathcal R\varphi\mid \varphi\in L^2(\S^1,\R^2)\}.
    $
    Using \eqref{eq:dmu-axisymm}, by direct computation, we have
    \begin{align}\label{eq:banana}
        \int_{\T}\langle \mathcal{R}\varphi, (\gamma^{(2)})^{-4} \mathcal{R}\curv_\gamma\rangle \dd\mu_{f_\gamma} = \int_{\S^1}\langle\varphi, \curv_\gamma\rangle_g\dd s_\gamma.
    \end{align}
    We will now show that
    \begin{align}\label{eq:2.12_claim_1_eq}
        H\cap \overline{\mathfrak C_{f_\gamma}(\vec\omega)} = \{\mathcal R\varphi\mid \varphi\in L^2(\S^1,\R^2),\ \int_{\T}\langle\mathcal R\varphi,(\gamma^{(2)})^{-4}\mathcal R\curv_{\gamma}\rangle \dd \mu_{f_\gamma} = 0\}.
    \end{align}
    Notice that the inclusion ``$\subseteq$'' follows from \Cref{lem:Rcurv-orth-Cfomega}. We next argue ``$\supseteq$''.
    Let $\varphi\in C^{\infty}(\S^1,\R^2)$ with $\int_{\T}\langle\mathcal R\varphi,(\gamma^{(2)})^{-4}\mathcal R\curv_{\gamma}\rangle \dd \mu_{f_\gamma} = 0$. By \eqref{eq:banana}, we have
    \begin{equation}\label{eq:ode-variation-1}
        \int_{\S^1}\langle\curv_\gamma,\varphi\rangle_g\dd s_\gamma=0.
    \end{equation}
    We first show that there exists a smooth family of immersions $\Gamma\colon(-\varepsilon,\varepsilon)\times\S^1\to\H^2$ with $\Gamma(0,\cdot)=\gamma$, $\partial_{t,0}\Gamma(t,\cdot)=\varphi$ and $\Ll(\Gamma(t,\cdot))=\Ll(\gamma)$ for all $t\in(-\varepsilon,\varepsilon)$.\\
    Fix $\tilde\varepsilon>0$ such that, for all $|t|,|\xi|<\tilde\varepsilon$, $\Gamma(t,\xi,\cdot)\vcentcolon=\gamma+t\varphi+\xi\curv_\gamma\in C^{\infty}_{\mathrm{imm}}(\S^1,\H^2)$ and such that  
    \begin{equation}
        G\colon (-\tilde\varepsilon,\tilde\varepsilon)\times (-\tilde\varepsilon,\tilde\varepsilon) \to \R,\qquad G(t,\xi)\vcentcolon=-\frac{\int_{\S^1}\langle \curv_{\Gamma(t,\xi,\cdot)},\varphi\rangle_g\dd s_{\Gamma(t,\xi,\cdot)}}{\int_{\S^1}\langle \curv_{\Gamma(t,\xi,\cdot)},\curv_\gamma\rangle_g\dd s_{\Gamma(t,\xi,\cdot)}},
    \end{equation}
   is well-defined.
    Then there exists $0<\varepsilon<\tilde\varepsilon$ and a smooth solution $\xi\colon(-\varepsilon,\varepsilon)\to(-\tilde\varepsilon,\tilde\varepsilon)$ of the ODE
    \begin{equation}\label{eq:ode-variation-2}
        \begin{cases}
            \xi'(t)=G(t,\xi(t))&\text{for $-\varepsilon<t<\varepsilon$}\\
            \xi(0)=0.
        \end{cases}
    \end{equation}
    Setting $\Gamma(t,\cdot)\vcentcolon=\Gamma(t,\xi(t),\cdot)$, $|t|<\varepsilon$, we have that $\Gamma$ is a smooth family of immersions with $\Gamma(0,\cdot)=\gamma$ and, using \eqref{eq:ode-variation-2} and that $G(0,0)=0$ by \eqref{eq:ode-variation-1},
    \begin{equation}
        \partial_{t,0}\Gamma(t,\cdot)=\varphi + \xi'(0) \curv_\gamma = \varphi.
    \end{equation}
    Moreover, using $\partial_t\Gamma(t,\cdot)=\varphi + \xi'(t)\curv_\gamma$,
    \begin{equation}
        \partial_t\Ll(\Gamma(t,\cdot)) = -\int_{\S^1} \langle \curv_{\Gamma(t,\cdot)},\varphi\rangle_g\dd s_{\Gamma(t,\cdot)} - \xi'(t) \int_{\S^1} \langle \curv_{\Gamma(t,\cdot)},\curv_\gamma\rangle_g\dd s_{\Gamma(t,\cdot)} = 0
    \end{equation}
    by \eqref{eq:ode-variation-2}. Using \eqref{eq:conf-class-of-torus-of-rev}, $F(t,\cdot)\vcentcolon=\mathcal R{\Gamma(t,\cdot)}$ is a variation as in the definition of $\mathfrak C_{f}(\vec\omega)$ and $\partial_{t,0} F(t,\cdot)=\mathcal R\varphi$. 
    Taking closures in $L^2(\dd\mu_{f_\gamma})$ and using \eqref{eq:banana}, we conclude ``$\supseteq$'' in \eqref{eq:2.12_claim_1_eq}.

    By the choice of $\lambda(\gamma)$ in \eqref{eq:lambda-conf-constr-will} and using \eqref{eq:dmu-axisymm},
    \begin{align}
        \int_{\T} \langle \mathcal{R}\Big(\frac{1}{4(\gamma^{(2)})^4}(\nabla\E(\gamma)-\lambda(\gamma)\curv_\gamma)\Big),(\gamma^{(2)})^{-4}\mathcal R\curv_\gamma\rangle\dd\mu = 0.
    \end{align}
    Hence, \eqref{eq:2.12_claim_1_eq} yields $\mathcal{R}\Big(\frac{1}{4(\gamma^{(2)})^4}(\nabla\E(\gamma)-\lambda(\gamma)\curv_\gamma)\Big)\in \overline{\mathfrak C_{f_\gamma}(\vec\omega)}$. Moreover, using \eqref{eq:will-grad-for-fgamma},
    \begin{equation}
        \nabla\W(f_\gamma)-\mathcal{R}\Big(\frac{1}{4(\gamma^{(2)})^4}(\nabla\E(\gamma)-\lambda(\gamma)\curv_\gamma)\Big) = \frac14\lambda(\gamma) \cdot (\gamma^{(2)})^{-4} \mathcal{R}\curv_\gamma .
    \end{equation}
    The right hand side in the equation above
    is orthogonal to $\overline{\mathfrak C_{f_\gamma}(\vec\omega)}$ in $L^2(\dd\mu)$ by \Cref{lem:Rcurv-orth-Cfomega}. Hence $$\nabla^{\vec\omega}\W(f_\gamma)=\Pi^{\overline{\mathfrak C_{f_\gamma}(\vec\omega)}}\nabla\W(f_\gamma)=\mathcal{R}\Big(\frac{1}{4(\gamma^{(2)})^4}(\nabla\E(\gamma)-\lambda(\gamma)\curv_\gamma)\Big).$$
    
    Finally, the equality above yields that if $f_{\gamma}$ is a constrained Willmore torus, then $\gamma$ is a $\lambda$-constrained elastica. The reverse follows using that for such $\gamma$, $\nabla \E(\gamma)- \lambda \vec{\kappa}=0$ implying that $\lambda(\gamma)$, defined in \eqref{eq:lambda-conf-constr-will}, is equal to $\lambda$ so that the right hand side of \eqref{eq:conf-constr-wf-eqI} is trivially zero.
\end{proof}

In analogy to \Cref{prop:Schl1}, \Cref{prop:conf-will-grad-axisym} yields the following.
\begin{proposition}\label{prop:Schl2}
   If $\gamma\colon[0,T)\times \S^1\to \H^2$ is a smooth solution of
    \begin{equation}\label{eq:ccwf-ev-eq}
        \partial_t \gamma = -\frac{1}{4(\gamma^{(2)})^4}(\nabla\E(\gamma)-\lambda(\gamma) \vec{\kappa}),
    \end{equation}
    with $\lambda$ as in \eqref{eq:lambda-conf-constr-will}, then the associated family of surfaces of revolution $f_{\gamma}\colon[0,T)\times\T\to \R^3$ solves \eqref{eq:2d-conf-constr-wf-eq} and moreover $\mathfrak c(f_{\gamma(t,\cdot)})=\vec\omega$ for all $0\leq t<T$.
\end{proposition}
\begin{proof}
This is a direct consequence of \Cref{prop:conf-will-grad-axisym} and \Cref{lem:Rcurv-orth-Cfomega} together with the fact that, by the choice of $\lambda$, \eqref{eq:length-dt} yields $\partial_t\Ll(\gamma(t))=0$.
\end{proof}
The previous result motivates the following.
\begin{definition}\label{def:conf-wf}
    A smooth family of immersions $\gamma\colon[0,T)\times\S^1\to\H^2$ solving
\begin{equation}\label{eq:conf-constr-wf-eq}
    \begin{cases}
        \partial_t\gamma = -\frac{1}{4(\gamma^{(2)})^4}\big(\nabla\E(\gamma)-\lambda(\gamma)\curv\big)&\text{on $[0,T)\times\S^1$}\\
        \gamma(0)=\gamma_0&\text{on $\S^1$}
    \end{cases}
\end{equation}
with $\lambda$ as in \eqref{eq:lambda-conf-constr-will} is called \emph{conformally constrained Willmore flow} starting at $\gamma_0$. 
\end{definition}

If $\gamma_0$ is sufficiently regular, a smooth solution to \eqref{eq:conf-constr-wf-eq} exists in some maximal time interval.

\begin{proposition}[Well-posedness]\label{prop:ste}
    Let $\gamma_0\colon\S^1\to\H^2$ be a smooth immersion. Then there exist a maximal time $T>0$ and a unique smooth family of immersions $\gamma\colon[0,T)\times\S^1\to\H^2$ solving the conformally constrained Willmore flow equation \eqref{eq:conf-constr-wf-eq} (respectively the Willmore flow equation \eqref{eq:wf-eq}).
\end{proposition}
A detailed proof of well-posedness is beyond the scope of this article. We briefly outline a possible proof-strategy. For both equations \eqref{eq:wf-eq} and \eqref{eq:conf-constr-wf-eq}, the same transformation used in \cite[Section~3.1]{dallacquaspener2017} for the elastic flow is applicable. The additional non-constant factor $\frac{1}{4(\gamma^{(2)})^4}$ does not introduce any additional challenges in proving \emph{local} well-posedness as it is uniformly controlled on compact time-intervals. Then, in case of \eqref{eq:wf-eq}, \cite{mantegazzamartinazzi2012} applies to the analog of \cite[Equation~(36)]{dallacquaspener2017}. For \eqref{eq:conf-constr-wf-eq}, the non-local term $\lambda(\gamma)$ can be controlled as in \cite[Section~2]{ruppspener2020}.

\begin{example}\label{ex:circle}
    Suppose that $\gamma_0\colon\S^1\to\H^2$ parametrizes a round circle, i.e., $\gamma_0$ is given by some $\gamma_r$ as in \Cref{app:cmc-tori}. 
    As it turns out, the three flows introduced in \eqref{eq:elasflowhyp}, \Cref{def:will-flow,def:conf-wf} differ in their qualitative behavior when starting at $\gamma_0$, respectively. In \Cref{lem:wf-and-round-circles}, we show that the Willmore flow starting at $\gamma_0$ immediately loses the shape of a round circle unless $\gamma_0$ is the profile curve of a Clifford torus in which case the Willmore flow is stationary. However, \cite{dallacquaspener2018} yields that the elastic flow starting at $\gamma_0$ is a family of round circles converging to the profile curve of a Clifford torus. Finally, round circles are the profile curves of conformally constrained Willmore tori, so the conformally constrained Willmore flow starting at a circle is always stationary, see \cite[Lemma~3.1]{dallacquaspener2018} and \Cref{prop:conf-will-grad-axisym}.
\end{example}

The following scaling behavior of Willmore flows is used repeatedly throughout this article.
\begin{lemma}\label{lem:par-scaling}
    Let $\gamma\colon[0,T)\times\S^1\to\H^2$ be a Willmore flow (resp.\ a conformally constrained Willmore flow) and let $\rho>0$, $p\in \R\times\{0\}$. Then also $\tilde\gamma\colon[0,T/\rho^4)\times\S^1\to\H^2$,
    \begin{equation}
        \tilde\gamma(t,x)\vcentcolon=\frac{\gamma(t\rho^4,x)-p}{\rho}
    \end{equation}
    is a Willmore flow (resp.\ a conformally constrained Willmore flow).
\end{lemma}
\begin{proof}
    Using \eqref{eq:curvature}, $\curv_{\tilde\gamma}(t,x)=\frac{1}{\rho}\curv_\gamma(t\rho^4,x)$ and thus $\nabla\E(\tilde\gamma(t,\cdot))=\frac{1}{\rho}\nabla\E(\gamma(t\rho^4,\cdot))$ since $z\mapsto \frac{z-p}{\rho}$ is an isometry of $\H^2$. Moreover, for the same reason, $\lambda(\tilde{\gamma}(t,\cdot))=\lambda(\gamma(t\rho^4,\cdot))$ by \eqref{eq:lambda-conf-constr-will}. So the claim follows from \eqref{eq:wf-eq} resp. \eqref{eq:conf-constr-wf-eq}, using $\partial_t\tilde\gamma(t,x)=\rho^3(\partial_t\gamma)(t\rho^4,x)$.
\end{proof}

\section{Weighted curvature estimates}\label{sec:curvature_evolutions}
\label{sec:2}

The goal of this section is to bound the curvature and its derivatives along the (conformally constrained) Willmore flow in a suitable weighted norm.
To this end, we adapt the notation of \cite{dziukkuwertschaetzle2002,lin2012,dallacquapozzi2014}. Compared to previous studies, the notation is more involved due to factors of $(\gamma^{(2)})^{-1}$ and their derivatives that naturally appear. These factors contribute only with \emph{components} of the curvature and not the entire curvature vector. In the following for $a,b,c \in \N_0$, $b\geq 1$, we use $P_b^{a,c}$
for any linear combination (with universal, constant coefficients) of terms of type
$$\frac{ \partial_s \gamma}{\gamma^{(2)}} * \cdots * \frac{\partial_s \gamma}{\gamma^{(2)}} * \frac{(\nabla_{s}^\bot)^{i_1} \vec{\kappa}}{\gamma^{(2)}}  * \cdots * \frac{(\nabla_{s}^\bot)^{i_b}  \vec{\kappa}} {\gamma^{(2)}}
$$
where $(v_1,\cdots, v_z) \mapsto v_1 * \cdots * v_z$, $z \in \N$, denotes a multilinear map from $(\R^2)^z$ with values in $\R^2$ or $\R$ with $z \geq b$, and $i_1 + \ldots + i_b = a$, $\max \ i_j \ \leq c$. If the multilinear map is vector-valued, the vector field is assumed to be \emph{normal} i.e., orthogonal to $\partial_s \gamma$. In the notation we do not take into account the number of factors of $\partial_s \gamma$. Notice that 
\begin{equation}\label{eq:CSPP}
 \Big|\frac{ \partial_s \gamma}{\gamma^{(2)}} * \cdots * \frac{\partial_s \gamma}{\gamma^{(2)}} * \frac{(\nabla_{s}^\bot)^{i_1} \vec{\kappa}}{\gamma^{(2)}}  * \cdots * \frac{(\nabla_{s}^\bot)^{i_b}  \vec{\kappa}} {\gamma^{(2)}}\Big| \leq C \prod_{j=1}^b |(\nabla_{s}^\bot)^{i_j}  \vec{\kappa} |_g, 
 \end{equation}
 for some constant $C$ depending only on the coefficients of the multilinear map.  Note further that $a$ gives the total number of derivatives, $b$ denotes the number of factors of the curvature (not necessarily of the whole vector) and $c$ gives a bound on the highest number of derivatives falling on one factor. 

We gather terms $P^{a,c}_b$ according to their ``order'', motivated by Gagliardo--Nirenberg type interpolation inequalities, see \Cref{cor:inter} below.  
More precisely, we denote by
$ \mathbb{P}^{A,C}_B$
finite sums of $P^{a,c}_b$ such that
\begin{equation}
    \label{eq:algebra}
    a+\frac12 b \leq A +\frac12 B \mbox{ and }c \leq C \, .
\end{equation}

From the defintion, a term of type $\mathbb{P}_{B_1}^{A_1,C_1} * \mathbb{P}_{B_2}^{A_2,C_2}$ also is of type $\mathbb{P}_{B_1+B_2}^{A_1+A_2,\max\{C_1,C_2\}}$. The following lemma illustrates that $\mathbb{P}_B^{A,C}$ behaves well with the weights $(\gamma^{(2)})^\beta$. Its proof is computational and thus moved to \Cref{app:techproofs}.

\begin{lemma}\label{lem:allesOK}
    For $m \in \N_0$ and $\beta \in \Z$, we have
    $$ (\nabla_s^{\bot})^m ( (\gamma^{(2)})^{\beta}\, \mathbb{P}^{A,C}_B ) = (\gamma^{(2)})^{\beta}\, \mathbb{P}^{A+m,C+m}_B . $$
\end{lemma}

\subsection{Interpolation estimates of Gagliardo--Nirenberg type}

\begin{proposition}[Multiplicative estimates]\label{prop:inter-multi}
    Let $\gamma\colon\mathbb{S}^1 \to \H^2 $ be a smooth immersion such that $\mathcal{L}(\gamma)\geq L>0$. Then for any $k \in\N$, $a,c \in\N_0$, $b \in \N$, $b\geq 2$ such that $ c\leq k-1$ and $a+\frac{b}{2} -1 < 2k,$ 
    there exists a constant $C$  depending on $a$, $b$, $k$, $L$ and the coefficients of the underlying multilinear map, satisfying
    \begin{equation}\label{eq:mult-inter-ineq-with-lower-length-bound}
        \int_{\S^1} |P_b^{a,c}| \dd s\leq C  \||\curv|_g\|_{L^2(\dd s)}^{b-\gamma} \big( \||(\nabla_s^{\bot})^k\curv|_g\|_{L^2(\dd s)}^{\gamma} + \||\curv|_g\|_{L^2(\dd s)}^{\gamma} \big)
    \end{equation}
    where $\gamma=(a+b/2-1)/k$. In particular, if $\E(\gamma)\leq M$, then there exists a constant $C$ depending on $a$, $b$, $k$, $M$ and the coefficients of the underlying multilinear map,  with
    \begin{equation}\label{eq:mult-inter-ineq}
        \int_{\S^1} |P_b^{a,c}| \dd s\leq C \big( 1+ \||(\nabla_s^{\bot})^k\curv|_g\|_{L^2(\dd s)}^{\gamma} \big).
    \end{equation}
\end{proposition}
\begin{proof}
    Using \eqref{eq:CSPP}, the estimate in \eqref{eq:mult-inter-ineq-with-lower-length-bound} is a combination of \cite[Proposition~4.3 and Corollary~4.2]{dallacquaspener2017}. By Fenchel's Theorem in $\H^2$, cf.\ \cite[Theorem 2.3]{dallacquaspener2017},
    \begin{equation}\label{eq:len-low-bound}
        \frac{1}{{\mathcal{L}(\gamma)}} \leq \frac{1}{(2 \pi)^2} \E (\gamma) \leq \frac{1}{(2 \pi)^2} M \, .
    \end{equation}
    Hence,  $\mathcal{L}(\gamma)\geq \ell=\ell(M)\vcentcolon= (2\pi)^2/M$. Equation \eqref{eq:mult-inter-ineq} thus follows from \eqref{eq:mult-inter-ineq-with-lower-length-bound}.
\end{proof}

\begin{proposition}
    \label{cor:inter}
    Let $\gamma\colon\mathbb{S}^1 \to \H^2 $ be a smooth immersion such that $\E(\gamma) \leq M$. Then for any $k \in\N$, $A,C \in\N_0$, $B \in \N$, $B\geq 2$ such that $ C\leq k$ and $A+\frac{B}{2} -1 < 2k,$ for any $\varepsilon \in(0,1)$,
    there exists a constant $\tilde{c}_\varepsilon$ depending on $k$, $A$, $B$, $C$, $M$, $\varepsilon$, and the coefficients of the underlying multilinear maps such that
    $$\int_{\mathbb{S}^1} |\mathbb{P}^{A,C}_B | \dd  s \leq \varepsilon  \int_{\mathbb{S}^1} | (\nabla_{s}^\bot)^k\vec{\kappa}  |^2_{g} \dd s  + \tilde{c}_\varepsilon .$$ 
\end{proposition}
\begin{proof}
    Using \eqref{eq:CSPP} and the same arguments as in \cite[Proposition~4.3]{dallacquaspener2017}, taking care of the case $c=k$ as in \cite[Lemma~3.4]{dallacqualinpozzi2017}, we obtain for any $\varepsilon \in(0,1)$, $a+\frac{b}{2}\leq A+\frac{B}{2}$, $c\leq C$,
$$\int_{\mathbb{S}^1} |P^{a,c}_b | \dd  s \leq \varepsilon  \int_{\mathbb{S}^1} | (\nabla_{s}^\bot)^k\vec{\kappa}  |^2_{g} \dd s  + \tilde{c} \varepsilon^{-\frac{\mu}{2-\mu}} \E(\gamma)^{\frac{b-\mu}{2-\mu}}+\tilde{c} \E(\gamma)^{\frac{b}{2}} \,,$$
with $\mu=(a+b/2-1)/k$, where $\tilde{c}$ depends on $a$, $b$, $k$, $L$, and the coefficients of the underlying multilinear map.
    Using \eqref{eq:len-low-bound}, the claim follows from the definition of $\mathbb{P}^{A,C}_B$, cf.\ \eqref{eq:algebra}.
\end{proof}

\subsection{Energy estimates}

Since maximum principle techniques are not available for Willmore-type flows, we rely on energy methods instead. We present a unified approach to treat both the Willmore flow and a class of constrained Willmore flows simultaneously. Consider a smooth map $\lambda\colon [0,T)\to\R$ and a smooth family of immersions $\gamma\colon[0,T)\times\S^1\to\H^2$ satisfying
\begin{equation}\label{eq:evol-eq-with-general-lambda}
    \partial_t\gamma = -\frac{1}{4(\gamma^{(2)})^4} (\nabla\E(\gamma)-\lambda \curv)\quad\text{on $[0,T)\times\S^1$}.
\end{equation}
Note that $\gamma$ is a Willmore flow for the choice $\lambda\equiv 0$ and a conformally constrained Willmore flow if $\lambda(t)$ is given by $\lambda(\gamma(t))$ as in \eqref{eq:lambda-conf-constr-will}. 

\begin{lemma}\label{lem:int-est-general-lambda}
    Let $\gamma\colon [0,T) \times \S^1 \to \H^2$ satisfy \eqref{eq:evol-eq-with-general-lambda} and suppose that $\gamma^{(2)}(t,\cdot) \leq r$ for all $t\in[0,T)$ and $\E(\gamma(0)) \leq M$. Fix any $\mu>0$. Then there exists $C=C(m,M,\mu)$ such that for all $0\leq t<T$,
    \begin{align}
        \frac{\dd}{\dd t} &\int_{\S^1} (\gamma^{(2)})^4 | (\nabla_s^{\perp})^m \vec{\kappa}|^2_g \dd s + \frac{\mu}{r^4}\int_{\S^1} (\gamma^{(2)})^4 | (\nabla_s^{\perp})^m \vec{\kappa}|^2_g \dd s  + \frac12\int_{\S^1}  | (\nabla_s^{\perp})^{m+2} \vec{\kappa}|^2_g \dd s \\
        &\leq C(m,M,\mu) + |\lambda| \Big( \int_{\S^1}|(\nabla_s^{\bot})^{m+1}\curv|_g^2\dd s + \int_{\S^1} |\mathbb{P}_4^{2m,m}|\dd s \Big).\label{eq:int-est-general-lambda}
    \end{align}
\end{lemma}
The proofs of \Cref{lem:int-est-general-lambda} and \Cref{prop:il-conf-constr-wf} below
are rather technical and postponed to \Cref{app:techproofs} for the sake of readability.

In the special case of the Willmore flow (i.e., $\lambda \equiv 0$) we obtain a control for curvature derivatives provided that $\gamma^{(2)}$ lies in a certain range.

\begin{proposition}\label{prop:il}
Let $\gamma\colon [0,T) \times \S^1 \to \H^2$ be a Willmore flow such that $\gamma^{(2)}(t,\cdot) \leq r$ for all $t$ and  $\E(\gamma(0)) \leq M$. Fix any $\mu>0$. Then there exists $C=C(m,M,\mu)$ such that for all $0\leq t<T$,
\begin{equation}\label{eq:il-1}
    \int_{\S^1} (\gamma^{(2)}/r)^4 | (\nabla_s^{\perp})^m \vec{\kappa}|^2_g \dd s \leq C(m,M,\mu) + \int_{\S^1} (\gamma^{(2)}/r)^4 | (\nabla_s^{\perp})^m \vec{\kappa}|^2_g \dd s \Big|_{t=0} \ e^{- \frac{\mu}{r^4}t}.
\end{equation}
\end{proposition}
\begin{proof}
    The claim follows from \Cref{lem:int-est-general-lambda} with a Gronwall argument, using that $\lambda\equiv 0$ in case of a Willmore flow.
\end{proof}

For the conformally constrained Willmore flow, i.e.,  $\lambda$ is given by \eqref{eq:lambda-conf-constr-will}, the estimates in \Cref{lem:int-est-general-lambda} are not yet sufficient for an analogous Gronwall estimate. But we do obtain the following.

\begin{lemma}\label{prop:il-conf-constr-wf}
Let $\gamma\colon [0,T) \times \S^1 \to \H^2$ be a conformally constrained Willmore flow, let $L_0=\Ll(\gamma(0))$ and $M,r>0$ such that $\gamma^{(2)}(t,\cdot) \leq r$ for all $t$ and $\E(\gamma(0)) \leq M$. Fix any $\mu>0$. Then there exists $C=C(m,L_0,M,\mu)$ such that for all $0\leq t<T$,
\begin{align}
        \frac{\dd}{\dd t} \int_{\S^1} (\gamma^{(2)})^4 | (\nabla_s^{\perp})^m \vec{\kappa}|^2_g \dd s& + \frac{\mu}{r^4}\int_{\S^1} (\gamma^{(2)})^4 | (\nabla_s^{\perp})^m \vec{\kappa}|^2_g \dd s  + \frac14\int_{\S^1}  | (\nabla_s^{\perp})^{m+2} \vec{\kappa}|^2_g \dd s \\
        &\leq C(m,L_0,M,\mu) + 2 (\lambda(\gamma))^2 \int_{\S^1}|(\nabla_s^{\bot})^{m}\curv|_g^2\dd s . \label{eq:conf-constr-wf-lambda-estimate}
    \end{align}
\end{lemma}

Notice that the remaining $\lambda$-term in \eqref{eq:conf-constr-wf-lambda-estimate} cannot be treated with \eqref{eq:subcritical-term-2} and the interpolation methods used above: the term is critical. So in the following, we employ a-priori estimates for the nonlocality to reduce its order in comparison to \eqref{eq:subcritical-term-2}. However, a scaling argument for $\Ll$ and $\E$ as used in \cite[proof of Theorem~3.3]{dziukkuwertschaetzle2002} cannot be employed as both $\E$ and $\Ll$ are scaling invariant in $\H^2$.

\begin{lemma}\label{lem:bounddistintegral}
    Let $\gamma\colon[0,T)\times\S^1\to\H^2$ be a Willmore or conformally constrained Willmore flow with $\E(\gamma_0)\leq M$.  
    Then, for all $t_1,t_2\in[0,T)$,
    \begin{equation}\label{eq:area-est}
        \left|\int_{\S^1}({\gamma^{(2)}}(t_1,x))^2\dd s-\int_{\S^1} ({\gamma^{(2)}}(t_2,x))^2\dd s \right|\leq \frac{1}{2} M \cdot\sqrt{|t_1-t_2|}.
    \end{equation}
    Moreover, for any $x\in \S^1$, we have
    \begin{equation}\label{eq:diam-est}
        (\gamma^{(2)}(t_2,x))^2\leq C(M)\Big( \int_{\S^1}(\gamma^{(2)})^2\dd s\Big|_{t=t_1} + |t_2-t_1|^{\frac12} \Big).
    \end{equation}
\end{lemma}
\begin{proof}
  The idea is to relate the integrals above to the area of the corresponding surface of revolution. Indeed, for each $t$
  \begin{equation}
        \int_{\S^1}({\gamma^{(2)}}(t,x))^2\dd s =  \int_{\S^1}{\gamma^{(2)}}(t,x) | \partial_x \gamma(t, x)|\dd x = \frac{1}{2 \pi} \mu_{f_{\gamma(t)}}(\S^1\times \S^1)\, .\label{eq:area_in_gamma}
    \end{equation}
    By \Cref{prop:Schl1,prop:Schl2}, $f_{\gamma}(t)$ evolves by (conformally constrained) Willmore flow. Thus 
    \begin{equation}
        \partial_t \W(f_\gamma(t)) = -\int_{\T} |\partial_tf_\gamma|^2\dd\mu,
    \end{equation}
    using \eqref{eq:dtW-for-conf-constr-WF} in case of the conformally constrained Willmore flow.
    Therefore, \eqref{eq:area-est} is proven with the same computations as in \cite[Lemma D.6]{dallacquamullerschatzlespener2020}, using \Cref{rem:endecr} and \eqref{eq:BG}. 
    The second inequality follows from \eqref{eq:area-est} and \eqref{eq:area_in_gamma} upon noting that $\gamma^{(2)}(t,x)\leq \frac12\mathrm{diam}(f_{\gamma(t)})$, and thus 
    \begin{align}
        4(\gamma^{(2)}(t_2,x))^2\leq (\diam(f_{\gamma}(t_2)))^2 & \leq C \W(f_{\gamma}(t_2)) \mu_{f_{\gamma(t_2)}}(\S^1\times \S^1) 
    \end{align}
    by Simon's diameter estimate \cite[Lemma 1.1]{simon1993}.
\end{proof}

\subsection{Controlling the non-local Lagrange multiplier}
\label{sec:a-priori-lambda-control}

In the following, denote by 
\begin{equation}
    \vec n \vcentcolon= \begin{pmatrix}0&-1\\1&0\end{pmatrix}\partial_s\gamma
\end{equation}
the normal associated to an immersion $\gamma\colon\S^1\to\H^2$. Then $\scurv\vcentcolon=\langle \curv,\vec n\rangle_g$ is the \emph{scalar} (hyperbolic) curvature of $\gamma$. Furthermore, $T(\gamma) \vcentcolon= \frac{1}{2\pi} \int_\gamma \kappa_{\mathbb{R}^2} \dd s_{\mathbb{R}^2}\in\Z$ denotes the \emph{turning number}
or \emph{Euclidean total curvature} of $\gamma$ (see \cite[Section~5]{muellerspener2020} for a comprehensive overview).

\begin{lemma}\label{lem:lagr-control-by-testing}
    Consider a conformally constrained Willmore flow $\gamma\colon[0,T)\times\S^1\to\H^2$ with $\Ll(\gamma_0)=L_0$, $\E(\gamma_0)\leq M$, and suppose that there exist $\eta,r>0$ such that
    \begin{equation}\label{eq:delta-r-conf-constr-wf}
        \gamma^{(2)}\leq r\text{ on $[0,T)\times\S^1$}\quad\text{and}\quad \Big|\int_{\S^1}\scurv \dd s\Big| \geq \eta \text{ on $[0,T)$}.
    \end{equation} 
    Then there exists a constant $C=C(L_0,M,\eta)$ such that, for all $t\in[0,T)$,\begin{equation}\label{eq:lagr-control-by-testing}
        (\lambda(\gamma(t)))^2 \leq C(L_0,M,\eta) \Big(1+\int_{\S^1}|\curv|_g^4\dd s - r^4 \partial_t\E(\gamma(t))\Big).
    \end{equation}
\end{lemma}
\begin{proof}
    By induction, we have $\langle (\nabla_s^\bot)^m\curv,\vec n\rangle_g = \partial_s^m\scurv$ for all $m\in\N_0$. Thus, by the fundamental theorem of calculus and using \eqref{eq:evol-eq-with-general-lambda},
    \begin{align}
        -4\int_{\S^1} (\gamma^{(2)})^4\langle \partial_t\gamma,\vec n\rangle_g\dd s &= \int_{\S^1} (2\partial_s^2\scurv + \scurv^3 - 2 
        \scurv)
        \dd s -\lambda(\gamma)\int_{\S^1}\scurv\dd s\\
        &= \int_{\S^1} (\scurv^3-2\scurv )\dd s-\lambda(\gamma)\int_{\S^1}\scurv\dd s.\label{eq:lagr-control-by-testing-2}
    \end{align}
    Since $0=-\partial_t\Ll(\gamma(t))=\int_{\S^1}\langle\curv,\partial_t\gamma\rangle_g\dd s$ and using \eqref{eq:1-var-elen}, we have
    \begin{align}
        4\int_{\S^1} (\gamma^{(2)})^4|\partial_t\gamma|_g^2 \dd s &= - \int_{\S^1} \langle \nabla \E(\gamma) -\lambda(\gamma)\curv,\partial_t\gamma\rangle_g\dd s =- \int_{\S^1} \langle \nabla \E(\gamma),\partial_t\gamma\rangle_g\dd s = -\partial_t\E(\gamma(t)).
    \end{align}
    Moreover, as $\Ll(\gamma(t))=L_0$, we have $\int_{\S^1}4(\gamma^{(2)})^4\dd s\leq 4r^4L_0$. Altogether, repeatedly using Cauchy-Schwarz and $\E(\gamma(t))\leq\E(\gamma_0)\leq M$, \eqref{eq:lagr-control-by-testing-2} yields
    \begin{align}
        (\lambda(\gamma))^2 &\leq \frac{1}{\eta^2} \Big(2\int_{\S^1}4(\gamma^{(2)})^4|\partial_t\gamma|_g^2\dd s\int_{\S^1}4(\gamma^{(2)})^4\dd s +  2\E(\gamma)(2\int_{\S^1}\scurv^4\dd s + 8\Ll(\gamma)) \Big)\\
        &\leq C(L_0,M,\eta) \Big(1+\int_{\S^1}|\curv|_g^4\dd s - r^4 \partial_t\E(\gamma(t))\Big)
    \end{align}
    as claimed.
\end{proof}

\begin{remark}\label{rem:hyp-tot-curv-contro-by-energy}
    As a consequence of the Gauss--Bonnet Theorem and the Jordan Curve Theorem, if $\gamma\colon\S^1\to\H^2$ is an embedding, it bounds a disk $D\subset\H^2$, and
    \begin{align}
       \pm \int_{\S^1} \scurv \dd s  - \mathrm{vol}_{g_{\H^2}}(D) = \pm \int_{\S^1} \scurv \dd s +\int_D  K_{\H^2} \dd\mathrm{vol}_{g_{\H^2}}  = 2\pi,
    \end{align}
    where $K_{\H^2}=-1$ is the Gauss curvature and the sign $\pm$ is determined by the orientation of $\gamma$. This implies $|\int_{\S^1}\scurv\dd s|\geq 2\pi$. Notice that any immersion $\gamma$ with $\E(\gamma)<16$ is necessarily an embedding by the Li--Yau inequality \cite{liyau1982}. 
    Moreover, if $\gamma\colon [0,T)\times\S^1\to\H^2$ is a conformally constrainedWillmore flow with $\E(\gamma_0)<16$, then $\E(\gamma(t))\leq\E(\gamma_0)<16$ for all $0\leq t<T$ by \Cref{rem:endecr}. So, in this case, the bound $|\int_{\S^1} \scurv \dd s|\geq 2 \pi$ is preserved along the flow.
\end{remark}

The following elementary inequality enables us to control $\int_{\S^1}\kappa \dd s$ purely in terms of quantities that are determined by the initial datum of a conformally constrained Willmore flow.
\begin{lemma}\label{lem:hyp-tot-curv-control-by-length}
    Let $\gamma\colon\S^1\to\H^2$ be an immersion. Then the total hyperbolic curvature satsifies
    \begin{align}
        \int_{\S^1}\kappa \dd s = \int_{\S^1}\kappa_{\R^2}\dd s_{\R^2} + \int_{\S^1}\frac{\partial_s \gamma^{(1)}}{\gamma^{(2)}} \dd s.
    \end{align}
    In particular, with the turning number $T(\gamma)=m\in\Z$, we have
    \begin{equation}
        \Big|\int_{\S^1}\scurv\dd s\Big| \geq 2\pi |m| - \Ll(\gamma).
    \end{equation}
\end{lemma}

\begin{proof}
    Without loss of generality, suppose that $|\partial_x\gamma|\equiv 1$, that is, $\gamma$ is parametrized by Euclidean arc-length. By \cite[Equation~(2.2)]{eichmanngrunau2019},
    \begin{align}
        \int_{\S^1}\scurv\dd s&=\int_{\S^1} \big(\partial_x^2\gamma^{(2)}\partial_x\gamma^{(1)}\gamma^{(2)} - \partial_x^2\gamma^{(1)}\partial_x\gamma^{(2)}\gamma^{(2)} + \partial_x\gamma^{(1)}\big) \frac{1}{\gamma^{(2)}} \dd x\\
        &= \int_{\S^1} \langle \partial_x^2\gamma,\begin{pmatrix}0&-1\\1&0\end{pmatrix}\partial_x\gamma\rangle \dd x + \int_{\S^1} \partial_x\gamma^{(1)}\frac{1}{\gamma^{(2)}}\dd x\\
        &= \int_{\S^1}\scurv_{\R^2}\dd s_{\R^2} + \int_{\S^1} \frac{\partial_s\gamma^{(1)}}{\gamma^{(2)}} \dd s = 2\pi m + \int_{\S^1} \frac{\partial_s\gamma^{(1)}}{\gamma^{(2)}} \dd s.
    \end{align}
    Now we use $|\partial_s\gamma^{(1)}|\leq \gamma^{(2)}|\partial_s\gamma|_g=\gamma^{(2)}$ and the inverse triangle inequality.
\end{proof}

Finally, we are ready to improve the energy estimate in \Cref{prop:il-conf-constr-wf} with the aid of the control of the Lagrange multiplier.

\begin{proposition}\label{prop:GronwallReplacement}
    Let $\gamma\colon[0,T)\times\S^1\to\H^2$ be a conformally constrained Willmore flow and let $L_0=\Ll(\gamma(0))$, $\E(\gamma_0)\leq M$, and suppose that there exist $\eta,r>0$ with \eqref{eq:delta-r-conf-constr-wf}. Moreover, let $\bar r\in (0,r)$ such that $\max_{\S^1}\gamma^{(2)}(t)\geq \bar r$ for all $0\leq t<T$. Fix any $\mu>0$. Then, for all $0\leq t<T$,\begin{align}
        \int_{\S^1} (\gamma^{(2)})^4 &| (\nabla_s^{\perp})^m \vec{\kappa}|^2_g \dd s \leq C(m,r,\bar r,L_0,M,\eta,\mu) \\
        &+ \int_{\S^1} (\gamma^{(2)})^4 | (\nabla_s^{\perp})^m \vec{\kappa}|^2_g \dd s \Big|_{t=0} \ C(r,\bar r,L_0,M,\eta) e^{ - \frac{\mu}{r^4}t}.\label{eq:il-conf-constr-1}
    \end{align}
\end{proposition}
Notice that the second constant in \eqref{eq:il-conf-constr-1} does not depend on $\mu$. We will use this in \eqref{eq:choice-mu-conf-constr} below. 
\begin{proof}[Proof of \Cref{prop:GronwallReplacement}]
    To conclude the claim from \eqref{eq:conf-constr-wf-lambda-estimate} with a Gronwall argument, we show that the $\lambda$-term on the right hand side of \eqref{eq:conf-constr-wf-lambda-estimate} can be suitably absorbed. To this end, by \eqref{eq:lagr-control-by-testing},
   \begin{align}
        (\lambda(\gamma))^2 \int_{\S^1} |(\nabla_s^\bot)^m\curv|_g^2\dd s &\leq C(L_0,M,\eta) (1+\int_{\S^1}|\curv|_g^4\dd s)\int_{\S^1}|(\nabla_s^\bot)^m\curv|_g^2\dd s \\
        & \quad + C(L_0,M,\eta) (-r^4\partial_t\E(\gamma(t))) \int_{\S^1}|(\nabla_s^\bot)^m\curv|_g^2\dd s. \label{eq:il-conf-constr-2}
    \end{align}
    For the first term on the right hand side of \eqref{eq:il-conf-constr-2}, using \eqref{eq:mult-inter-ineq} with $(a,b,k)=(0,4,m+2)$ and $(a,b,k)=(2m,2,m+2)$,
    \begin{align}
        (1+\int_{\S^1}|\curv|_g^4\dd s)&\int_{\S^1}|(\nabla_s^\bot)^m\curv|_g^2\dd s \\
        &\leq C(m,M)(1+\||(\nabla_s^\bot)^{m+2}\curv|_g\|_{L^2(\dd s)}^{\frac{1}{m+2}})(1+\||(\nabla_s^\bot)^{m+2}\curv|_g\|_{L^2(\dd s)}^{\frac{2m}{m+2}})
    \end{align}
    which, using Young's inequality can be absorbed in \eqref{eq:conf-constr-wf-lambda-estimate} since $\frac{1}{m+2}+\frac{2m}{m+2} < 2$. For the second term on the right hand side of \eqref{eq:il-conf-constr-2}, write $\varphi(t)\vcentcolon=-\partial_t\E(\gamma(t))$. By \Cref{rem:hyplengthestimatesrange(optimal)},
    \begin{align}
        \varphi(t) \int_{\S^1}|(\nabla_s^\bot)^m\curv|_g^2\dd s &\leq \varphi(t) \Big(\frac{e^{L_0/2}}{\max_{\S^1}\gamma^{(2)}(t)}\Big)^4 \int_{\S^1}(\gamma^{(2)})^4|(\nabla_s^\bot)^m\curv|_g^2\dd s\\
        &\leq C(\bar r,L_0) \varphi(t) \int_{\S^1} (\gamma^{(2)})^4|(\nabla_s^\bot)^m\curv|_g^2\dd s.
    \end{align}
    Altogether, combined with \eqref{eq:il-conf-constr-2} and \eqref{eq:conf-constr-wf-lambda-estimate}, 
   \begin{align}
        \partial_t &\int_{\S^1}(\gamma^{(2)})^4|(\nabla_s^\bot)^m\curv|_g^2\dd s + \frac{\mu}{r^4} \int_{\S^1}(\gamma^{(2)})^4|(\nabla_s^\bot)^m\curv|_g^2\dd s \\
        &\leq C(m,L_0,M,\eta,\mu) + C(\bar r,L_0,M,\eta) r^4 \varphi(t) \int_{\S^1}(\gamma^{(2)})^4|(\nabla_s^\bot)^m\curv|_g^2\dd s.
    \end{align}
    Using the estimate $\int_0^T\varphi(t)\dd t=\E(\gamma_0)-\lim_{t\nearrow T}\E(\gamma(t))\leq M$, a Gronwall argument yields \eqref{eq:il-conf-constr-1}.
\end{proof}

\section{Proof of convergence}\label{sec:convergence}

\subsection{Life-span bound via the distance to the rotation axis}

With the a-priori control on all derivatives of the curvature in \eqref{eq:il-1}, we conclude that a uniform control from below on $\gamma^{(2)}$ suffices to extend the Willmore flow beyond a finite time $T$.
For the conformally constrained Willmore flow, we obtain the same using \eqref{eq:il-conf-constr-1} if additionally $|\int_{\S^1}\scurv\dd s|$ has a uniform positive lower bound.

\begin{lemma}\label{lem:ewigesLeben}
    Let $\gamma\colon[0,T)\times\S^1\to\H^2$, $T<\infty$, either be a Willmore flow or a conformally constrained Willmore flow. 
    Suppose that $\gamma^{(2)}(t,x)\geq \alpha>0$  for all $(t,x)\in [0,T) \times \S^1$ for some $\alpha>0$. If $\gamma$ is a conformally constrained Willmore flow, additionally assume that there exists $\eta>0$ with $|\int_{\S^1}\scurv\dd s|\geq \eta$ on $[0,T)$.    
    Then the maximal existence time $T_{\mathrm{max}}$ of the flow satisfies $T_{\mathrm{max}}>T$.
\end{lemma}
\begin{proof}
    First of all, since $T<\infty$, applying \eqref{eq:diam-est} with $t_1=0$ yields that there exists $R=R(T,\gamma_0)>0$ with $\gamma^{(2)}(t,x)\leq R$ for all $(t,x)\in[0,T)\times\S^1$. We now show that all derivatives of $\gamma$ can be uniformly bounded up to time $T$ in order to obtain an extension as in the claim by short-time existence. To this end, for $m\in\N$, define
    \begin{equation}\label{eq:Am}
        A_m \vcentcolon= \sum_{i=0}^m \int_{\S^1} |(\nabla_s^{\bot})^i\curv|_g^2\dd s \Big\vert_{t=0} = \sum_{i=0}^m \| |(\nabla_s^{\bot})^i\curv|_g\|_{L^2(\dd s)}^2 \Big\vert_{t=0}.
    \end{equation}
    Using \eqref{eq:il-1} in case of a Willmore flow and \eqref{eq:il-conf-constr-1} in case of a conformally constrained Willmore flow, in both cases with $\mu=1$, and using $\alpha\leq \gamma^{(2)}\leq R$, we find
    \begin{align}
        \| |(\nabla_s^{\bot})^m\curv|_g\|_{L^2(\dd s)}^2 
        &\leq C(m,A_m,M,R,\alpha,[L(\gamma_0),\eta])\label{eq:int-step2}
    \end{align}
    for all $0\leq t<T$. Here and in the following, the parentheses $[\cdot]$ indicate that the constant depends on $\eta$ and $\Ll(\gamma_0)$ only for the conformally constrained Willmore flow. Our goal is to conclude from \eqref{eq:int-step2} the boundedness of all space derivatives of $\gamma$.

    As in Step 3 of the proof of \cite[Theorem 1.1]{dallacquaspener2017}, using \cite[Proposition 4.3 and Lemma 2.7]{dallacquaspener2017} with \eqref{eq:len-low-bound}, we deduce 
    \begin{equation}
        \||\nabla_s^m\curv|_g\|_{L^2(\dd s)}^2
        \leq  C(m,A_m,M,R,\alpha,[L(\gamma_0),\eta])\quad\text{for all $0\leq t<T$}
    \end{equation}
    from \eqref{eq:int-step2} which, by interpolation (cf.\ \cite[Proposition 4.1]{dallacquaspener2017} and using \eqref{eq:len-low-bound}) improves to 
    \begin{equation}\label{eq:int-step3}
        \| |\nabla_s^m\curv|_g\|_{L^{\infty}(\S^1)}
        \leq C(m,A_{m+1},M,R,\alpha,[L(\gamma_0),\eta]) \quad\text{for all $0\leq t<T$}.
    \end{equation}
    The final step involves showing that $\||\nabla_x^m\curv|_g\|_{L^{\infty}(\S^1)}$ is uniformly bounded on $[0,T)$. To this end, we may proceed exactly as in Step 4 of the proof of \cite[Theorem 1.1]{dallacquaspener2017}. More precisely, using \cite[Lemma 2.8]{dallacquaspener2017}, $\nabla_x^m\curv$ is controlled by controlling $\nabla_s^m\curv$ and $|\partial_x\gamma|_g,\dots,\partial_x^{m-1}|\partial_x\gamma|_g$.  In \emph{finite time}, i.e., for $T<\infty$, this is obtained from computing the evolution equations for the partial derivatives of $|\partial_x\gamma|_g$ and using \eqref{eq:int-step3} to obtain $L^{\infty}$-bounds for the coefficients arising in these evolution equations. 

    We conclude that $\gamma$ can be smoothly extended to $[0,T]\times\S^1$ and thus the claim follows by short-time existence for \eqref{eq:wf-eq}, respectively \eqref{eq:conf-constr-wf-eq}, after restarting the flow in $\gamma(T)$, cf.\ \Cref{prop:ste}.
\end{proof}

As a consequence of \Cref{lem:ewigesLeben}, we obtain the following global existence result.

\begin{corollary}\label{cor:global-existence}
    Let $\gamma\colon[0,T)\times\S^1\to\H^2$ be a smooth family of immersions such that
    \begin{equation}
        \gamma^{(2)}(t,x)\geq \rho>0\quad\text{for all $0\leq t<T$ and $x\in\S^1$}.
    \end{equation}
    If $\gamma$ is a maximal Willmore flow, or if $\gamma$ is a maximal conformally constrained Willmore flow such that $\liminf_{t\nearrow T}|\int_{\S^1}\scurv\dd s| > 0$, then $T=\infty$.
\end{corollary}

From the $L^2$-control of the distance 
to the rotation axis in \Cref{lem:bounddistintegral}, we now deduce $L^{\infty}$-type estimates depending only on the initial elastic energy and a uniform bound on the hyperbolic length. 

\begin{lemma}\label{lem:BD}
    Let $\gamma\colon[0,T)\times\S^1\to\H^2$ be a maximal Willmore or conformally constrained Willmore flow such that, for some $\rho>0$ and $B>1$, $\rho\leq \gamma_0^{(2)}\leq \rho B$. Further suppose that $\E(\gamma_0)\leq M$ and, 
    \begin{enumerate}[(a)]
        \item in case of a Willmore flow, suppose that $\mathcal{L}(\gamma(t))\leq L$ for all $t\in[0,T)$;
        \item in case of a conformally constrained Willmore flow, denote $\Ll(\gamma_0)=L$ and suppose that $\liminf_{t\nearrow T} |\int_{\S^1}\scurv\dd s|>0$.
    \end{enumerate}
  Then, there
  exist $\delta_0= \delta_0(M)$ and 
    $D=D(L,M)>1$ such that $T > \delta_0\rho^4$ and
    \begin{equation}\label{eq:g2-est}
        \frac{\rho}{D} \leq \gamma^{(2)}(t,x)\leq \rho B D \quad\text{for all $t\in[0,\delta_0\rho^4]$ and $x\in \S^1$}.
    \end{equation}
\end{lemma}
\begin{proof}
By parabolic scaling as in \Cref{lem:par-scaling}, we may w.l.o.g.\ assume $\rho=1$. As argued in \eqref{eq:len-low-bound}, $\mathcal{L}(\gamma(t))\geq \ell\vcentcolon= (2\pi)^2/M$ for all $t\in[0,T)$. Further, 
\Cref{rem:hyplengthestimatesrange(optimal)} yields that for all $t\in [0,T)$ we have
\begin{equation}\label{eq:maxmin-len}
    \max_{\S^1}\gamma^{(2)}(t,\cdot)\leq e^{L/2}\min_{\S^1}\gamma^{(2)}(t,\cdot) .
   \end{equation}
For a constant $D>1$ that will be fixed later, define 
$$
    \delta_1 \vcentcolon=\sup\{t_1\in(0,T): D^{-1}\leq \gamma^{(2)}(t,x)\leq BD \text{ for all $t\in[0,t_1]$ and $x\in\S^1$}\}.
$$
Suppose that $\delta_1=T$. Then, \Cref{cor:global-existence} applies and yields $T=\infty$. In the remaining case,  $\delta_1<T$ and, by continuity, we then have $\min_{\S^1}\gamma^{(2)}(\delta_1,\cdot)=\frac{1}{D}$ or $\max_{\S^1}\gamma^{(2)}(\delta_1,\cdot) = BD$. We establish a non-trivial  uniform lower bound $\delta_0$ for $\delta_1$ only depending on $M$. 

    \textbf{Case 1:} Suppose that $\min_{\S^1}\gamma^{(2)}(\delta_1,\cdot)=\frac{1}{D}$. 

    Using \eqref{eq:maxmin-len} and \Cref{lem:bounddistintegral} and recalling $1=\rho\leq \gamma_0^{(2)}\leq B$, we find
    \begin{align}
        \frac{e^{L}}{D^2}  L \geq  \int_{\S^1}{\gamma^{(2)}(\delta_1,\cdot)}^2\dd s \geq \int_{\S^1}({\gamma^{(2)}(0,\cdot)})^2\dd s - \frac12 M \sqrt{\delta_1} \geq \ell - \frac12 M \sqrt{\delta_1} .
    \end{align}
    That is, by choosing $D$ sufficiently large,
    \begin{equation}\label{eq:linfty-bd-1}
       \delta_1 \geq \frac{4}{M^2} \Big(\ell-\frac{e^{L}}{D^2} L\Big)^2 >0.
    \end{equation}

    \textbf{Case 2:} $\max_{\S^1}\gamma^{(2)}(\delta_1,\cdot)= BD$.

    Using again \eqref{eq:maxmin-len} and \Cref{lem:bounddistintegral}, it follows that
    \begin{align}
        \ell \left(\frac{BD}{e^{L/2}}\right)^2 &\leq \int_{\S^1}(\gamma^{(2)}(\delta_1,x))^2\dd s \leq \int_{\S^1}\gamma^{(2)}(0,x)^2\dd s + \frac{M}{2} \sqrt{\delta_1}\leq B^2L+ \frac{M}{2} \sqrt{\delta_1}.
    \end{align}
    Therefore, in this case, we obtain by again choosing $D$ sufficiently large,
    \begin{equation}\label{eq:linfty-bd-2}
        \delta_1\geq \frac{4}{M^2} \Bigg(\ell \left(\frac{BD}{e^{L/2}}\right)^2-B^2L\Bigg)^2.
    \end{equation}

    Using that $B>1$, we can choose $D=D(L,M)>1$ sufficiently large such that the right-hand-sides in \eqref{eq:linfty-bd-1} and \eqref{eq:linfty-bd-2} both are greater than or equal to $\frac{\ell^2}{M^2}$.  Particularly, $\delta_1\geq \frac{\ell^2}{M^2}=\vcentcolon \delta_0=\delta_0(M)$.
\end{proof}

\subsection{Global bounds and sub-convergence via an iterative blow-up scheme}\label{sec:iterative_blow_up}

Throughout this section, we consider the following situation.
\begin{hypothesis}\label{hyp:wf-cwf}
Let $\gamma\colon[0,T)\times\S^1\to\H^2$ be a maximal Willmore or conformally constrained Willmore flow such that $\E(\gamma_0)\leq M$ and, 
\begin{enumerate}[(a)]
    \item in case of a Willmore flow, suppose that $\mathcal{L}(\gamma(t))\leq L$ for all $t\in[0,T)$;
    \item in case of a conformally constrained Willmore flow, denote $\Ll(\gamma_0)=\vcentcolon L$ and suppose that $\eta=\inf_{t\in[0,T)} |\int_{\S^1}\scurv\dd s| >0$.
\end{enumerate}
\end{hypothesis}

We wish to show that $T=\infty$ and that the flow converges. The standard methods
for one-dimensional geometric evolution equations based on interpolation inequalities as employed, e.g., in \cite{dziukkuwertschaetzle2002,lin2012,dallacquapozzi2014,schlierf2024}, are not directly applicable without some a-priori control of the factor $1/(\gamma^{(2)})^4$, cf.\ \Cref{cor:global-existence}. Therefore, we iteratively rescale the flow along a sequence of times $(t_j)_{j\in\N}$ in order to obtain uniform control of the distance to the rotation axis in the rescaled flows. In \Cref{lem:BD}, choose $B\vcentcolon= e^{L/2}$ and fix $\delta_0=\delta_0(M)>0$ as well as $D=D(L,M)>1$ as in the lemma. The sequence of times is constructed inductively. Let $t_0=0$ and, if $t_j$ is already defined and finite for some $j\in\N_0$, write 
$$\rho_j\vcentcolon=\min_{x\in\S^1}\gamma^{(2)}(t_j,x)$$ and define 
\begin{equation}
    t_{j+1} \vcentcolon= \sup\{\Bar{t}\in[t_j,T): \frac{\rho_j}{D}\leq \gamma^{(2)}(t,x) \leq \rho_j\cdot BD \text{ for all $t\in[0,\Bar{t}]$ and $x\in\S^1$} \}\in (t_j,T].
\end{equation}
Note that $t_{j+1}>t_j$ follows from \Cref{rem:hyplengthestimatesrange(optimal)}. The iteration stops if, for some $J\geq 1$, $t_J=T$. Otherwise, set $J\vcentcolon=\infty$. We will now show that
\begin{align}
t_j\nearrow T \text{ for } j\to J. 
\label{cl:tj-to-T}
\end{align}

\begin{proof}[{Proof of \eqref{cl:tj-to-T}}]
    W.l.o.g.\ $J=\infty$ and, for the sake of contradiction, suppose $T'\vcentcolon=
    \lim_{j\to\infty}t_j<T$. 
    If $\liminf_{j\to\infty}\rho_j=0$, then choose $(x_j)_{j\in\N}\subseteq\S^1$ with $\rho_j=\gamma^{(2)}(t_j,x_j)$. There exists a sub-sequence with $x_{j_k}\to x'\in\S^1$. Since $T'<T$, we find $\gamma^{(2)}(T',x')=\lim_{k\to\infty}\gamma^{(2)}(t_{j_k},x_{j_k})=0$, a contradiction! Therefore, there exists $\varepsilon>0$ with $\rho_j\geq \varepsilon$ for all $j\in\N$. 
     Using the above definition of $t_j$, \Cref{lem:BD} yields
    \begin{equation}
        t_{j+1}-t_j \geq \varepsilon^4 \delta_0.
    \end{equation}
    Therefore, $T'=\sum_{j=0}^{\infty}(t_{j+1}-t_j) = \infty$ which especially yields $T=T'$, a contradiction!
\end{proof}
For $0\leq j<J$, define
\begin{equation}
    \gamma_j(t,x)\vcentcolon=\frac{1}{\rho_j} \gamma(t_j+\rho_j^4 t,x).
\end{equation}
By \Cref{lem:par-scaling}, for $0\leq j<J$, each $\gamma_j$ again solves \eqref{eq:wf-eq} (respectively \eqref{eq:conf-constr-wf-eq}) with initial datum $\gamma_{j,0}=\frac{1}{\rho_j}\gamma(t_j)$. Using \Cref{rem:hyplengthestimatesrange(optimal)},
\begin{equation}\label{eq:apple}
    1\leq \gamma_{j,0}^{(2)}(x)\leq e^{L/2} = B\quad\text{for all $x\in\S^1$}.
\end{equation} 
By \Cref{lem:BD} and the scaling invariance of the hyperbolic length, the elastic energy and the total geodesic curvature in $\H^2$, we have the following. For $0\leq j<J$, if $T_j \vcentcolon= (t_{j+1}-t_j)/{\rho_j^4}<\infty$, each $\gamma_j$ is defined on $[0,T_j]\times\S^1$ with $T_j\geq \delta_0$ (cf.\ \Cref{lem:BD}) and, by definition,
\begin{equation}\label{eq:broccoli}
    \gamma_{j-1}(T_{j-1},\cdot) = \frac1{\rho_{j-1}} \gamma(t_{j},\cdot) = \frac{\rho_{j}}{\rho_{j-1}} \gamma_{j}(0,\cdot).
\end{equation}
Otherwise, $\gamma_j$ is already globally defined. In both cases, 
\begin{equation}\label{eq:unif-rot-dist-bounds}
    \frac{1}{D} \leq \gamma_j^{(2)}(t,x) \leq BD\quad\text{for all $0\leq t<T_j$ and $x\in\S^1$}.
\end{equation}

For $0\leq j<J$ and $m\in\N$, define $\vec{N}_j\vcentcolon=(\gamma_j^{(2)})^2(\nabla_s^{\bot})^m\curv_j$, where $\curv_j$ denotes the curvature vector of $\gamma_j$. Since scalings by $\frac{1}{\rho_j}$ and $\frac{1}{\rho_{j-1}}$ are isometries of $\H^2$, for $1\leq j <J$ and $x\in \S^1$ we have
\begin{align}
    |\Vec{N}_j(0,x)|_g &= (\gamma_j^{(2)}(0,x))^2|({\nabla_s^{\bot}})^m\curv_j(0,x)|_g = \frac{1}{\rho_j^2}(\gamma^{(2)}(t_j,x))^2 |(\nabla_s^{\bot})^m\curv(t_j,x)|_g\\
    &=\frac{\rho_{j-1}^2}{\rho_j^2}(\gamma^{(2)}_{j-1}(T_{j-1},x))^2 |(\nabla_s^{\bot})^m\curv_{j-1}(T_{j-1},x)|_g  \label{mustep2}
    = \Big(\frac{\rho_{j-1}}{\rho_j}\Big)^2  |\Vec{N}_{j-1}(T_{j-1},x)|_g.
\end{align}
In the following, we write $\dd s_j$ for the arc-length element of $\gamma_j$.

\textbf{Case 1: $\gamma$ is a Willmore flow.}
Choose $\tilde\mu>0$ sufficiently large such that
\begin{equation}\label{eq:choice-mu}
    D^4 \cdot e^{-\tilde\mu\delta_0} < \frac12.
\end{equation}
Applying \Cref{prop:il} with $r=BD$, $\mu=(BD)^4\tilde\mu$ and using $T_j\geq \delta_0$,
\begin{align}
     \int_{\S^1}|\Vec{N}_j|_g^2 \dd s_j\Big\vert_{t=T_j} &\leq \int_{\S^1}|\Vec{N}_j|_g^2 \dd s_j\Big\vert_{t=0}\cdot e^{-\tilde\mu T_j} + C(m,L,M)\\
     &\leq \int_{\S^1}|\Vec{N}_j|_g^2 \dd s_j\Big\vert_{t=0}\cdot e^{-\tilde\mu \delta_0} + C(m,L,M). \label{mustep1}
\end{align}
Finally using \eqref{mustep1}, \eqref{mustep2} and $\dd s_{j-1}\vert_{t=T_{j-1}}=\dd s_j\vert_{t=0}$ by \eqref{eq:broccoli}, and estimating $\rho_j=\min_{\S^1}\gamma^{(2)}(t_j)\geq \rho_{j-1}/D$, we find
\begin{align}
    \int_{\S^1}|\Vec{N}_j|_g^2\dd s_j\Big\vert_{t=T_j} &\leq \Big(\frac{\rho_{j-1}}{\rho_j}\Big)^{4} \int_{\S^1}|\vec{N}_{j-1}|_g^2\dd s_{j-1}\Big\vert_{t=T_{j-1}} \cdot e^{-\tilde\mu\delta_0}+ C(m,L,M)\\
    &\leq D^{4} e^{-\tilde\mu\delta_0} \cdot \int_{\S^1}|\vec{N}_{j-1}|_g^2\dd s_{j-1}\Big\vert_{t=T_{j-1}}+ C(m,L,M) \label{eq:iteration}
\end{align}
for all $1\leq j<J$. By \eqref{eq:choice-mu}, we have $q\vcentcolon=D^{4}e^{-\tilde\mu\delta_0}<\frac12$, so that \eqref{eq:iteration} yields 
\begin{align}
    \int_{\S^1}|\Vec{N}_j|_g^2\dd s_j\Big\vert_{t=T_j} &\leq q^j\cdot \int_{\S^1}|\vec{N}_{0}|_g^2\dd s_0\Big\vert_{t=0} + \sum_{i=0}^{j-1} q^i\cdot C(m,L,M) \leq  C(m,A_m,L,M)
\end{align}
for all $1\leq j<J$ 
where $A_m$ is defined as in \eqref{eq:Am} 
 using \eqref{eq:apple} for $j=0$. 
With \eqref{eq:il-1} and \eqref{eq:apple},
\begin{equation}
    \sup_{t\in[0,T_j)}\int_{\S^1}|\Vec{N}_j|_g^2\dd s_j \leq C(m,A_m,L,M)\quad\text{for all $0\leq j<J$}.
\end{equation}

\textbf{Case 2: $\gamma$ is a conformally constrained Willmore flow.}
Choosing $\tilde\mu>0$ sufficiently large,
\begin{equation}\label{eq:choice-mu-conf-constr}
    C_1\cdot D^4 \cdot e^{-\tilde\mu\delta_0} < \frac12
\end{equation}
where $C_1\vcentcolon=C(r,\bar r,L_0,M,\eta)$ denotes the second constant in \eqref{eq:il-conf-constr-1} for $r=BD, \bar r=1/D$. Applying \eqref{eq:il-conf-constr-1} with $r=BD$, $\mu=(BD)^4\tilde\mu$, $\bar r=1/D$, and using $T_j\geq \delta_0$,
\begin{align}
     \int_{\S^1}|\Vec{N}_j|_g^2\dd s_j\Big\vert_{t=T_j} &\leq \int_{\S^1}|\Vec{N}_j|_g^2\dd s_j\Big\vert_{t=0}\cdot C_1 e^{-\tilde\mu T_j} + C(m,L,M,\eta)\\
     &\leq \int_{\S^1}|\Vec{N}_j|_g^2\dd s_j\Big\vert_{t=0}\cdot C_1 e^{-\tilde\mu \delta_0} + C(m,L,M,\eta). \label{mustep1-conf-constr}
\end{align}
Finally using that $\rho_j=\min_{\S^1}\gamma^{(2)}(t_j)\geq \rho_{j-1}/D$, and again that scalings are isometries, we obtain from \eqref{mustep1-conf-constr} and \eqref{mustep2} the iteration
\begin{align}
    \int_{\S^1}|\Vec{N}_j|_g^2\dd s_j\Big\vert_{t=T_j} &\leq \Big(\frac{\rho_{j-1}}{\rho_j}\Big)^{4} \int_{\S^1}|\vec{N}_{j-1}|_g^2\dd s_{j-1}\Big\vert_{t=T_{j-1}} \cdot C_1e^{-\tilde\mu\delta_0}+ C(m,L,M,\eta)\\
    &\leq C_1D^{4} e^{-\tilde\mu\delta_0} \cdot \int_{\S^1}|\vec{N}_{j-1}|_g^2\dd s_{j-1}\Big\vert_{t=T_{j-1}}+ C(m,L,M,\eta) \label{eq:iteration-conf-constr}
\end{align}
for all $1\leq j<J$. By \eqref{eq:choice-mu-conf-constr}, we have $q\vcentcolon=C_1D^{4}e^{-\tilde\mu\delta_0}<\frac12$ and thus, we may argue as in \emph{Case~1} with \eqref{eq:conf-constr-wf-lambda-estimate} replaced by \eqref{eq:il-conf-constr-1} to show
\begin{equation}
    \sup_{t\in[0,T_j)}\int_{\S^1}|\Vec{N}_j|_g^2\dd s_j \leq C(m,A_m,L,M,\eta) \quad\text{for all $0\leq j<J$}.
\end{equation}

By the uniform bounds on $\gamma_j^{(2)}$ in \eqref{eq:unif-rot-dist-bounds}, using the same the notation as in \eqref{eq:Am}, we have
\begin{equation}\label{eq:bu-step2}
    \sup_{t\in[0,T_j)}\int_{\S^1}|({\nabla_s^{\bot}})^m\curv_j|_g^2\dd s_j\leq C(m,A_m,L,M,[\eta])\quad\text{for all $0\leq j<J$}.
\end{equation}
Since scalings are isometries of $\H^2$ as well as \Cref{cl:tj-to-T}, we just proved the following proposition.
\begin{proposition}

Assume \Cref{hyp:wf-cwf}. With $A_m$ defined as in \eqref{eq:Am}, for all $m\in\N$ we have
    \begin{equation}\label{eq:glob-en-est}
        \sup_{t\in[0,T)} \int_{\S^1} |(\nabla_s^{\bot})^m\curv|_g^2\dd s\leq C(m,A_m,L,M,[\eta]).
    \end{equation}        
\end{proposition}

The following sub-convergence result is a direct consequence.

\begin{corollary}[Sub-convergence]\label{cor:better-sub-convergence}
    Assume \Cref{hyp:wf-cwf}. For any sequence of times $\tau_j\nearrow T$ and for any $(\rho_j)_{j\in\N}\subseteq(0,\infty)$ and $(p_j)_{j\in\N}\subseteq \R\times\{0\}$ satisfying
    \begin{equation}\label{eq:linfty-control}
        \frac{\gamma(\tau_j,\S^1)-p_j}{\rho_j} \subseteq K\quad\text{for all $j\in\N$}
    \end{equation}
    for some compact set $K\subseteq\H^2$, there exists a critical point $\hat{\gamma}\colon\S^1\to\H^2$ 
    \begin{enumerate}[(a)]
        \item of $\E$ in case that $\gamma$ is a Willmore flow,
        \item of $\E + \lambda\Ll$ for some $\lambda\in\R$ in case that $\gamma$ is a conformally constrained Willmore flow,
    \end{enumerate}
    such that
    \begin{equation}
        \frac{\gamma(\tau_j)-p_j}{\rho_j} \circ\phi_j \to \hat{\gamma}\quad\text{in the $C^{\infty}(\S^1)$ topology}
    \end{equation}
    after passing to a subsequence. Here, $\phi_j\colon\S^1\to\S^1$ are diffeomorphisms such that $\gamma(\tau_j)\circ\phi_j$ is parametrized by constant hyperbolic speed.
\end{corollary}
The possibility to take \emph{any} sequence $\tau_j\nearrow T$ in the previous corollary is exploited in  \Cref{thm:main-conf-constr-wf} below.
\begin{remark}[The Willmore flow with $\W(f_{\gamma_0})< 8\pi$]
    Consider the Willmore flow and assume that $\W(f_{\gamma_0}) < 8\pi$. By \cite[Theorem~1.1]{muellerspener2020}, there exists $L>0$ such that $\Ll(\gamma(t,\cdot))\leq L$ for all $0\leq t<T$. Hence, \Cref{hyp:wf-cwf} is satisfied. In the setting of \Cref{cor:better-sub-convergence}, let $\tau_j\nearrow T$ and make the specific choice $\rho_j\vcentcolon=\min_{\S^1}\gamma^{(2)}(\tau_j)$, $p_j=(p_j^{(1)},0)\in\R\times\{0\}$ such that $(p_j^{(1)},\rho_j)\in\gamma(\tau_j,\S^1)$. Then 
    \begin{equation}
        \frac{\gamma(\tau_j,x)-p_j}{\rho_j} \in \overline{B}_{\frac12L}((0,1))\subseteq\H^2 \quad\text{for all $x\in\S^1$}
    \end{equation}
    so that \eqref{eq:linfty-control} is satisfied. Moreover, \cite[Corollary~6.6]{muellerspener2020} yields that $\hat{\gamma}$ in \Cref{cor:better-sub-convergence} necessarily parametrizes the round circle with center $(0,\sqrt{2}/(\sqrt{2}-1))$ and radius $1/(\sqrt{2}-1)$. In particular, up to reparametrization, the limit does not depend on the sequence $(\tau_j)_{j\in\N}$.
\end{remark}
\begin{proof}[Proof of \Cref{cor:better-sub-convergence}.]
    As in \Cref{lem:par-scaling}, consider $\gamma_j\vcentcolon=(\gamma(\tau_j+\rho_j^4\ \cdot\ ,\ \cdot\ )-p_j)/\rho_j$ with $\rho_j$, $p_j$ as in the statement. Further, as $K\subseteq\H^2$ is compact, there exists $\rho\in (0,1)$ with $\rho\leq u^{(2)}\leq 1/\rho$ for all $u=(u^{(1)},u^{(2)})\in K$. 
    
    Using \eqref{eq:linfty-control} and \Cref{lem:BD}, there exist $\delta_0=\delta_0(M)$ and  $D=D(L,M)$ such that each $\gamma_j$ exists on $[0,\delta_0\rho^4]\times\S^1$ with
    \begin{equation}\label{eq:gammaj2-unif-bd}
        \frac{\rho}{D}\leq \gamma_j^{(2)}\leq \frac{D}{\rho} \quad\text{on $[0,\delta_0\rho^4]\times\S^1$ for all $j\in\N$}.
    \end{equation}    
    Further, by \eqref{eq:glob-en-est}, writing again $\curv_j$ for the curvature vector of $\gamma_j$, 
    \begin{equation}
        \sup_{t\in[0,\delta_0\rho^4]}\||(\nabla_s^{\bot})^m\curv_j|_g\|_{L^2(\dd s)} \leq C(m,A_m,L,M,[\eta])\quad\text{for all $j\in\N$ and $m\in\N_0$}. 
    \end{equation}
    The Gagliardo--Nirenberg type interpolation inequality in \cite[Proposition 4.1]{dallacquaspener2017} and \eqref{eq:len-low-bound}  combined with $\||\curv_j|_g\|_{L^2(\dd s)}\leq C(M)$ immediately imply
    \begin{equation}
        \sup_{t\in[0,\delta_0\rho^4]}\||(\nabla_s^{\bot})^m\curv_j|_g\|_{L^{\infty}(\S^1)} \leq C(m,A_{m+1},L,M,[\eta])\quad\text{for all $j\in\N$ and $m\in\N_0$}. 
    \end{equation}
    Fix now diffeomorphisms $\phi_j\colon\S^1\to\S^1$ such that $\gamma_j(0)\circ\phi_j$ is parametrized by constant hyperbolic speed. That is,
    \begin{equation}
        |\partial_x(\gamma_j\circ\phi_j)|_g\Big|_{t=0} = \frac{\mathcal{L}(\gamma_j(0))}{2\pi}.
    \end{equation}
     The length $\mathcal{L}(\gamma_j(0))$ is uniformly bounded from above and below, cf.\ \eqref{eq:len-low-bound} for the lower bound. Proceeding as on \cite[pp. 20-22]{dallacquaspener2017}, using the above $L^{\infty}$-bounds, the control of the parametrization at $t=0$, and \eqref{eq:linfty-control}, for $m\in\N_0$, we have
    \begin{equation}
        \|\partial_x^m(\gamma_j\circ\phi_j)\|_{L^{\infty}([0,\delta_0\rho^4]\times\S^1)},\ \|\partial_x^m\partial_t(\gamma_j\circ\phi_j)\|_{L^{\infty}([0,\delta_0\rho^4]\times\S^1)} \leq C(m,L,\gamma_0).
    \end{equation}
    Using again \eqref{eq:linfty-control}, there exists a smooth family of immersions $\hat{\gamma}\colon[0,\delta_0\rho^4]\times\S^1\to\H^2$ such that, after passing to a subsequence without relabeling,
    \begin{equation}\label{eq:subconv-cor-smooth-convergence}
        \gamma_j(\cdot,\phi_j)=\frac{\gamma(\tau_j+\rho_j^4\ \cdot\ ,\phi_j)-p_j}{\rho_j}\to\hat{\gamma}\quad\text{in the $C^{\infty}([0,\delta_0\rho^4]\times\S^1)$ topology}.
    \end{equation}
    In the case where $\gamma$ is a Willmore flow, \eqref{eq:subconv-cor-smooth-convergence}, the reparametrization invariance of $\nabla\mathcal{E}$, and \Cref{rem:endecr} imply
    \begin{align}
        \frac{1}{4} \int_{0}^{\delta_0\rho^4}\int_{\S^1}&\frac{1}{(\hat{\gamma}^{(2)})^4}|\nabla\E(\hat{\gamma})|_g^2\dd s\dd t \\
        &=\lim_{j\to\infty} \frac{1}{4}\int_{0}^{\delta_0\rho^4} \int_{\S^1}\frac{1}{(\gamma_j^{(2)})^4} |\nabla\E(\gamma_{j})|_g^2\dd s_j\dd t \\
        &= \lim_{j\to\infty} \int_{0}^{\delta_0\rho^4} -\partial_t\E(\gamma_j(t))\dd t =\lim_{j\to\infty} \E(\gamma_j({0}))-\E(\gamma_j(\delta_0\rho^4)) \\
        &= \lim_{j\to\infty} \E(\gamma(\tau_j))-\E(\gamma(\tau_j+\rho_j^4\delta_0\rho^4)) = 0,\label{eq:blowups-are-critical}
    \end{align} 
    since $\lim_{t\to T}\E(\gamma(t))$ exists. Using \Cref{lem:par-scaling} and \eqref{eq:subconv-cor-smooth-convergence} we conclude $\partial_t\hat\gamma=0$ and $\nabla\E(\hat{\gamma})\equiv 0$. 
    
    In the case where $\gamma$ is a conformally constrained Willmore flow, using \eqref{eq:subconv-cor-smooth-convergence}, \eqref{eq:lambda-conf-constr-will}, and the fact that the denominator in \eqref{eq:lambda-conf-constr-will} is uniformly controlled, we have that
    \begin{equation}
        \lambda(\gamma_j(t)\circ\phi_j) \to \lambda(\hat{\gamma}(t)) \quad\text{in the $C^{\infty}([0,\delta_0\rho^4])$ topology}.
    \end{equation}
    A similar computation as in \eqref{eq:blowups-are-critical} yields $\partial_t\hat{\gamma}=0$ and $\nabla\E(\hat \gamma)-\lambda(\hat{\gamma})\hat\curv\equiv 0$ on $[0,\delta_0\rho^4]\times\S^1$.
\end{proof}

\subsection{Global existence and convergence of the conformally constrained Willmore flow}
\label{sec:conv-conf-constr-will}

In this section, we prove the following convergence theorem for the conformally constrained Willmore flow of tori of revolution. 

Let $\gamma\colon\S^1\to\H^2$ be a smooth immersion with $\mathcal{L}(\gamma)=L$. Writing $\S^1=\R/2\pi\Z$, consider the diffeomorphism $\psi\colon\S^1\to\S^1$ with
\begin{equation}
    \psi(y)=\frac{2\pi}{L} \int_{0}^{y} |\partial_x\gamma|_g\dd x
\end{equation}
for $y\in\S^1$. Writing $\varphi=\psi^{-1}$, we define $\widetilde{\gamma} \vcentcolon= \gamma\circ\varphi$, and in the following refer to this special parametrization as the reparametrization of $\gamma$ by \emph{constant hyperbolic speed.}

\begin{theorem}\label{thm:main-conf-constr-wf}
    Let $\gamma\colon[0,T)\times\S^1\to\H^2$ be a maximal conformally constrained Willmore flow with 
    \begin{equation}\label{eq:tot-hyp-curv-constraint}
        \inf_{0\leq t<T}\Big|\int_{\S^1}\scurv\dd s(t)\Big| > 0.
    \end{equation}
    Then $T=\infty$ and, if $\widetilde{\gamma}(t)$ is the reparametrization of $\gamma(t)$ by constant hyperbolic speed for each $0\leq t<T$, then
    \begin{equation}
        \widetilde{\gamma}(t) \to \hat{\gamma} \quad\text{in the $C^{\infty}(\S^1)$ topology}
    \end{equation}
    as $t\to\infty$ where $\hat{\gamma}$ is the profile curve of a conformally constrained Willmore torus of revolution in the same conformal class as the initial datum. 
\end{theorem}

\begin{remark}
    By \Cref{rem:hyp-tot-curv-contro-by-energy,lem:hyp-tot-curv-control-by-length}, the condition \eqref{eq:tot-hyp-curv-constraint} is necessarily satisfied along a conformally constrained Willmore flow starting at $\gamma_0$ if we either suppose that $\E(\gamma_0)< 16$ or alternatively that $\Ll(\gamma_0)<2\pi |m|$ where $m\in\Z$ is the turning number of $\gamma_0$. We show in \Cref{rem:circular-elastica} that, in any rectangular conformal class, initial data with arbitrarily high elastic energy satisfying the latter condition exist.
\end{remark}

Using \cite[Corollary~5.2]{rupp2020}, and proceeding similarly as in \cite{ruppspener2020} and \cite{pozzetta2022}, one can prove the following length-constrained \L ojasiewicz--Simon gradient inequality for the elastic energy in $\H^2$.

\begin{theorem}\label{thm:loja-e-constr}
    Let $\hat{\gamma}\colon\S^1\to\H^2$ be a smooth critical point of $\E+\lambda\Ll$ for some $\lambda\in\R$. Then there exist $C=C(\hat{\gamma})>0$, $\theta\in(0,\frac12]$ and $\sigma>0$ with the following property. For all immersions $\gamma\colon\S^1\to\H^2$ with $\Ll(\gamma)=\Ll(\hat{\gamma})$ such that there exists a diffeomorphism $\phi\colon\S^1\to\S^1$ satisfying $\|\hat{\gamma}-\gamma\circ \phi\|_{W^{4,2}(\S^1,\R^2)}<\sigma$, we have
    \begin{equation}\label{eq:loja-e-constr}
        |\E(\gamma)-\E(\hat{\gamma})|^{1-\theta}\leq C\||\nabla\E(\gamma)-\frac{\int_{\S^1}\langle \nabla\E(\gamma),\curv\rangle_g\dd s}{\E(\gamma)}\cdot \curv|_g\|_{L^2(\dd s)}.
    \end{equation} 
\end{theorem}

The following lemma is fundamental in applying \Cref{thm:loja-e-constr} and is proven with the same elementary computation as \cite[Lemma 4.10]{ruppspener2020}.

\begin{lemma}\label{lem:beschdes-lemma}
    Consider $T\in (0,\infty]$ and a family of immersions $\gamma\colon [0,T)\times \S^1\to\H^2$. If $\widetilde{\gamma}(t)$ is the constant hyperbolic speed reparametrization of $\gamma(t)$ for each $0\leq t<T$, we have
    \begin{equation}\label{eq:rs-hl}
        \||\partial_t \widetilde{\gamma}(t)|_g\|_{L^2(\dd x)} \leq \Big(\frac{4\pi}{\mathcal{L}(\gamma(t))} +  16 \pi \E(\gamma(t))\Big)^{\frac12} \cdot \||\partial_t\gamma(t)|_g\|_{L^2(\dd s)}.
    \end{equation}
\end{lemma}

\begin{proof}[Proof of \Cref{thm:main-conf-constr-wf}]
    Write $L\vcentcolon=\Ll(\gamma_0)$. Then $\Ll(\gamma(t))=L$ for all $0\leq t<T$. We fix any sequence $t_j\nearrow T$, let $\rho_j\vcentcolon=\min_{\S^1}\gamma^{(2)}(t_j)$ and choose $p_j\in\R\times\{0\}$ such that $(p_j^{(1)},\rho_j)\in\gamma(t_j,\S^1)$. After passing to a subsequence without relabeling, we find
    \begin{equation}\label{eq:conv-1-constr}
        \frac{\widetilde{\gamma}(t_j)-p_j}{\rho_j} \to \hat{\gamma} \quad\text{in the $C^{\infty}(\S^1)$ topology}
    \end{equation}
    for some critical point $\hat{\gamma}$ of $\E+\lambda\Ll$, for some $\lambda\in\R$, using \Cref{cor:better-sub-convergence}. 
    Indeed, $\rho_j \leq \gamma^{(2)}(t_j) \leq e^{L/2} \rho_j$ and therefore
    \begin{equation}
        L = \mathcal{L}(\gamma(t_j)) \geq \frac{1}{\lVert \gamma^{(2)}(t_j)\rVert_\infty} \mathcal{L}_{\mathbb{R}^2}(\gamma(t_j)) \quad \textrm{so that \; } \mathcal{L}_{\mathbb{R}^2}(\gamma(t_j)) \leq Le^{L/2}\rho_j.
    \end{equation}
    From this follows that, if $p_j = (p_j^{(1)},0)$ is chosen suitably, we have $ \frac{\tilde{\gamma}(t_j)- p_j}{\rho_j} \in [-Le^{L/2}, L e^{L/2}] \times [ 1, e^{L/2}]$ and \eqref{eq:linfty-control} is ensured.

    By a standard argument, using \eqref{eq:dtW-for-conf-constr-WF}, we can w.l.o.g.\ assume that $\E(\gamma(t))>\E(\hat{\gamma})$ for all $0\leq t<T$.
    Moreover, as $\E(\gamma(t))$ is non-increasing, $\lim_{t\nearrow T}\E(\gamma(t))=\E(\hat{\gamma})$ by \eqref{eq:conv-1-constr}.

    By \Cref{thm:loja-e-constr}, there exist $C_{\mathrm{LS}}>0$, $\theta\in(0,\frac12]$ and $\sigma>0$ such that \eqref{eq:loja-e-constr} is satisfied. For $j\in\N$, define the smooth function $G_j\colon [0,(T-t_j)/\rho_j^4)\to(0,\infty)$,
    \begin{equation}
        G_j(t)= \Big( \E\big((\gamma(t_j+t\rho_j^4)-p_j)/\rho_j\big)-\E(\hat{\gamma}) \Big)^{\theta}=(\E(\gamma_j(t))-\E(\hat\gamma))^{\theta}
    \end{equation}
    where $\gamma_j(t,\cdot)\vcentcolon=(\gamma(t_j+t\rho_j^4,\cdot)-p_j)/\rho_j$. By \eqref{eq:conv-1-constr}, passing to a subsequence, we may assume that $\|(\widetilde{\gamma}(t_j)-p_j)/\rho_j-\hat{\gamma}\|_{W^{4,2}}<\sigma$ for all $j\in\N$. Then define
    \begin{equation}\label{eq:conv-2-constr}
        s_j\vcentcolon= \sup\{ \bar{t}\in [0,(T-t_j)/\rho_j^4)\mid \Big\|\frac{\widetilde{\gamma}(t_j+t\rho_j^4)-p_j}{\rho_j}-\hat{\gamma}\Big\|_{W^{4,2}}<\sigma\text{ for all $0\leq t\leq \bar{t}$} \}.
    \end{equation} 
    For any $0< t<s_j$, using $\|(\widetilde{\gamma}(t_j+t\rho_j^4)-p_j)/\rho_j-\hat{\gamma}\|_{W^{4,2}}<\sigma$, w.l.o.g.\ choosing $\sigma$ sufficiently small, we may assume
    \begin{equation}\label{eq:conv-4-constr}
       0 < \alpha < \frac{\gamma^{(2)}(t_j + t \rho_j^4) - p_j^{(2)}}{\rho_j} = \gamma_j^{(2)}(t) < \frac{1}{\alpha}
    \end{equation}
    for some $\alpha=\alpha(\hat{\gamma},\sigma)\in(0,1)$. Thus, using \eqref{eq:conf-constr-wf-eq}, and the fact that
    \begin{equation}
        0=-\lambda(\gamma_j(t)) \partial_t\Ll(\gamma_j(t)) = \int_{\S^1}\langle \lambda(\gamma_j(t))\curv,\partial_t\gamma_j\rangle_g\dd s,
    \end{equation}
    we have 
    \begin{align}
        -&G_j'(t)=\theta \Big( \E\big(\gamma_j(t)\big)-\E(\hat{\gamma}) \Big)^{\theta-1}  \cdot \int_{\S^1} \langle -\nabla\E\big(\gamma_j(t)\big),\partial_t\gamma_j(t) \rangle_g \dd s\\
        &=\theta \Big( \E\big(\gamma_j(t)\big)-\E(\hat{\gamma}) \Big)^{\theta-1}  \cdot \int_{\S^1} \langle -(\nabla\E\big(\gamma_j(t)\big)-\lambda(\gamma_j(t))\curv_{\gamma_j(t)}),\partial_t\gamma_j(t) \rangle_g \dd s\\
        &= \frac{\theta}{4} \Big( \E\big(\gamma_j(t)\big)-\E(\hat{\gamma}) \Big)^{\theta-1} \cdot\Big\|\frac{1}{(\gamma_j^{(2)}(t))^2}|\nabla \E (\gamma_j(t))-\lambda(\gamma_j(t))\curv_{\gamma_j(t)} |_g\Big\|_{L^2(\dd s)}^2\\
        &\geq \frac{\alpha^4\theta}{4}  \Big( \E\big(\gamma_j(t)\big)-\E(\hat{\gamma}) \Big)^{\theta-1} \||\nabla \E (\gamma_j(t))-\lambda(\gamma_j(t))\curv_{\gamma_j(t)} |_g\|_{L^2(\dd s)}^2
    \end{align}
    such that, by the $L^2(\dd s)$-orthogonality estimate
    \begin{align}
        \||\nabla \E (\gamma_j(t))-&\lambda(\gamma_j(t))\curv_{\gamma_j(t)} |_g\|_{L^2(\dd s)} \\
        &\geq \||\nabla \E (\gamma_j(t))-\frac{\int_{\S^1}\langle\nabla\E(\gamma_j(t)),\curv_{\gamma_j(t)}\rangle_g\dd s}{\E(\gamma_j(t))}\cdot\curv_{\gamma_j(t)} |_g\|_{L^2(\dd s)}
    \end{align}
    and \eqref{eq:loja-e-constr}, using $0<t<s_j$ and \eqref{eq:conv-2-constr},
    \begin{align}
        -\frac{4}{\theta}G_j'(t) &\geq \frac{\alpha^4}{C_{\mathrm{LS}}} \||\nabla \E (\gamma_j(t))-\lambda(\gamma_j(t))\curv_{\gamma_j(t)} |_g\|_{L^2(\dd s)}\\
        &\geq 4\frac{\alpha^{8}}{C_{\mathrm{LS}}} \Big\|\frac{1}{4(\gamma_j^{(2)}(t))^{4}}|\nabla \E (\gamma_j(t))-\lambda(\gamma_j(t))\curv_{\gamma_j(t)} |_g\Big\|_{L^2(\dd s)} = 4\frac{\alpha^{8}}{C_{\mathrm{LS}}} \||\partial_t \gamma_j(t)|_g\|_{L^2(\dd s)} \\
        &= 4\frac{\alpha^{8}}{C_{\mathrm{LS}}} \||\partial_t (\gamma(t_j+t\rho_j^4)-p_j)/\rho_j|_g\|_{L^2(\dd s)},
    \end{align}
    also using \Cref{lem:par-scaling} and \eqref{eq:conf-constr-wf-eq} in the penultimate step. Using \Cref{lem:beschdes-lemma} and \eqref{eq:len-low-bound}
    as well as \eqref{eq:dtW-for-conf-constr-WF},
    \begin{align}
        -G_j'(t) &\geq C(C_{\mathrm{LS}},\E(\gamma_0),\alpha) \|\partial_t(\widetilde{\gamma}(t_j+t\rho_j^4)-p_j)/\rho_j\|_{L^2(\dd x)},
    \end{align}
    where we again used \eqref{eq:conv-4-constr} to estimate the $|\cdot|_g$ norm from below by the Euclidean norm. Integrating on $[0,t]$ therefore yields for $0<t<s_j$
    \begin{align}
        \Big\| \frac{\widetilde{\gamma}(t_j)-p_j}{\rho_j} - \frac{\widetilde{\gamma}(t_j+t\rho_j^4)-p_j}{\rho_j} \Big\|_{L^2(\dd x)} &\leq C(C_{\mathrm{LS}},\E(\gamma_0),\alpha) \big( G_j(0) - G_j(t) \big) \\
        &\leq C(C_{\mathrm{LS}},\E(\gamma_0),\alpha) G_j(0)       
        \to 0, \label{eq:conv-3-constr}
    \end{align}
    as $j\to\infty$, uniformly in $0<t<s_j$. We will now show that
    \begin{align}
        \label{cl:conv-constr}
        \text{ there exists $J\in\N$ with $s_J=(T-t_J)/\rho_J^4$}.
    \end{align}
    To this end, for the sake of contradiction, suppose that $s_j\in (0,(T-t_j)/\rho_j^4)$ for all $j\in\N$. By continuity, \eqref{eq:conv-3-constr} continues to hold for $t=s_j$ for any $j\in\N$. Applying \Cref{cor:better-sub-convergence} to the sequence $\tau_j\vcentcolon=t_j+s_j\rho_j^4$ and using that \eqref{eq:linfty-control} is satisfied by \eqref{eq:conv-2-constr}, passing to a subsequence, we have 
    \begin{equation}
        \frac{\widetilde{\gamma}(t_j+s_j\rho_j^4)-p_j}{\rho_j} \to \overline{\gamma}
    \end{equation}
    for some critical point $\overline{\gamma}$ of $\E+\tilde\lambda \mathcal{L}$ for some $\tilde\lambda\in\R$. By \eqref{eq:conv-2-constr} and \eqref{eq:conv-1-constr}, we have $\|\hat{\gamma}-\overline{\gamma}\|_{W^{4,2}}=\sigma>0$, so particularly $\|\hat{\gamma}-\overline{\gamma}\|_{L^2(\dd x)}>0$. However, passing to the limit in \eqref{eq:conv-3-constr} gives
    \begin{align}
        \|\hat{\gamma}-\overline{\gamma}\|_{L^2(\dd x)} \leq \liminf_{j\to\infty} C(C_{\mathrm{LS}},\E(\gamma_0),\alpha)  G_j(0)
        = 0,
    \end{align}
    a contradiction, so \eqref{cl:conv-constr} is proven. \Cref{cor:global-existence} and \eqref{eq:conv-4-constr} immediately yield that $T=\infty$. Arguing as in \eqref{eq:conv-3-constr}, we find that
    \begin{equation}
        \Big(\frac{\widetilde{\gamma}(t_J+t\rho_J^4)-p_J}{\rho_J}\Big)_{t\in[0,\infty)} \subseteq L^2(\dd x)
    \end{equation}
    is a Cauchy-sequence. Finally, the theorem is proven with a standard subsequence argument, repeatedly using \Cref{cor:better-sub-convergence} to obtain convergence in the topology of $C^{\infty}(\S^1)$.
\end{proof}

\begin{proof}[Proof of \Cref{intro:thm:main-conf-constr-wf}]
    If \eqref{item:intro:thm_conformal_case_1} is not satisfied, then, after replacing $f(0)$ with $f(T-\delta)$ with $\delta>0$ sufficiently small, we may assume \eqref{eq:tot-hyp-curv-constraint}. In this case, Part \eqref{item:intro:thm_conformal_case_2} of the statement follows from \Cref{thm:main-conf-constr-wf} together with \Cref{prop:Schl2}. The last  statement is a consequence of \Cref{rem:hyp-tot-curv-contro-by-energy} and \Cref{lem:hyp-tot-curv-control-by-length}, recalling that the turning number does not change along the flow and that the length is also fixed along the conformally constrained Willmore flow.
\end{proof}

We now prove \Cref{cor:intro_new_conf_willmore_tori}. For the definition of the Willmore energy for surfaces in $\S^3$ or $\R^4$, we refer to \cite[Section 1]{kuwertschaetzle2012}. Due to the conformal invariance of the Willmore functional, after composing with a stereographic projection, we may equivalently consider the tori $f_{m,r}$ from \Cref{cor:intro_new_conf_willmore_tori} (which are immersed in $\S^3$) to be in $\R^3$ instead. Choosing the stereographic projection appropriately, one can show that the symmetry with respect to $u\in \S^1=\R /2\pi\Z$ of the tori $f_{m,r}$ corresponds to a rotational symmetry. This is due to the fact that stereographic projections preserve circles. After applying an isometry of $\R^3$, we may assume that the $x$-axis is the axis of revolution. \Cref{cor:intro_new_conf_willmore_tori} is thus a consequence of the following result for surfaces of revolution in $\R^3$.

\begin{corollary}\label{cor:new_conf_willmore_tori}
    Let $f_{m,r}=f_{\gamma_{m,r}}$ where $\gamma_{m,r}\colon \S^1\to\mathbb{H}^2$ is an $m$-fold circle of radius $r\in (0,1)$ centered at $(0,1)$.
    For any $m\in\N$, $m\geq 3$, there exists $r\in (0,1)$ and a rotationally symmetric conformally constrained Willmore torus $f$ with $8\pi\leq \mathcal{W}(f)< \mathcal{W}(f_{m,r})$ whose profile curve is not circular and has turning number $m$.
\end{corollary}

\begin{proof}
    Fix $m\in\N$. The immersion $f_{m,r}$ is a conformally constrained Willmore torus with Lagrange multiplier $\lambda_r \vcentcolon= r^{-2}-2$  and the hyperbolic length $L\vcentcolon=\mathcal{L}(\gamma_{m,r})$ satisfies $L<2\pi m$ provided $r<1/\sqrt{2}$, see \Cref{app:cmc-tori}. As a consequence of the second variation formula in \cite{langersinger1984} (see \Cref{lem:circle_stability} for the details) we have that $\gamma_{m,r}$, $r\in (0,1/\sqrt{2})$, is unstable for $\mathcal{E}$ with respect to length-preserving variations provided
    \begin{align}\label{eq:range_r}
        \frac{\sqrt{1+2m}}{1+m}<r<\frac{1}{\sqrt{2}}.
    \end{align}
    Such an $r$ exists if $m\geq 3$, in which case we have $\mathcal{E}(\gamma_{m,r})>16$, see  \eqref{eq:LET_gamma_mr}. \Cref{lem:circle_stability} now yields a smooth perturbation $\tilde{\gamma}$ of $\gamma$ with $\mathcal{L}(\tilde{\gamma})=\mathcal{L}(\gamma_{m,r})$, $T(\tilde{\gamma})=T(\gamma_{m,r})=m$, and $16<\E(\tilde{\gamma})<\E(\gamma_{m,r})$. By \Cref{intro:thm:main-conf-constr-wf}\eqref{item:intro:thm_conformal_case_2}, the conformally constrained Willmore flow with initial datum $f_{\tilde{\gamma}}$ exists globally and converges to a  rotationally symmetric conformally constrained Willmore torus $f_\infty$. The limiting profile curve $\gamma_\infty$ satisfies $T(\gamma_\infty)=m\geq 3$, so $\mathcal{W}(f_\infty)\geq 8\pi$ by the Li--Yau inequality. 
\end{proof}

\begin{remark}\label{rem:admissible_conf_class}
The admissible values of $r$ in the proof of \Cref{cor:new_conf_willmore_tori} can be used to determine the admissible rectangular conformal classes. By continuity and monotonicity in $r$ of the length \eqref{eq:LET_gamma_mr}, for any $L_0 \in (2\pi\sqrt{1+2m},2\pi m)$, there exists $r$ as in \eqref{eq:range_r} and $m\geq 3$ such that $\mathcal{L}_{\gamma_{m,r}}=L_0$. Taking into account all $m\geq 3$, we find that for any $L_0\in (2\pi \sqrt{7},6\pi)\cup (6\pi,\infty)$ there exists $m\geq 3$ and $r$ satisfying \eqref{eq:range_r} with $\mathcal{L}(\gamma_{m,r})=L_0$. By \Cref{lem:Rcurv-orth-Cfomega}, the set of admissible conformal classes is thus given by $\big((\sqrt{7},3)\cup (3,\infty)\big) \cdot(0,1)\subset\mathcal{M}$. We do not expect this range to be optimal.
\end{remark}

\section{Singularities of the Willmore flow: the inverted catenoid}\label{sec:catenoid}
In \cite{dallacquamullerschatzlespener2020}, singular examples for the Willmore flow of tori of revolution are studied. It is shown that curves of vanishing Euclidean total curvature exhibit a non-convergent behavior. Here, we seek to understand their singular behavior. We first show that, if such flows are global and the area is uniformly bounded from above and below, they converge to an inverted catenoid in a weak sense. We then study how to start the Willmore flow from the inverted catenoid.

\subsection{Flowing into an inverted catenoid}

In this section, we identify $\S^1=[-1,1]\ /\sim$ where the end-points $\pm1$ are equivalent by ``$\sim$''. As in \Cref{sec:a-priori-lambda-control}, $T(\gamma)$ denotes the turning number of an immersion $\gamma\colon\S^1\to\R^2$. 
\begin{remark}\label{rem:on-inf-length-elastica}
    Throughout this section, the terminology in the classification of elastica in \cite{langersinger1984,muellerspener2020} is used. In particular, we recall the following facts, see also \cite[Remark~2.17]{schlierf2023}. Let $\gamma\colon (a,b) \to\H^2$ be a free elastica with $\Ll(\gamma)=\infty$ and bounded elastic energy. Then $\gamma$ either is (a segment of) a geodesic or of an asymptotically geodesic elastica. If $\lim_{x\searrow a}\gamma(x)=\lim_{x\nearrow b}\gamma(x)$, then $\gamma$ parametrizes an asymptotically geodesic elastica and $\E(\gamma)=8$. Moreover, if $\lim_{x\searrow a}\gamma^{(2)}(x)=\lim_{x\nearrow b}\gamma^{(2)}(x)=0$ but $\lim_{x\searrow a}\gamma(x)\neq\lim_{x\nearrow b}\gamma(x)$, then $\gamma$ parametrizes a geodesic.
\end{remark}

\begin{lemma}\label{lem:conseq-of-tot-curv-0}
    Consider an immersion $\gamma\colon\S^1\to\R^2$ with 
    \begin{equation}\label{eq:symm-lem}
        \gamma(-x)=\begin{pmatrix}
            -1&0\\0&1
        \end{pmatrix}\gamma(x)\quad\text{for all $x\in\S^1$}
    \end{equation}
    and $T(\gamma)=0$. Then $\gamma(0),\gamma(\pm1)\in \{0\}\times\R$ and $\partial_x\gamma^{(1)}(0),\partial_x\gamma^{(1)}(\pm1)\neq 0$. Moreover, there exists $h\in(0,1)$ with $\gamma(h)=\gamma(-h)\in \{0\}\times\R$.
\end{lemma}
\begin{proof}
    Equation \eqref{eq:symm-lem} immediately yields $\gamma(0),\gamma(\pm1)\in \{0\}\times\R$.
    Moreover, differentiating \eqref{eq:symm-lem} 
    we find $\partial_x\gamma^{(2)}(0)=\partial_x\gamma^{(2)}(\pm1)=0$ which yields 
    \begin{equation}\label{eq:conseq-of-tot-curv-0-1}
        \partial_x\gamma^{(1)}(0),\partial_x\gamma^{(1)}(\pm1)\neq 0.
    \end{equation}   
    Differentiating \eqref{eq:symm-lem} twice with respect to the Euclidean arc-length parameter $s_{\R^2}$, $\scurv_{\R^2}(-x)=\scurv_{\R^2}(x)$ for all $x\in\S^1$ so that
    \begin{equation}
        0=T(\gamma)= \frac1\pi \int_0^1 \scurv_{\R^2}\dd s_{\R^2}.
    \end{equation}
    Writing $\partial_x\gamma(x)/|\partial_x\gamma(x)|=(\cos(\theta(x)),\sin(\theta(x)))$ for a smooth function $\theta\colon\S^1\to\R$, we get $\scurv_{\R^2}=\partial_{s_{\R^2}}\theta$, and therefore $0=\theta(1)-\theta(0)$, that is,
    \begin{equation}
        \frac{1}{|\partial_x\gamma(0)|}\partial_x\gamma^{(1)}(0)=\frac{1}{|\partial_x\gamma(1)|}\partial_x\gamma^{(1)}(1).
    \end{equation}
    However, if no $h\in(0,1)$ as in the statement exists, using $\gamma^{(1)}(0)=\gamma^{(1)}(\pm 1)=0$ and \eqref{eq:conseq-of-tot-curv-0-1}, we find $\partial_x\gamma^{(1)}(0)\cdot\partial_x\gamma^{(1)}(1) < 0$, a contradiction!
\end{proof}

\begin{lemma}\label{lem:max-g2-vs-diam}
    Let $\gamma\colon\S^1\to\H^2$ be an immersion with $\E(\gamma)\leq M$. Then there exists a constant $c(M)>0$ with
    \begin{equation}\label{eq:max-g2-vs-diam-1}
        \max_{\S^1} \gamma^{(2)} \geq c(M) \cdot \mathrm{diam}(f_{\gamma}).
    \end{equation}
\end{lemma}
\begin{proof}
    For the sake of contradiction, suppose that there is a sequence of immersions $(\gamma_j)_{j\in\N}$ with $\max_{\S^1}\gamma_j^{(2)} \to 0$, $\E(\gamma_j)\leq M$ and $\mathrm{diam}(f_{\gamma})=1$. In particular, by \cite[Lemma~2.6]{dallacquamullerschatzlespener2020}, $\inf_{j\in\N}\Ll_{\R^2}(\gamma_j)>0$. Therefore, \cite[Proposition~2.8]{schlierf2023} yields that $\E(\gamma_j)\to\infty$, a contradiction!
\end{proof}

\begin{proposition}\label{prop:sing}
    Let $\gamma_0\colon\S^1\to\H^2$ be an immersion with turning number $T(\gamma_0)=0$ satisfying
    \begin{equation}\label{eq:symm}
        \gamma_0(-x)=\begin{pmatrix}
            -1&0\\0&1
        \end{pmatrix}\gamma_0(x)\quad\text{for all $x\in\S^1$}
    \end{equation}
   and let $\E(\gamma_0)\in (16,24)$, that is, $\W(f_{\gamma_0})\in (8\pi,12\pi)$. Denote by $\gamma\colon [0,T)\times\S^1\to\H^2$ the maximal Willmore flow starting at $\gamma_0$ and let $f(t,\cdot)\vcentcolon=f_{\gamma(t,\cdot)}$. Suppose that 
    \begin{equation}\label{eq:ass-sing-ex}
        T=\infty\quad\text{and}\quad 0<\inf_{t\in[0,\infty)}\A(f(t)) \leq \sup_{t\in[0,\infty)} \A(f(t))<\infty,
    \end{equation}
    with $ \A(f(t))$ the surface area of $f(t)$. Then, for every sequence $t_j\nearrow \infty$, there are reparametrizations $\tilde\gamma_j$ of $\gamma(t_j)$ by constant Euclidean speed converging to an asymptotically geodesic elastica uniformly on $\S^1$ and in the topology of $C^{\infty}_{\mathrm{loc}}(\S^1\setminus\{0\})$, after horizontal translation and passing to a subsequence.
\end{proposition}

\begin{remark}\label{rem:numerics}
    Our singular limit classification in \Cref{prop:sing} requires \eqref{eq:ass-sing-ex} which we cannot verify analytically. However, the validity of this hypothesis is suggested by numerical experiments for the axi-symmetric Willmore flow of tori in \cite[Appendix~B, Figures~B.2 – B.4 and p.~883]{barrettgarckenuernberg2021}.
\end{remark}

\begin{proof}[Proof of \Cref{prop:sing}]
    By \cite[Lemma 1.1]{simon1993} and \Cref{rem:endecr}, \eqref{eq:ass-sing-ex} yields
    \begin{equation}\label{eq:prop-sing-0}
        0<\inf_{t\in[0,\infty)} \mathrm{diam}(f(t)) \leq \sup_{t\in[0,\infty)} \mathrm{diam}(f(t)) < \infty.
    \end{equation}
    Thus, using \Cref{lem:max-g2-vs-diam}, also 
    \begin{equation}\label{eq:prop-sing-1}
        \inf_{t\in[0,\infty)}\max_{\S^1}\gamma^{(2)}(t,\cdot) > 0.
    \end{equation}
    \begin{claim}\label{cl:sing-1}
        There are $0<\ell<L<\infty$ with $\Ll_{\R^2}(\gamma(t))\in (\ell,L)$ for all $0\leq t<\infty$.
    \end{claim}
    The upper bound on the Euclidean length follows from \cite[Lemma~2.5]{dallacquamullerschatzlespener2020} and \eqref{eq:ass-sing-ex}. For the sake of contradiction, suppose that there is no lower bound $\ell>0$ as in the claim. Then there exists $\tau_j\nearrow \infty$ with $\Ll_{\R^2}(\gamma(\tau_j))\to 0$. Now, choosing $j$ sufficiently large such that $\inf_{t\geq 0}\max_{\S^1}\gamma^{(2)}(t,\cdot) > \Ll_{\R^2}(\gamma(\tau_j))$, and using that $\min_{\S^1}\gamma^{(2)}(\tau_j,\cdot) \geq \max_{\S^1}\gamma^{(2)}(\tau_j,\cdot) - \frac12\Ll_{\R^2}(\gamma(\tau_j))$, we obtain
    \begin{align}
        \Ll_{\R^2}(\gamma(\tau_j)) &= \int_{\S^1} \frac{|\partial_x\gamma(\tau_j,x)|}{\gamma^{(2)}(\tau_j,x)}\gamma^{(2)}(\tau_j,x)\dd x \geq \big(\max_{\S^1}\gamma^{(2)}(\tau_j,\cdot)- \frac12\Ll_{\R^2}(\gamma(\tau_j))\big)
        \Ll(\gamma(\tau_j))\\
        &\geq \Big(\inf_{t\geq 0}\max_{\S^1}\gamma^{(2)}(t,\cdot) - \frac{1}{2}\Ll_{\R^2}(\gamma(\tau_j))\Big) \cdot \frac{4\pi^2}{\E(\gamma_0)},
    \end{align}
    using \eqref{eq:len-low-bound} in the last step. This contradicts $\Ll_{\R^2}(\gamma(\tau_j))\to 0$ and thus \Cref{cl:sing-1} is proven.
    
    There is a smooth family $\varphi\colon[0,\infty)\times\S^1\to\S^1$ of diffeomorphisms of $\S^1$ with
    \begin{equation}
        (\varphi(t,\cdot)^{-1})(y)=\frac{2}{\Ll_{\R^2}(\gamma(t))} \int_0^y |\partial_x\gamma(t,x)|\dd x,
    \end{equation}
    cf.\ \Cref{sec:conv-conf-constr-will}.
     Define $\tilde\gamma(t,x)\vcentcolon=\gamma(t,\varphi(t,x))$ and abbreviate $\tilde\gamma_j\vcentcolon=\tilde\gamma(t_j,\cdot)$ where $t_j\nearrow\infty$ is as in the statement. Note that $|\partial_x\tilde\gamma(t,x)|=\frac{1}{2}\Ll_{\R^2}(\gamma(t))$ for all $t\geq 0$ and $x\in\S^1$. Since the turning number
     is preserved along the evolution and independent of the parametrization, see \cite[Remark~5.4]{muellerspener2020}, $T(\tilde\gamma(t))=T(\gamma_0)=0$ for all $0\leq t<\infty$.
    
    As in \cite[proof of Theorem~1.5]{schlierf2023}, we have
    \begin{equation}\label{eq:sing-2}
        \tilde{\gamma}_j(x) = \begin{pmatrix}
            -1&0\\0&1
        \end{pmatrix}\tilde{\gamma}_j(-x)\quad\text{for all $x\in \S^1$}.
    \end{equation}
    Note that using \eqref{eq:prop-sing-0}, we have
    \begin{equation}
        0<R\vcentcolon=\sup_{t\geq 0} \max_{\S^1}\gamma^{(2)}(t,\cdot) \leq \frac12\sup_{t\geq 0}\diam(f(t)) < \infty.
    \end{equation}
    Since $|\partial_x\tilde{\gamma}_j| \equiv \frac12\Ll_{\R^2}(\gamma(t_j))$, by \Cref{cl:sing-1}, there exists $\tilde{\gamma}_{\infty}\in W^{1,\infty}(\S^1,\R\times[0,\infty))$ such that, after horizontal translation and passing to a subsequence,
    \begin{equation}\label{eq:prop-sing-2}
        \tilde\gamma_j\to\tilde\gamma_{\infty}\text{ uniformly and }\tilde\gamma_j\rightharpoonup^*\tilde\gamma_{\infty}\text{ in $W^{1,\infty}(\S^1)$}.
    \end{equation}
    Using $\E(\gamma_0) < 24 = 3\cdot8$, \cite[Theorem~2.6]{schlierf2023} yields that $Z\vcentcolon=\{\tilde\gamma_{\infty}^{(2)}=0\}$ consists of at most two points. Write $m\in\{0,1,2\}$ for the number of elements in $Z$. By \eqref{eq:sing-2}, $Z$ is \emph{symmetric}, i.e., $x\in Z$ implies $-x\in Z$.
    \begin{claim}\label{cl:sing-2}
        We have $\tilde\gamma_j\to \tilde\gamma_{\infty}$ in $C^{\infty}_{\mathrm{loc}}(\S^1\setminus Z)$ and $\nabla\E(\tilde\gamma_{\infty})=0$ on $\S^1\setminus Z$. Moreover,
        \begin{equation}\label{eq:sing-3}
            \E(\tilde\gamma_{\infty}|_{\S^1\setminus Z}) \leq \E(\gamma_0) - 8m.
        \end{equation}
    \end{claim}
    Let $\varepsilon\in (0,1)$ and note that $I_{\varepsilon}\vcentcolon=\S^1\setminus B_{\varepsilon}(Z)$ is compact. If $Z=\emptyset$, then $\tilde{\gamma}_j$ are uniformly bounded from below. If $Z\neq\emptyset$, for $\varepsilon$ sufficiently small and $j$ sufficiently large, using \eqref{eq:prop-sing-1},
    \begin{align}
     0< c(\varepsilon) \leq \min_{I_{\varepsilon}}\gamma^{(2)}(t_j,\cdot) \leq  \frac16 \max_{\S^1}\gamma^{(2)}(t_j,\cdot)\leq \frac16 R,\label{eq:fenchel}
     \end{align}
    with $c(\varepsilon)\vcentcolon= \frac12 \inf_{I_{\varepsilon}} \tilde{\gamma}_{\infty}^{(2)}$. 
     Using \Cref{prop:il}, we get for $j\in\N$ sufficiently large
    \begin{equation}\label{eq:sing-4}
        \|(\nabla_s^{\bot})^k\curv_{\tilde\gamma_j}\|_{L^2(\dd s\measurerestr I_\varepsilon)} \leq C(c(\varepsilon),k,R,\gamma_0) \quad\text{ for all $k\in\N_0$}.
    \end{equation} 
    We argue now that the length is uniformly bounded from below. If $Z=\emptyset$, then $\Ll(\tilde\gamma_j|_{I_{\varepsilon}})=\Ll(\tilde\gamma_j)\geq 4\pi^2/\E(\gamma_0)$ by \eqref{eq:len-low-bound}. If $Z\neq\emptyset$, for $\varepsilon$ sufficiently small and $j$ sufficiently large, by \Cref{rem:hyplengthestimatesrange(optimal)}
    \begin{equation}
       \frac12  \Ll(\tilde\gamma_j|_{I_{\varepsilon}}) \geq \log(\max_{I_{\varepsilon}}\tilde\gamma_j^{(2)} / \min_{I_{\varepsilon}}\tilde\gamma_j^{(2)})=\log(\max_{\S^1}\tilde\gamma_j^{(2)} / \min_{I_{\varepsilon}}\tilde\gamma_j^{(2)}) \geq 1.
    \end{equation}
    
    So in both cases, using \cite[Proposition~4.1]{dallacquaspener2017} (also see \cite[Proposition~3.12]{schlierf2024} for a version for open curves), we have the $L^{\infty}$ bounds
    \begin{equation}\label{eq:sing-5}
        \|(\nabla_s^{\bot})^k\curv_{\tilde\gamma_j}\|_{L^\infty(I_\varepsilon)} \leq C(k+1,\gamma_0,\varepsilon) \quad\text{ for all $j\in\N$ and $k\in\N_0$}.
    \end{equation}
    Proceeding as in \cite[Lemma~3.4]{schlierf2023}, one finds that $(\tilde\gamma_j)_{j\in\N}$ is bounded in any $W^{k,\infty}(I_{\varepsilon})$ and therefore the smooth convergence locally in $\S^1\setminus Z$ follows from a standard subsequence argument. In particular, $\tilde\gamma_{\infty}\colon \S^1\setminus Z\to\H^2$ is a smooth immersion.

    For $\sigma>0$ sufficiently small, consider a smooth cut-off function $\eta_{\sigma}$ with $\chi_{(2\sigma,\infty)}\leq\eta_{\sigma}\leq \chi_{(\sigma,\infty)}$ and $|\eta_{\sigma}'|\leq 2\sigma^{-1}\chi_{(\sigma,2\sigma)}$. Define
    \begin{equation}
        h_{\sigma}\colon[0,\infty)\to\R,\quad h_{\sigma}(t) \vcentcolon= \int_{\S^1} \frac{1}{4(\gamma(t,x)^{(2)})^4} |\nabla\E(\gamma(t,x))|_g^2\eta_{\sigma}(\gamma^{(2)}(t,x))\dd s.
    \end{equation}
    Since $\eta_{\sigma}\leq 1$, \Cref{rem:endecr} implies
    \begin{equation}
        h_{\sigma}(t) \leq -\partial_t \E(\gamma(t,\cdot))\quad\text{so that}\quad h_{\sigma}\in L^1([0,\infty),\dd t).
    \end{equation}
    Using \Cref{lemma:evo} and computing as in \Cref{lem:evolkappa}, \eqref{eq:sing-4} and \eqref{eq:sing-5} yield that $\partial_th_{\sigma}\in L^{\infty}([0,\infty))$ with norm depending on $\sigma$. This yields $h_{\sigma}(t)\to 0$ for $t\to\infty$. Using \eqref{eq:prop-sing-2}, choose $\varepsilon=\varepsilon(\sigma)>0$ such that $\inf_{x\in I_{\varepsilon}}\tilde\gamma^{(2)}(t_j,x) > 2\sigma$ for $j$ sufficiently large. Then
    \begin{align}
        \int_{I_{\varepsilon}} &\frac{1}{4(\tilde\gamma(t_j,x)^{(2)})^4} |\nabla\E(\tilde\gamma(t_j,x))|_g^2\dd s \\
        &\leq \int_{\S^1} \frac{1}{4(\gamma(t_j,x)^{(2)})^4} |\nabla\E(\gamma(t_j,x))|_g^2\chi_{(2\sigma,\infty)}(\gamma^{(2)}(t_j,x))\dd s \leq h_{\sigma}(t_j)\to 0\quad\text{for $j\to\infty$},
    \end{align} 
    so that $\nabla\E(\tilde\gamma_{\infty})=0$ on $\S^1\setminus Z$. Finally, \eqref{eq:sing-3} follows exactly as in \cite[Proof of Theorem~1.1]{schlierf2023}.
    \begin{claim}\label{cl:sing-3}
        $Z$ consists of exactly one point.
    \end{claim}
    If $Z=\emptyset$, by \Cref{cl:sing-2}, $\tilde\gamma_{\infty}$ is a critical point of $\E$ with vanishing turning number. This contradicts \cite[Corollary~5.8]{muellerspener2020}. For the sake of contradiction, we assume that $Z$ consists of exactly two points. Since $Z$ is symmetric, either $Z=\{0,\pm1\}$ or $Z=\{\tilde x,-\tilde x\}$ for some $\tilde x\in (0,1)$. 

    In the former case, \eqref{eq:sing-2} yields $\tilde\gamma_{\infty}(0)=\tilde\gamma_{\infty}(\pm1)=(0,0)$. Using  \Cref{rem:hyplengthestimatesrange(optimal)} and \Cref{rem:on-inf-length-elastica}, $\tilde\gamma_{\infty}|_{(-1,0)}$ and $\tilde\gamma_{\infty}|_{(0,1)}$ both parametrize an asymptotically geodesic elastica. Moreover, with \eqref{eq:sing-3},
    \begin{equation}\label{eq:prop-sing-3}
        16 = \E(\tilde\gamma_{\infty}|_{\S^1\setminus Z}) \leq \E(\gamma_0)-8m < 24 - 16 = 8,
    \end{equation}
    a contradiction!

    So $Z=\{\tilde x,-\tilde x\}$ for some $\tilde x\in (0,1)$. If $\tilde\gamma_{\infty}(\tilde x)=\tilde\gamma_{\infty}(-\tilde x)$, then \Cref{cl:sing-2,rem:on-inf-length-elastica} yield that $\tilde\gamma_{\infty}|_{(-\tilde x,\tilde x)}$ and $\tilde\gamma_{\infty}|_{\S^1\setminus [-\tilde x,\tilde x]}$ both parametrize an asymptotically geodesic elastica and the same estimate as in \eqref{eq:prop-sing-3} yields a contradiction. So $\tilde\gamma_{\infty}(\tilde x)\neq\tilde\gamma_{\infty}(-\tilde x)$. So again using \Cref{cl:sing-2,rem:on-inf-length-elastica}, and the characterization of geodesics in $\H^2$, $\tilde\gamma_{\infty}|_{(-\tilde x,\tilde x)}$ and $\tilde\gamma_{\infty}|_{\S^1\setminus [-\tilde x,\tilde x]}$ both parametrize the same geodesic (namely the one ``connecting'' $\tilde\gamma_{\infty}(-\tilde x)$ and $\tilde\gamma_{\infty}(\tilde x)$). Due to the choice of parametrization, this yields $\tilde x=\frac12$. Using \eqref{eq:sing-2} and $\tilde\gamma_{\infty}(-\frac12)\neq \tilde\gamma_{\infty}(\frac12)$, we have $\tilde\gamma_{\infty}(\pm\frac12)\neq (0,0)$. 
    
    By \Cref{lem:conseq-of-tot-curv-0}, there exist $h_j\in (0,1)$ with $\tilde\gamma_j(- h_j)=\tilde\gamma_j(h_j)\in\{0\}\times(0,\infty)$. Since $\tilde\gamma_{\infty}(\pm \frac12)\notin \{0\}\times[0,\infty)$, we have $\varepsilon\vcentcolon=\liminf_{j\to\infty}|h_j-\frac12|>0$. Therefore, $\tilde\gamma_j|_{[-\frac12(1-\varepsilon),\frac12(1-\varepsilon)]}$ or $\tilde\gamma_j|_{\S^1\setminus [-\frac12(1+\varepsilon),\frac12(1+\varepsilon)]}$ has a  self-intersection for $j$ sufficiently large. So using \cite[Theorem~6.1]{schlierf2024}, \Cref{cl:sing-2}, and \Cref{rem:on-inf-length-elastica},
    \begin{equation}
        8\leq \E(\tilde\gamma_j|_{[-\frac12(1-\varepsilon),\frac12(1-\varepsilon)]})+ \E(\tilde\gamma_j|_{\S^1\setminus [-\frac12(1+\varepsilon),\frac12(1+\varepsilon)]}) \to 0 + 0 = 0
    \end{equation}
    for $j\to\infty$, a contradiction! In particular, \Cref{cl:sing-3} is established. 

    With \Cref{cl:sing-2}, \Cref{cl:sing-3} and \Cref{rem:on-inf-length-elastica}, it is clear that $\tilde\gamma_{\infty}|_{\S^1\setminus Z}$ parametrizes an asymptotically geodesic elastica. This completes the proof.
\end{proof}

\subsection{Starting at an inverted catenoid}

We now construct a Willmore flow starting at the inverted catenoid and prove \Cref{thm:intro:wf-starting-in-ic}. This initial datum  is not smooth, in fact it is only in $C^{1,\alpha}(\S^2,\R^3)\setminus C^2(\S^2,\R^3)$. In particular, monotonicity (or even finiteness) of the Willmore energy is not immediate.

Local well-posedness for the Willmore flow \eqref{eq:intro:WF} is well-known, cf.\
\cite{simonett2001,lecroneshaosimonett2020}, and is based on the following classical observation, see also \cite[proof of Lemma~4.1]{chillfasangovaschaetzle2009}.

\begin{remark}\label{rem:rep-of-tangential-Wf}
    For a closed surface $\Sigma$, consider a smooth, maximal Willmore flow $f\colon [0,T_{\max})\times \Sigma\to\R^3$ starting at $f_0$ with maximal existence time $T_{\max}$ and suppose that $\tilde{f}\colon[0,\tilde T)\times\Sigma\to\R^3$ is a smooth family of immersions with $\tilde{f}(0)=f_0\circ\Phi_0$ for a diffeomorphism $\Phi_0\colon\Sigma\to\Sigma$ and satisfying   \begin{equation}\label{eq:WF_normal_component}
        \partial_t^{\bot} \tilde f \vcentcolon= \langle \partial_t\tilde f,N_{\tilde f}\rangle N_{\tilde f} =  -\nabla\W(\tilde f)\quad\text{on $[0,\tilde T)\times \Sigma$}.
    \end{equation}
    Then $T_{\max}\geq \tilde T$ and there is a smooth family of diffeomorphisms $\Phi\colon [0,\tilde T)\times\Sigma\to\Sigma$ of $\Sigma$ with $\Phi(0)=\Phi_0$ such that $f(t)\circ \Phi(t)=\tilde f(t)$ for all $0\leq t<\tilde T$. Considering graphs over the initial datum, Equation \eqref{eq:WF_normal_component} corresponds to a strictly parabolic problem. The well-posedness of this has been precisely studied in \cite{simonett2001,lecroneshaosimonett2020}, see \Cref{sec:app:well-posedness}.
\end{remark}

If $\Sigma\subseteq\R^3$ is embedded with inward-pointing normal $\nu_{\Sigma}$, we denote by $L_{\Sigma}\vcentcolon=-d\nu_{\Sigma}$ the associated shape operator.

\begin{lemma}\label{lem:wf-starting-in-ic-1}
    There exist a smooth embedded sphere $\Sigma\subseteq\R^3$ and a continuous function  $a\colon\Sigma\to (0,\infty)$, $x\mapsto a_x$ such that $\det(I+\rho L_{\Sigma}(x))\neq 0$ for all $\rho\in\R$ with $|\rho|<a_x$, for all $x\in\Sigma$, and such that there exists $\rho_0\in W^{2,p}(\Sigma)$ for all $p\in[1,\infty)$ with $|\rho_0(x)|<a_x$ for all $x\in\Sigma$ and such that
    \begin{equation}\label{eq:lem-constr-Sigma-rho0}
        f_0\colon\Sigma\to\R^3, f_0(x)= x + \rho_0(x)\nu_{\Sigma}(x)
    \end{equation}
    parametrizes an inverted catenoid.
\end{lemma}
\begin{figure}[htb]
    \centering
    \begin{subfigure}{6cm}
        \includegraphics[width=\linewidth]{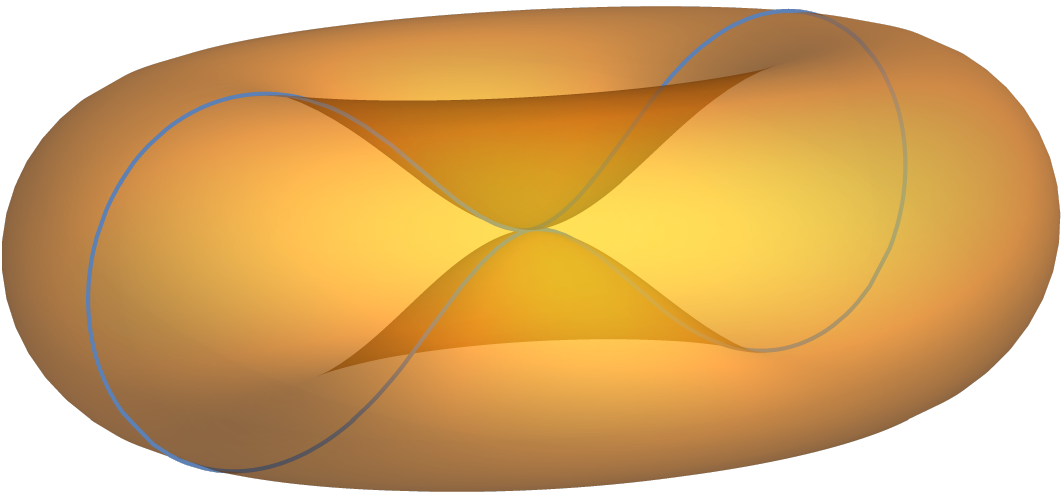}
        \vspace{0.5cm}
        \caption*{The inverted catenoid $\Sigma_{\mathrm{ic}}$.}
    \end{subfigure}
    \qquad
    \begin{subfigure}{8cm}
        \includegraphics[width=\linewidth]{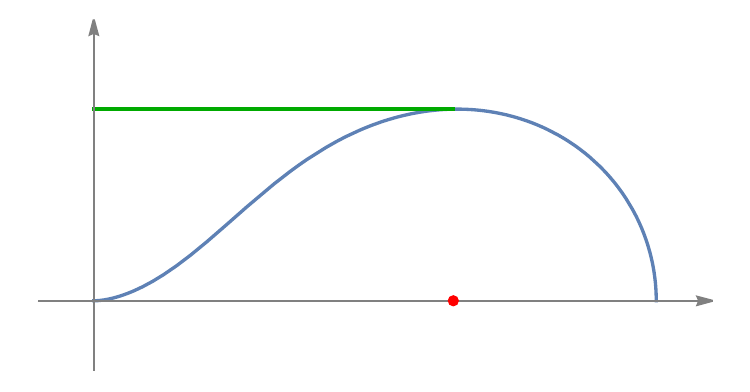}
        \caption*{The functions $u$ (green), $u_{\mathrm{ic}}$ (blue) as well as $(r^*,0)$ (red) used in the proof of \Cref{lem:wf-starting-in-ic-1}.}
    \end{subfigure}
    \caption{Illustrations for \Cref{lem:wf-starting-in-ic-1} and its proof.}
    \label{fig:inv_catenoid_and_Sigma}
\end{figure}
\begin{proof}
    Denote by $\Sigma_{\mathrm{ic}}\subseteq\R^3$ the inverted catenoid with diameter $2$ such that, for
    \begin{equation}
        R_{\varphi} \vcentcolon= \begin{pmatrix}
            \cos \varphi & -\sin\varphi &0\\
            \sin \varphi & \cos \varphi &0\\
            0 & 0 & 1
        \end{pmatrix},
    \end{equation}
    $R_{\varphi}\Sigma_{\mathrm{ic}} = \Sigma_{\mathrm{ic}}$ for all $\varphi\in\R$. In particular, here we consider rotation about the vertical axis, see \Cref{fig:inv_catenoid_and_Sigma}. By \cite[Lemma~3.4 and its proof]{chenli2017}, we have the following. There exists a function $u_{\mathrm{ic}}$ such that $B_1(0)\to \R^3$, $x\mapsto (x,u_{\mathrm{ic}}(|x|))$ parametrizes $\Sigma_{\mathrm{ic}}\cap \big(B_1(0)\times [0,\infty)\big)$. Furthermore, the function $u_{\mathrm{ic}}\colon[0,1]\to[0,\infty)$ is continuous, smooth in $(0,1)$ and in $W^{2,p}((0,\frac12),\R)$ for all $1\leq p<\infty$ with $u_{\mathrm{ic}}'(0)=0$. Moreover, the following are well-defined:
    \begin{equation}
        h\vcentcolon=\max_{(0,1)} u_{\mathrm{ic}} \quad\text{and}\quad r^* \vcentcolon= \arg\max_{(0,1)} u_{\mathrm{ic}}\in (0,1).
    \end{equation}
    With this description, the \lq singular point\rq \, of the inverted catenoid is at the origin. Notice that the points where this graph representation  is not $C^1$ are points where the inverted catenoid is actually smooth.
    
    Define $u\colon(-\infty,1)\to (0,\infty)$, $u_{\varepsilon}\colon (-\infty,1-\varepsilon)\to(0,\infty)$
    \begin{equation}
        u(r)\vcentcolon=\begin{cases}
            u_{\mathrm{ic}} ( r) &\text{for $r^*\leq r < 1$},\\
            h&\text{for $-\infty < r \leq r^*$}
        \end{cases},\quad u_{\varepsilon}(r)\vcentcolon=\int_{r-\varepsilon}^{r+\varepsilon} u(y)\varphi_{\varepsilon}(r-y)\dd y
    \end{equation}
    where $(\varphi_{\varepsilon})_{\varepsilon>0}$ are the standard mollifiers with $\mathrm{supp}(\varphi_{\varepsilon})\subseteq (-\varepsilon,\varepsilon)$, also see \Cref{fig:inv_catenoid_and_Sigma}. Notice that $u \in C^1((-\infty,1))$, which is why for any $\delta \in (0,1)$ one has $\lVert u-u_\varepsilon \rVert_{C^1([0,1-\delta])} \rightarrow 0$ as $\varepsilon \rightarrow 0$.  Consider a family of smooth cut-off functions $\eta_{\varepsilon'}\colon\R\to [0,1]$, $\varepsilon'>0$, such that
    \begin{equation}
        \chi_{(r^*+2\varepsilon',\infty)} \leq \eta_{\varepsilon'} \leq \chi_{(r^*+\varepsilon',\infty)} \quad\text{ and }\quad |\partial_r\eta_{\varepsilon'}|\leq \frac{2}{\varepsilon'}\chi_{(r^*+\varepsilon',r^*+2\varepsilon')},\, |\partial_r^2\eta_{\varepsilon'}|\leq \frac{C}{\varepsilon'^2}\chi_{(r^*+\varepsilon',r^*+2\varepsilon')}.
    \end{equation}
    For $0<\varepsilon<\varepsilon'\vcentcolon=\frac{1}{4}(1-r^*)$, consider $u_{\varepsilon,\varepsilon'}\in C^{\infty}((-\infty,1),(0,\infty))$ 
    \begin{equation}
        u_{\varepsilon,\varepsilon'}(r) \vcentcolon= (1-\eta_{\varepsilon'}(r))u_{\varepsilon}(r) + \eta_{\varepsilon'}(r) u(r).
    \end{equation}
     Finally $\Sigma_{\varepsilon,\varepsilon'}\subseteq\R^3$ denotes the smooth spherical surface one obtains by reflecting the image of $(x,u_{\varepsilon,\varepsilon'}(|x|))$
    in $\overline{B}_1(0)\times[0,\infty)$ by the plane $\R^2\times\{0\}$. This surface is embedded as we have $u_\varepsilon>0$ and smooth, using that $u_{\varepsilon,\varepsilon'}=u_{\mathrm{ic}}$ on $(1-\varepsilon',1)$ and that $\Sigma_{\mathrm{ic}}$ is smooth away from the origin. 
    Since
    \begin{align}
        &u_{\varepsilon,\varepsilon'}(r)=u_{\varepsilon}(r)=h\quad\text{for all }0\leq r<r^*-\varepsilon,\\
        &\|u_{\varepsilon}-u\|_{C^1([0,1-\varepsilon'])} \to 0\quad\text{for }\varepsilon\searrow 0,\\
        &\|u_{\varepsilon,\varepsilon'}''\|_{C^0([0,1-\varepsilon'])} \leq  \|u_{\mathrm{ic}}''\|_{C^0([r^*,1-\varepsilon'])} + C_{\varepsilon'} \|u_{\varepsilon}-u\|_{C^1([0,1-\varepsilon'])},
    \end{align}
    using the explicit formulae in \cite[Example~5 on p.~162]{docarmo1976}, there exists $R=R(\varepsilon')>0$ such that $\det(I-\rho L_{\Sigma_{\varepsilon,\varepsilon'}}(\hat x))\neq 0$ for all $\rho\in\R$ with $|\rho|<\frac{1}{R}$ and for all $\hat{x}\in \Sigma_{\varepsilon,\varepsilon'}$. Moreover, using $\|u_{\varepsilon,\varepsilon'}-u\|_{C^0([0,1])}\leq \|u_\varepsilon-u\|_{C^0([0,1-\varepsilon'])} \to 0$ for $\varepsilon\searrow 0$ and $u=u_{\mathrm{ic}}$ on $[r^*,1]$, there exists $\varepsilon>0$ sufficiently small such that
    \begin{equation}\label{eq:approx-property-of-uepseps'}
       \|u_{\mathrm{ic}}-h\|_{C^0([r^*-\varepsilon,r^*])}+\|u_{\mathrm{ic}}-u_{\varepsilon,\varepsilon'}\|_{C^0([r^*-\varepsilon,1])} < \frac{1}{2R}.
    \end{equation}
    With this choice of $\varepsilon$, let $\Sigma\vcentcolon=\Sigma_{\varepsilon,\varepsilon'}$ and observe that
    \begin{equation}\label{eq:sigma-flat-parts}
        \Sigma\cap \big( B_{r^*-\varepsilon}(0)\times\{\pm h\} \big)= B_{r^*-\varepsilon}(0)\times\{\pm h\}.
    \end{equation}
    Define 
    \begin{align}
        a_{\hat x} &\vcentcolon= \frac{1}{R} \text{ for $\hat x\in \Sigma\setminus B_{r^*-\varepsilon}(0)\times\{\pm h\}$}\quad\text{and}\\
        a_{(x,\pm h)} &\vcentcolon= \frac{1}{R} + u_{\mathrm{ic}}(r^*-\varepsilon)-u_{\mathrm{ic}}(|x|)\text{ for $x\in B_{r^*-\varepsilon}(0)$}.\label{eq:def-of-ax}
    \end{align}
    Notice that $\rho_0$ as in the claim is uniquely determined and satisfies $\rho_0(x,\pm h)=h-u_{\mathrm{ic}}(|x|)$ for $(x,\pm h)\in B_{r^*-\varepsilon}(0)\times\{\pm h\}$. In particular, \eqref{eq:approx-property-of-uepseps'}, \eqref{eq:sigma-flat-parts} and \eqref{eq:def-of-ax} yield $|\rho_0(\hat x)|<a_{\hat x}$ for all $\hat x\in\Sigma$. Moreover, $\rho_0\in W^{2,p}(\Sigma)$ follows using $u_{\mathrm{ic}}\in W^{2,p}([0,\frac12])$ and $u_{\mathrm{ic}}'(0)=0$.
\end{proof}

\begin{lemma}\label{lem:will-lsc}
    Consider a closed and oriented surface $\Sigma$ and let $f_j,f\in C^1(\Sigma,\R^n)$ be proper immersions. Suppose that $U\subseteq\R^n$ is open such that $f_j|_{f_j^{-1}(U)}$ for all $j\in\N$ and $f|_{f^{-1}(U)}$ are $C^2$. If $f_j\to f$ in $C^1$, then
    \begin{equation}
        \int_{f^{-1}(U)} H_f^2\dd\mu_f\leq \liminf_{j\to\infty} \int_{f_j^{-1}(U)} H_{f_j}^2\dd\mu_{f_j}.
    \end{equation}
\end{lemma}
\begin{proof}
    Denote by $V_{f_j}$ and $V_f$ the rectifiable varifolds induced by $f_j$ and $f$, respectively (see \cite[p.~79]{SimonLectures1983}). Note that $V_{f_j}$ and $V_f$ have (locally) bounded first variation when restricted to $U$ with generalized mean curvature $\vec H_{V_{f_j}}, \vec H_{V_f}$. For $X\in C_c^1(U,\R^3)$ we obtain
    \begin{align}
        -&\int_{\R^n}\langle X(x),\vec H_{V_{f_j}}\rangle \dd \|V_{f_j}\| = \delta V_j(X) = \int_{\Sigma} (\mathrm{div}_{df_j(T_p\Sigma)}X)(f_j(p))\dd\mu_{f_j}(p)\\
        &\to \int_{\Sigma} (\mathrm{div}_{df(T_p\Sigma)}X)(f(p))\dd\mu_{f}(p) = \delta V(X) = -\int_{\R^n}\langle X(x),\vec H_{V_{f}}\rangle \dd \|V_f\|,
    \end{align}
    using $f_{\#}\mu_f=\|V_f\|$ and $f_{\#}\mu_{f_j}=\|V_{f_j}\|$ 
    (see \cite[15.7 (and p.~69)]{SimonLectures1983}) 
    and the $C^1$ convergence $f_j\to f$. Therefore, 
    \begin{equation}
        \int_U \|\vec H_{V_{f}}\|^2\dd\|V_{f}\|\leq \liminf_{j\to\infty} \int_U \|\vec H_{V_{f_j}}\|^2\dd\|V_{f_j}\|.
    \end{equation}
    Since $\vec H_{V_f}\circ f=\Delta_{g_f}f=\vec H_f$ $\mu_f$-a.e.\ on $f^{-1}(U)$ and $\vec H_{V_{f_j}}\circ f_j=\Delta_{g_{f_j}}f_j=\vec H_{f_j}$ $\mu_{f_j}$-a.e.\ on $f_j^{-1}(U)$ by \cite[Theorem~4.1]{schaetzle2009}, the claim follows after again using $f_{\#}\mu_f=\|V_f\|$ and $f_{\#}\mu_{f_j}=\|V_{f_j}\|$.
\end{proof}

\begin{lemma}\label{lem:wf-starting-in-ic-2}
    Denote by $\rho\colon[0,T)\times \Sigma\to \R$ the solution in \Cref{thm:lecroneshaosimonett} with $\rho_0$ and $\Sigma$ as constructed in \Cref{lem:wf-starting-in-ic-1}. Writing $f(t,x)=x+\rho(t,x)\nu_{\Sigma}(x)$, we have
    \begin{equation}
        \W(f(t,\cdot))\leq 8\pi\quad\text{for all }0\leq t<T.
    \end{equation}
\end{lemma}
\begin{proof}
    Since $\rho_0\in W^{2,p}(\Sigma)$ for all $p\in[1,\infty)$, $|\rho_0(x)|< a_x$ for all $x \in \Sigma$, and since $x\mapsto a_x$ is continuous, there exists a sequence $(\rho^{j}_0)_{j\in\N}\subseteq C^{\infty}(\Sigma)$ converging to $\rho_0$ in $W^{2,p}(\Sigma)$ for all $1\leq p<\infty$ with $|\rho_0^j(x)|<a_x$. Especially, $\|\rho_0^j-\rho_0\|_{C^{1,\alpha}(\Sigma)}\to 0$ for any $\alpha \in (0,1)$.
    By \Cref{thm:lecroneshaosimonett}, after passing to a subsequence, $\rho(t,\rho_0^j)$ exists for $0\leq t<T$, $j\in\N$, and 
    \begin{equation}\label{eq:ulm5:1}
        \|\rho(t,\rho_0^j)-\rho(t)\|_{C^{1,\alpha}(\Sigma)}
         \to 0\quad\text{for $j\to\infty$ for all $0\leq t<T$}.
    \end{equation}
    Furthermore, since $\rho_0^j$ is smooth, \cite{simonett2001} especially yields that $\rho(t,\rho_0^j)\to \rho_0^j$ in $C^{2,\alpha}(\Sigma)$ as $t\searrow 0$, so that, writing $f_j(t,x)\vcentcolon=x+\rho((t,x),\rho_0^j)\nu_{\Sigma}(x)$, for all $t\in [0,T)$, $j\in\N$, we have
    \begin{equation}\label{eq:cat-en-dec-1}
        \W(f_j(t,\cdot))\leq \lim_{s\searrow 0 } \W(f_j(s,\cdot)) = \W(f_j(0,\cdot)).
    \end{equation}
    Moreover, since $\rho_0^j\to \rho_0$ in $W^{2,p}(\Sigma)$ for $p>2$, 
    \begin{equation}\label{eq:cat-en-dec-2}
        \W(f_j(0,\cdot)) \to \W(f_0) = 8\pi.
    \end{equation}    
   Consider any $t\in (0,T)$. Since $f_j(t,\cdot)\to f(t,\cdot)$ in $C^{1,\alpha}(\Sigma)$ by \eqref{eq:ulm5:1}
    and as all immersions are smooth, \Cref{lem:will-lsc} yields
    \begin{equation}
        \W(f(t,\cdot)) \leq \liminf_{j\to\infty} \W(f_j(t,\cdot)) \leq \liminf_{j\to\infty} \W(f_j(0,\cdot)),
    \end{equation}
    using \eqref{eq:cat-en-dec-1} in the last estimate. Now the claim follows from $\eqref{eq:cat-en-dec-2}$.
\end{proof}

\begin{proof}[Proof of \Cref{thm:intro:wf-starting-in-ic}]
    We work in the setting of \Cref{lem:wf-starting-in-ic-1,lem:wf-starting-in-ic-2}. In particular, $f_0$ is given by \eqref{eq:lem-constr-Sigma-rho0} and we consider the solution $\rho\colon[0,T)\times\Sigma\to\R$ of \eqref{eq:wf-in-rho} starting at $\rho_0$ given by \Cref{thm:lecroneshaosimonett}. Write $f(t,x)\vcentcolon=x+\rho(t,x)\nu_{\Sigma}(x)$ for $0\leq t<T$ and $x\in\Sigma$. 
    Let $\delta\in(0,T)$. Since $\W(f(\delta))\leq 8\pi$ and since $f(\delta)$ is a smooth spherical immersion, the Willmore flow of \cite[Theorem~5.2]{kuwertschaetzle2004} starting at $f(\delta)$ exists globally and converges to a round sphere up to reparametrization. A gluing argument yields global existence and convergence, see \Cref{rem:rep-of-tangential-Wf} and \cite[Appendix~D]{ruppspener2020}.

    We show next continuity of the Willmore energy for $t\searrow 0$. For \eqref{eq:wf-starting-in-ic-en-cont}, using $f(t,\cdot)\to f_0$ in $C^{1,\alpha}(\S^2)$ and \Cref{lem:will-lsc} with $U=\R^3\setminus\{0\}$,
    \begin{equation}
        8\pi=\W(f_0|_{f_0^{-1}(U)})\leq \liminf_{t\searrow 0}\W(f(t,\cdot))\leq \limsup_{t\searrow 0}\W(f(t,\cdot)) \leq 8\pi,
    \end{equation}
    using \Cref{lem:wf-starting-in-ic-2} in the last estimate. 
    Finally, we argue that  $\mathcal{W}(f(t))< 8\pi$ for all $t> 0$. Fix $\delta>0$. Since $f(\delta)$ is smooth and $\mathcal{W}(f(\delta))\leq 8\pi$ we conclude from \cite[Theorem~5.2]{kuwertschaetzle2004} that the Willmore flow of $f(\delta)$ is not stationary and hence $f(\delta)$ is not a Willmore immersion (unless $f(\delta)$ is the round sphere, in which case $4\pi=\mathcal{W}(f(\delta))< 8\pi$). Therefore $\mathcal{W}(f(2\delta))< 8\pi$.
\end{proof}

\appendix

\section{Technical proofs}\label{app:techproofs}

\begin{proof}[Proof of \Cref{lem:allesOK}]
For $m=0$ the statement is clear. We prove first the statement for $m=1$ and $\beta=0$. We consider first a (normal) vector field appearing in $ \mathbb{P}^{A,C}_B$. Since we are in codimension one, this can be written as 
$$ \frac{ \partial_s \gamma}{\gamma^{(2)}} * \cdots * \frac{ \partial_s \gamma}{\gamma^{(2)}} * \frac{(\nabla_{s}^\bot)^{i_1} \vec{\kappa}}{\gamma^{(2)}}  * \cdots * \frac{(\nabla_{s}^\bot)^{i_{b-1}}  \vec{\kappa}   }{\gamma^{(2)}} \  \frac{ (\nabla_{s}^\bot)^{i_{b}}  \vec{\kappa} }{\gamma^{(2)}} =\vcentcolon p(s) \frac{(\nabla_{s}^\bot)^{i_{b}}  \vec{\kappa}}{\gamma^{(2)}} ,$$
with $p(s)$ of the type $P^{a-i_{b},c}_{b-1}$, where $a+\frac12 b\leq A +\frac12 B$. Then we compute
$$ \nabla_s^{\bot} \Big(p(s) \frac{(\nabla_{s}^\bot)^{i_{b}}  \vec{\kappa}}{\gamma^{(2)}}\Big)  = p(s) \frac{(\nabla_{s}^\bot)^{i_{b}+1}  \vec{\kappa}}{\gamma^{(2)}}  + p(s)\ \frac{\partial_s\gamma}{\gamma^{(2)}}*\frac{(\nabla_s^{\bot})^{i_b}\Vec{\kappa}}{\gamma^{(2)}} + (\partial_s p(s) ) \frac{(\nabla_s^{\bot})^{i_b}\Vec{\kappa}}{\gamma^{(2)}}\, ,$$
and it is sufficient to prove that $\partial_s p =  \mathbb{P}^{a-i_b+1,c+1}_{b-1}$.  It is useful to note that from \eqref{eq:NablaLocalH} it follows that for a vector field $\Phi$ along $\gamma$ we have
$ \partial_s \Phi = \nabla_s \Phi + \Phi * \frac{\partial_s \gamma}{\gamma^{(2)}}.$
Then, 
\begin{equation} \label{eq:secondderivativesinPnotation} \partial_s \big( \frac{\partial_s \gamma}{\gamma^{(2)}} \big) = \frac{\partial_s (\partial_s \gamma)}{\gamma^{(2)}}  - \frac{\partial_s \gamma}{(\gamma^{(2)})^2} \partial_s \gamma^{(2)} = \frac{\vec{\kappa}}{\gamma^{(2)}}  + \frac{\partial_s \gamma}{\gamma^{(2)}} * \frac{\partial_s \gamma}{\gamma^{(2)}} = \mathbb{P}^{0,0}_{1}\, , \end{equation}
\begin{align} \partial_s \big( \frac{(\nabla_{s}^\bot)^{i_{j}}  \vec{\kappa}}{\gamma^{(2)}} \big)& = \frac{\partial_s (\nabla_{s}^\bot)^{i_{j}}  \vec{\kappa}}{\gamma^{(2)}}  - \frac{(\nabla_{s}^\bot)^{i_{j}}  \vec{\kappa}}{(\gamma^{(2)})^2} \partial_s \gamma^{(2)}
= \frac{\nabla_s (\nabla_{s}^\bot)^{i_{j}}  \vec{\kappa}}{\gamma^{(2)}}  + \frac{(\nabla_{s}^\bot)^{i_{j}}  \vec{\kappa}}{\gamma^{(2)}} * \frac{\partial_s \gamma}{\gamma^{(2)}} \\
& =  \frac{(\nabla_{s}^\bot)^{i_{j}+1}  \vec{\kappa}}{\gamma^{(2)}}   + \frac{(\nabla_{s}^\bot)^{i_{j}}  \vec{\kappa}}{\gamma^{(2)}} * \frac{\vec{\kappa}}{\gamma^{(2)}} *\frac{\partial_s \gamma}{\gamma^{(2)}} + \frac{(\nabla_{s}^\bot)^{i_{j}}  \vec{\kappa}}{\gamma^{(2)}} * \frac{\partial_s \gamma}{\gamma^{(2)}}
= \mathbb{P}^{i_j+1,i_j+1}_{1}\, , 
\end{align}
using \eqref{eq:NormalDerOfNormal}. By the differentiation rule for products it follows that $\partial_s p$ is of the type $\mathbb{P}^{a-i_{b}+1,c+1}_{b-1}$. 
The same computations yield the same result for functions appearing in $\mathbb{P}^{A,C}_B$. 
By induction the claim is true for all $m \in \N$ and $\beta=0$. 

Now we consider the case $\beta \ne 0$. For $m=1$ we find
 $$ \nabla_s^{\bot} ( (\gamma^{(2)})^{\beta}\, \mathbb{P}^{A,C}_B ) = (\gamma^{(2)})^{\beta}\, \mathbb{P}^{A+1,C+1}_B +  (\gamma^{(2)})^{\beta} \frac{\partial_s \gamma}{\gamma^{(2)}} *\mathbb{P}^{A,C}_B, $$
 and the claim is proven for $m=1, \beta\in \Z\setminus\{0\}$. The case $m\geq 2$ follows by induction.
 \end{proof}

\begin{lemma}\label{lem:evolkappa}
    Consider a smooth family of immersions satisfying
    \eqref{eq:evol-eq-with-general-lambda}. Then we have
    \begin{equation}
        \nabla_t^{\bot} (\nabla_s^{\bot})^{m} \vec{\kappa} - (\nabla_s^{\bot})^{m}\nabla_t^{\bot}  \vec{\kappa} = \frac{1}{(\gamma^{(2)})^3} \mathbb{P}_{3}^{m+2,m+1} +  \frac{\lambda}{(\gamma^{(2)})^3} \mathbb{P}_{3}^{m,m}
    \end{equation}
    for all $m\geq 1$. Moreover,
    \begin{align}
        \nabla_t^{\bot} (\nabla_s^{\bot})^{m} \vec{\kappa} = & -\frac{1}{2 (\gamma^{(2)})^4} (\nabla_s^{\bot})^{m+4} \vec{\kappa} + (2m +4) \frac{\partial_s \gamma^{(2)}}{ (\gamma^{(2)})^5} (\nabla_s^{\bot})^{m+3} \vec{\kappa} +  \frac{1}{(\gamma^{(2)})^3} \mathbb{P}_{3}^{m+2,m+2} \\
        & + \lambda\Big(\frac{1}{4(\gamma^{(2)})^4}(\nabla_s^{\bot})^{m+2}\curv - (m+2)\frac{\partial_s\gamma^{(2)}}{(\gamma^{(2)})^5}(\nabla_s^{\bot})^{m+1}\curv + \frac{1}{(\gamma^{(2)})^3}\mathbb{P}_3^{m,m}\Big).
    \end{align} 
\end{lemma}
\begin{proof}
    In the following $V$ denotes the right hand side in \eqref{eq:evol-eq-with-general-lambda}. Note that by \eqref{eq:nabla-L2}
    \begin{equation}\label{eq:V-with-lambda}
        V=\partial_t \gamma= - \frac{1}{2(\gamma^{(2)})^4} (\nabla_s^{\perp})^2 \vec{\kappa} + \frac{1}{(\gamma^{(2)})^3} \mathbb{P}_3^{0,0} + \frac{\lambda}{4(\gamma^{(2)})^4}\curv= \frac{1}{(\gamma^{(2)})^3}\mathbb{P}_1^{2,2} + \frac{\lambda}{(\gamma^{(2)})^3}\mathbb{P}_1^{0,0}.
    \end{equation}
    Using \eqref{eq:Commutator_Normal_Normal_S}, we find
    \begin{align}
        \nabla_t^{\bot} (\nabla_s^{\bot})^{m} \vec{\kappa} &- (\nabla_s^{\bot})^{m}\nabla_t^{\bot}  \vec{\kappa}  = \sum_{i=0}^{m-1} (\nabla_s^{\bot})^i (\nabla_t^{\bot} \nabla_s^{\bot}  - \nabla_s^{\bot}\nabla_t^{\bot})   (\nabla_s^{\bot})^{m-1-i}\vec{\kappa} \\
        & = \sum_{i=0}^{m-1} (\nabla_s^{\bot})^i (\langle V,\vec{\kappa}\rangle_g   (\nabla_s^{\bot})^{m-i}\vec{\kappa}) \\
        & = \sum_{i=0}^{m-1} (\nabla_s^{\bot})^i      
        \Big(\frac{1}{(\gamma^{(2)})^3}
        \mathbb{P}_{3}^{m+2-i,\max\{m-i,2\}}+\frac{\lambda}{(\gamma^{(2)})^3}\mathbb{P}_3^{m-i,m-i}\Big) \\
        & = \frac{1}{(\gamma^{(2)})^3}
        \mathbb{P}_{3}^{m+2,m+1}+ \frac{\lambda}{(\gamma^{(2)})^3}\mathbb{P}_3^{m,m}
   \end{align}
    where we used \Cref{lem:allesOK} and $\max\{m-i,2\}+i\leq m+1$ for $0\leq i\leq m-1$. This yields the first part of the claim since $m \geq 1$. For the second equation, using \eqref{eq:EvoKappa} and \Cref{lem:allesOK}  
    a computation based on the Leibniz rule for differentiating products yields
    \begin{align}
        (&\nabla_s^{\bot})^{m}\nabla_t^{\bot}   \vec{\kappa}  = (\nabla_s^{\bot})^{m} \Big[-(\nabla_{s}^\bot)^2 \Big(\frac{1}{2 (\gamma^{(2)})^4} (\nabla_s^{\bot})^{2} \vec{\kappa} \Big) +   \frac{1}{(\gamma^{(2)})^3}
        \mathbb{P}_{3}^{2,2}  \\
        &\qquad + \lambda\Big((\nabla_s^{\bot})^2\Big(\frac{1}{4(\gamma^{(2)})^4}\curv\Big)+\frac{1}{(\gamma^{(2)})^3}\mathbb{P}_3^{0,0}\Big)\Big]\\
        & = (\nabla_s^{\bot})^{m} \Big[- \frac{1}{2 (\gamma^{(2)})^4} (\nabla_s^{\bot})^{4} \vec{\kappa} +4 \frac{1}{ (\gamma^{(2)})^5}  \partial_s \gamma^{(2)}(\nabla_s^{\bot})^{3} \vec{\kappa}+   \frac{1}{(\gamma^{(2)})^3}
        \mathbb{P}_{3}^{2,2} \Big]\\
        &\qquad + \lambda \Big((\nabla_s^{\bot})^{m+2}\Big(\frac{1}{4(\gamma^{(2)})^4}\curv\Big)+\frac{1}{(\gamma^{(2)})^3}\mathbb{P}_3^{m,m}\Big)\\
        & = - \frac{1}{2 (\gamma^{(2)})^4} (\nabla_s^{\bot})^{m+4} \vec{\kappa} +(2m+4) \frac{1}{ (\gamma^{(2)})^5}  \partial_s \gamma^{(2)}(\nabla_s^{\bot})^{m+3} \vec{\kappa}  +   \frac{1}{(\gamma^{(2)})^3}
        \mathbb{P}_{3}^{m+2,m+2}\\
        &\qquad + \lambda\Big(\frac{1}{4(\gamma^{(2)})^4}(\nabla_s^{\bot})^{m+2}\curv - (m+2)\frac{1}{(\gamma^{(2)})^5}\partial_s\gamma^{(2)}(\nabla_s^{\bot})^{m+1}\curv + \frac{1}{(\gamma^{(2)})^3}\mathbb{P}_3^{m,m}\Big).
    \end{align}
    This combined with the first equation gives the second equality.
\end{proof}

\begin{proof}[Proof of \Cref{lem:int-est-general-lambda}]
    We claim that 
    \begin{align}
        \frac{\dd}{\dd t} \int_{\S^1} (\gamma^{(2)})^4&  | (\nabla_s^{\perp})^m \vec{\kappa}|^2_g \dd s  + \frac{\mu}{r^4}\int_{\S^1} (\gamma^{(2)})^4 | (\nabla_s^{\perp})^m \vec{\kappa}|^2_g \dd s   + \int_{\S^1}  | (\nabla_s^{\perp})^{m+2} \vec{\kappa}|^2_g \dd s \\
        & \leq \int_{\S^1} (1+ (\gamma^{(2)})^4\frac{\mu}{r^4} ) |\mathbb{P}^{2m+3,m+2}_2| \dd s  \label{eq:EvolutionOfweightedNorm}\\
        &\quad + |\lambda| \Big( \int_{\S^1}|(\nabla_s^{\bot})^{m+1}\curv|_g^2\dd s + \int_{\S^1} |\mathbb{P}_4^{2m,m}|\dd s \Big). \label{eq:EvolutionOfweightedNorm-lambda-terms}
    \end{align}
    To this end we compute (again with $V$ being the right hand side in \eqref{eq:evol-eq-with-general-lambda})
    \begin{align}
 & \frac{\dd}{\dd t} \int_{\mathbb{S}^1} (\gamma^{(2)})^4 | (\nabla_s^\perp)^m \vec{\kappa} |_g^2 \; \dd s  \\ & = \int_{\mathbb{S}^1} 4 (\gamma^{(2)})^3 V^{(2)} |(\nabla_s^\perp)^m \vec{\kappa}|_g^2 \dd s +2  \int_{\mathbb{S}^1}  (\gamma^{(2)})^4 \langle (\nabla_s^\perp)^m \vec{\kappa} , \nabla_t^\perp (\nabla_s^\perp)^m \vec{\kappa} \rangle_g \dd s  \\ & \quad  - \int_{\mathbb{S}^1} (\gamma^{(2)})^4 |(\nabla_s^\perp)^m \vec{\kappa}|_g^2 \langle V,  \vec{\kappa} \rangle_g \dd s.
    \end{align}
    Using the last equality in \eqref{eq:V-with-lambda}, one readily checks that the first and the third term can be absorbed in the term in \eqref{eq:EvolutionOfweightedNorm} and in the second summand in \eqref{eq:EvolutionOfweightedNorm-lambda-terms}. In the second integral we use \Cref{lem:evolkappa} to rewrite it as 
    \begin{align}
 & - \int_{\mathbb{S}^1} \langle (\nabla_s^\perp)^m \vec{\kappa} , (\nabla_s^\perp)^{m+4} \vec{\kappa}  \rangle_g  \dd s  + 2(2m+4) \int_{\mathbb{S}^1} \frac{\partial_s\gamma^{(2)}}{\gamma^{(2)}} \langle (\nabla_s^\perp)^m \vec{\kappa}, (\nabla_s^\perp)^{m+3} \vec{\kappa} \rangle_g \dd s  \\ & + \int_{\mathbb{S}^1} \gamma^{(2)} \langle (\nabla_s^\perp)^m \vec{\kappa}, \mathbb{P}_{3}^{m+2,m+2} \rangle_g \dd s \label{eq:explainMiddleTerm}\\
 &+ \lambda\Big( \frac12\int_{\S^1} \langle (\nabla_s^{\bot})^m\curv,(\nabla_s^{\bot})^{m+2}\curv\rangle_g \dd s - 2(m+2)\int_{\S^1}\frac{\partial_s\gamma^{(2)}}{\gamma^{(2)}}\langle(\nabla_s^{\bot})^m\curv,(\nabla_s^{\bot})^{m+1}\curv\rangle_g\dd s \\
 &\qquad + \int_{\S^1} \gamma^{(2)}\langle (\nabla_s^{\bot})^m\curv,\mathbb{P}_3^{m,m}\rangle_g\dd s\bigg). 
    \end{align}
    Integrating by parts in the first term in the penultimate line and using Young's inequality on the second term in the penultimate line yields terms that are contained in the $\lambda$-terms in \eqref{eq:EvolutionOfweightedNorm-lambda-terms}.
    
    Integrating by parts twice in the first term in the first line of
    \eqref{eq:explainMiddleTerm}, one obtains the integral 
    $- \int_{\mathbb{S}^1} |(\nabla_s^\perp)^{m+2} \vec{\kappa}|_g^2 \dd s$
    which is to be found on the left hand side of \eqref{eq:EvolutionOfweightedNorm} with the opposite sign. Using that $\langle \cdot , \cdot \rangle_g = \frac{1}{(\gamma^{(2)})^2} \langle\cdot,\cdot\rangle$ we obtain that the integrand in the second line of \eqref{eq:explainMiddleTerm} is of the form $\frac{(\nabla_s^\perp)^m \vec{\kappa}}{\gamma^{(2)}} * \mathbb{P}_3^{m+2,m+2}$, yielding a term in $\mathbb{P}_4^{2m+2,m+2}$ (i.e., also in $\mathbb{P}_2^{2m+3,m+2}$). Next we integrate by parts in the second term of the first line of \eqref{eq:explainMiddleTerm} and obtain (without the prefactor $2(2m+4)$)  
    \begin{equation}
- \int_{\mathbb{S}^1} \frac{\partial_s \gamma^{(2)}}{\gamma^{(2)}} \langle (\nabla_s^\perp)^{m+1} \vec{\kappa}, (\nabla_s^\perp)^{m+2} \vec{\kappa} \rangle_g \dd s - \int_{\mathbb{S}^1}
 \partial_s \big( \frac{\partial_s \gamma^{(2)}}{\gamma^{(2)}} \big) \langle (\nabla_s^\perp)^{m} \vec{\kappa}, (\nabla_s^\perp)^{m+2} \vec{\kappa}   \rangle_g . \end{equation}
 The first term now clearly lies in $\mathbb{P}_3^{2m+3,m+2}$. The same follows for the second term after an application of \eqref{eq:secondderivativesinPnotation}. \eqref{eq:EvolutionOfweightedNorm} follows now readily from all the previous computations.
    
 We may now use \Cref{cor:inter} with $k = m+2$ and $\gamma^{(2)}/r \leq 1$ to estimate the first term on the right hand side of \eqref{eq:EvolutionOfweightedNorm} by $\frac{1}{2} \int_{\mathbb{S}^1} | (\nabla_{s}^\bot)^{m+2} \vec{\kappa}  |^2_{g} \dd s  + C(m,M,\mu)$. The claim follows after absorbing.
\end{proof}

\begin{proof}[Proof of \Cref{prop:il-conf-constr-wf}]
    To prove the claim, we must suitably absorb the $\lambda(\gamma)$-terms in \eqref{eq:int-est-general-lambda}. For the first term, we integrate by parts and estimate
    \begin{align}
        |\lambda(\gamma)| \int_{\S^1} |(\nabla_s^\bot)^{m+1}\curv|_g^2\dd s &= - |\lambda(\gamma)| \int_{\S^1} \langle (\nabla_s^\bot)^{m+2}\curv,(\nabla_s^{\bot})^m\curv\rangle_g\dd s \\
        &\leq \frac18 \int_{\S^1} |(\nabla_s^\bot)^{m+2}\curv|_g^2\dd s + 2(\lambda(\gamma))^2 \int_{\S^1} |(\nabla_s^{\bot})^m\curv|_g^2\dd s
    \end{align}
    and the first term can be absorbed on the left hand side of \eqref{eq:int-est-general-lambda}.\\
    For the second $\lambda(\gamma)$-term in \eqref{eq:int-est-general-lambda}, we first estimate $|\lambda(\gamma)|$ as follows.
    For the denominator in \eqref{eq:lambda-conf-constr-will}, at any time $t\in [0,T)$,
    \begin{equation}
        \int_{\S^1}\frac{|\curv|_g^2}{(\gamma^{(2)})^4}\dd s\Big|_t \geq \frac{1}{(\max_{\S^1}\gamma^{(2)}(t))^4} \E(\gamma)\geq \frac{1}{(\max_{\S^1}\gamma^{(2)}(t))^4} 4\pi,
    \end{equation}
    using \Cref{rem:min-e} in the last estimate. Moreover, integrating by parts in the numerator of \eqref{eq:lambda-conf-constr-will} and using \eqref{eq:nabla-L2}, \Cref{rem:hyplengthestimatesrange(optimal)}, and $\mathcal{L}(\gamma(t))=L_0$, we thus obtain 
    \begin{align}
        |\lambda(\gamma(t))| &\leq \frac{1}{2\pi} \int_{\S^1}\Big(\frac{\max_{\S^1}\gamma^{(2)}(t)}{\gamma^{(2)}}\Big)^4\big(|\nabla_s^{\bot}\curv|_g^2 + 4 |\nabla_s^{\bot}\curv|_g|\curv|_g +  \frac12|\curv|_g^4\big) \dd s  + 2 \\
        &\leq \frac{e^{ 2 L_0}}{2\pi} \int_{\S^1} \big(3|\nabla_s^{\bot}\curv|_g^2 + 2|\curv|_g^2 + \frac12|\curv|_g^4\big) \dd s  + 2\leq C(L_0,M) \big( 1 + \int_{\S^1} \big(|\nabla_s^{\bot}\curv|_g^2 + |\curv|_g^4\big) \dd s \big).\label{eq:lambda-trivial-estimate}
    \end{align}
    Using \eqref{eq:mult-inter-ineq} with $(a,b,k)=(2,2,m+2)$ and $(a,b,k)=(0,4,m+2)$,
    \begin{align}
        \int_{\S^1} (|\nabla_s^\bot\curv|_g^2+|\curv|_g^4)\dd s  \leq C(m,M) (1 + \||(\nabla_s^\bot)^{m+2}\curv|_g\|_{L^2(\dd s)}^{\frac{2}{m+2}})
    \end{align}
    so that, by \eqref{eq:lambda-trivial-estimate},
    \begin{align}
        |\lambda(\gamma)|\leq C(m,L_0,M)(1 + \||(\nabla_s^\bot)^{m+2}\curv|_g\|_{L^2(\dd s)}^{ \frac{2}{m+2}}).\label{eq:subcritical-term-2}
    \end{align}
    Any term $P_{b_j}^{a_j,c_j}$ in $\mathbb P_{4}^{2m,m}$ in the last line of \eqref{eq:int-est-general-lambda} satisfies $a_j+\frac12b_j \leq 2m+2$, 
    hence, using \eqref{eq:mult-inter-ineq},
    \begin{align}
        \int_{\S^1} | P_{b_j}^{a_j,c_j}|\dd s\leq C(m,M)  (1+\||(\nabla_s^{\bot})^{m+2}\curv|_g\|_{L^2(\dd s)}^{\frac{2m+1}{m+2}})\label{eq:subcritical-term-1}.
    \end{align}
    Combining \eqref{eq:subcritical-term-2} and \eqref{eq:subcritical-term-1},
    \begin{equation}
        |\lambda(\gamma)|\int_{\S^1}|P_{b_j}^{a_j,c_j}|\dd s\leq C(m,L_0,M) \big(1+\||(\nabla_s^{\bot})^{m+2}\curv|_g\|_{L^2(\dd s)}^{\frac{2m+3}{m+2}}\big).
    \end{equation}
    Altogether, since $\frac{2m+3}{m+2}<2$, by Young's inequality,
    \begin{equation}
        |\lambda(\gamma)|\int_{\S^1} |\mathbb P_4^{2m,m}|\dd s \leq \frac18 \int_{\S^1} |(\nabla_s^\bot)^{m+2}\curv|_g^2\dd s + C(m,L_0,M).
    \end{equation}
    Again, we can absorb the first term on the left hand side of \eqref{eq:int-est-general-lambda}.
\end{proof}

\section{Parabolic theory: Well-posedness of the Willmore flow}\label{sec:app:well-posedness}

Throughout this section, $\Sigma\subseteq\R^3$ denotes a closed smooth embedded surface with inward-pointing unit normal $\nu_{\Sigma}$, $L_{\Sigma}=-d\nu_{\Sigma}$ denotes the associated shape operator, and we fix a continuous mapping $a\colon\Sigma\to(0,\infty)$, $x\mapsto a_x$ such that 
\begin{equation}\label{eq:conddiff}
\det(I-\rho L_{\Sigma}(x))\neq 0  \text{ for all $\rho\in\R$ with } |\rho|<a_x,  
\end{equation}
and for all $x\in\Sigma$. 
Moreover, for $k\in\N_0$ and $\alpha\in(0,1)$, $bc^{k+\alpha}(\Sigma)$ denotes the little-Hölder space of order $k+\alpha$, i.e., the closure of
$C^\infty(\Sigma)$ in $\|\cdot\|_{C^{k,\alpha}}$, see \cite[Section 2.2]{lecroneshaosimonett2020}.

Consider any $\rho_0\in bc^{1+\alpha}(\Sigma)$ such that $|\rho_0(x)|<a_x$ for all $x\in\Sigma$. 
Then, $f_0\colon\Sigma\to\R^3$ with $f_0(x)=x+\rho_0(x)\nu_h(x)$ is an immersion by \eqref{eq:conddiff}. In addition, the computations in \cite[Section~5]{lecroneshaosimonett2020} apply since they are purely local. That is, a family $f\colon[0,T)\times\Sigma\to\R^3$, $f(t,x)=x+\rho(t,x)\nu_{\Sigma}(x)$ with $|\rho(t,x)|<a_x$ satisfies
\begin{equation}
    \begin{cases}
        (\partial_tf)^{\bot} = - \nabla\W(f)&\text{on $[0,T)\times\Sigma$}\\
        f(0)=f_0&\text{on $\Sigma$}
    \end{cases}
\end{equation}
if and only if
\begin{equation}\label{eq:wf-in-rho}
    \begin{cases}
        \partial_t\rho = -\frac{1}{\beta(\rho)} (\Delta_{\rho} H_{\rho} + 2H_{\rho}(H_{\rho}^2-K_{\rho}))    &\text{on $[0,T)\times\Sigma$},\\
        \rho(0)=\rho_0&\text{on $\Sigma$},
    \end{cases}
\end{equation}
using the notation of \cite{lecroneshaosimonett2020}. 

Again, the computations in \cite[Sections~4 and 5]{lecroneshaosimonett2020} are purely local and do not require $f_0$ to be in a tubular neighborhood of $\Sigma$ with a fixed height, as long as the estimate $|\rho(t,x)|<a_x$ is satisfied. Therefore, one can still apply \cite[Theorem~2.4]{lecroneshaosimonett2020} to get the following analog of \cite[Theorem~5.1]{lecroneshaosimonett2020}, see \cite[Section 2.3]{lecroneshaosimonett2020} for the function space notation.

\begin{theorem}\label{thm:lecroneshaosimonett}
    Let $\alpha\in(0,1)$ and $\rho_0\in bc^{1+\alpha}(\Sigma)$ with $|\rho_0(x)|<a_x$ for all $x \in \Sigma$. Then there exists $T_+=T_+(\rho_0)>0$ and a unique maximal solution
        \begin{equation}
            \rho(\cdot,\rho_0)\in C_{3/4} ([0,T_+),bc^{4+\alpha}(\Sigma))\cap C^{1}_{3/4}([0,T_+),bc^{\alpha}(\Sigma)
        \end{equation}
        to \eqref{eq:wf-in-rho}.
        Moreover, $\rho(\cdot,\rho_0)\in C([0,T_+),bc^{1+\alpha}(\Sigma))$. Furthermore, for $T\in (0,T_+)$ there exists $\varepsilon=\varepsilon(\rho_0,T)>0$ such that, for all $\rho^1_0,\rho^2_0\in bc^{1+\alpha}(\Sigma)$ with $\|\rho^i_0-\rho_0\|_{C^{1,\alpha}(\Sigma)}<\varepsilon$ for $i=1,2$, we have $T_+(\rho_0^i)\geq T$ and 
        \begin{equation}
            \|\rho(t,\rho^1_0)-\rho(s,\rho^2_0)\|_{C^{1,\alpha}(\Sigma)} \leq \sigma \big(\|\rho^1_0-\rho^2_0\|_{C^{1,\alpha}(\Sigma)}+|t-s|\big)\quad\text{for all } 0\leq t,s<T,
        \end{equation}
        where $\sigma=\sigma(\rho_0,T)>0$. 
\end{theorem}

\section{Global existence and convergence of the Willmore flow}
\label{sec:conv-unconstr-wf}

In this section, we discuss the analog of \Cref{thm:main-conf-constr-wf} for the (unconstrained) Willmore flow \eqref{eq:wf-eq}.

\begin{theorem}\label{thm:main-1d}
    Let $\gamma\colon[0,T)\times\S^1\to\H^2$ be a maximal Willmore flow with $\mathcal{L}(\gamma(t))\leq L<\infty$ for all $0\leq t<T$. Then $T=\infty$ and, if $\widetilde{\gamma}(t)$ is the reparametrization of $\gamma(t)$ by constant hyperbolic speed for each $0\leq t<T$, then
    \begin{equation}
        \widetilde{\gamma}(t) \to \hat{\gamma} \quad\text{in the $C^{\infty}(\S^1)$ topology}
    \end{equation}
    as $t\to\infty$ where $\hat{\gamma}$ is the profile curve of a Willmore torus of revolution.
\end{theorem}

In contrast to \Cref{sec:conv-conf-constr-will}, for the Willmore flow, we use an \emph{unconstrained} version of the \L ojasiewicz--Simon gradient inequality for the elastic energy in $\H^2$. By \cite[Corollary 3.27]{pozzetta2022}, we have the following.

\begin{theorem}\label{thm:loja-e}
    Let $\hat{\gamma}\colon\S^1\to\H^2$ be a smooth critical point of $\E$. Then there exist $C=C(\hat{\gamma})>0$, $\theta\in(0,\frac12]$ and $\sigma>0$ such that, for all immersions $\gamma\colon\S^1\to\H^2$ such that there exists a diffeomorphism $\phi\colon\S^1\to\S^1$ satisfying $\|\hat{\gamma}-\gamma\circ \phi\|_{W^{4,2}(\S^1,\R^2)}<\sigma$, we have
    \begin{equation}\label{eq:loja-e}
        |\E(\gamma)-\E(\hat{\gamma})|^{1-\theta}\leq C\||\nabla\E(\gamma)|_g\|_{L^2(\dd s)}.
    \end{equation} 
\end{theorem}

\Cref{thm:loja-e} is not exactly as in \cite[Corollary 3.27]{pozzetta2022}. Indeed, \cite{pozzetta2022} considers the energy $\E_{\lambda}(\gamma)\vcentcolon=\lambda \mathcal{L}(\gamma)+\frac12\E(\gamma)$ with $\lambda=1$ instead of $\E$. But the arguments also apply in the case $\lambda=0$. Moreover, the norms appearing in \cite[Corollary 3.27]{pozzetta2022} make use of the intrinsic geometry of $\H^2$ while in \Cref{thm:loja-e} we measure the distance of $\hat{\gamma}$ to $\gamma$ with the Euclidean metric. However, this does not pose any problems since \eqref{eq:loja-e} is a local statement in a $W^{4,2}$-neighborhood of a reference curve $\hat{\gamma}$ --- in any such neighborhood, the Euclidean metric and the hyperbolic metric are equivalent with uniform constants. Lastly, using the parametrization invariance of the energy and the $L^2$-gradient, \Cref{thm:loja-e} assumes closeness to the critical point $\bar{\gamma}$ only after reparametrization.

With the specific reparametrizations by constant speed introduced in \Cref{sec:conv-conf-constr-will}, the \L ojasiewicz-Simon gradient inequality improves \Cref{cor:better-sub-convergence} to \Cref{thm:main-1d}. Indeed, in order to prove \Cref{thm:main-1d}, one can follow the proof of \Cref{thm:main-conf-constr-wf}, formally setting $\lambda=0$ and replacing \Cref{thm:loja-e-constr} by \Cref{thm:loja-e}. No additional difficulties arise from the fact that the hyperbolic length is non-constant in general since we assume it to be bounded in \Cref{thm:main-1d}.

\section{On a special class of conformally constrained Willmore tori}
\label{app:cmc-tori}

In this section, we collect properties of conformally constrained Willmore tori one obtains by rotating an $m$-fold circle of radius $r\in (0,1)$ centered at $(0,1)$ around the $x$-axis. More precisely, for $r\in(0,1)$, let $\gamma_r\colon\S^1\to\H^2$, $\gamma_r(x)\vcentcolon=(r\cos x,r\sin x+1)$. Then
\begin{equation}
    \Ll(\gamma_r)=\int_0^{2\pi} \frac{r}{r\sin x+1}\dd x = \frac{2\pi r}{\sqrt{1-r^2}}.
\end{equation}
Moreover, by \cite[Lemma~3.5]{dallacquaspener2018}, $\scurv_{\gamma_r}=|\curv_{\gamma_r}|_g\equiv \frac1r$ so that $\E(\gamma_r)=\frac{2\pi}{r\sqrt{1-r^2}}$. For $m\in\N$, define $\gamma_{m,r}=\gamma_r(m x)$. Clearly, with \eqref{eq:nabla-L2}
\begin{align}
    & \nabla\E(\gamma_{m,r}) - \lambda_r \curv_{\gamma_{m,r}} = 0\quad\text{for }\lambda_r\vcentcolon=r^{-2}-2\in (-1,\infty), \\
    & \Ll(\gamma_{m,r})=\frac{2\pi m r}{\sqrt{1-r^2}}, \quad \E(\gamma_{m,r})=\frac{2\pi m}{r\sqrt{1-r^2}} \quad\text{and}\quad T(\gamma_{m,r})=m.\label{eq:LET_gamma_mr}
\end{align}
In particular, for $0<r<\frac{1}{\sqrt{2}}$ and $m\in\N$ arbitrary, we have $\Ll(\gamma_{m,r}) < 2\pi m$. 

\begin{lemma}\label{rem:circular-elastica}
    Given any rectangular conformal class $\vec\omega\in \{0\}\times [1,\infty)$ and any $E>0$, for all $m\in\N$ sufficiently large, there exist $r=r(m)\in (0,1/\sqrt{2})$ such that $\mathfrak c(f_{\gamma_{m,r}})=\vec\omega$, $\E(\gamma_{m,r})>E$, and $\Ll(\gamma_{m,r})<2\pi|T(\gamma_{m,r})|= 2 \pi m$.
\end{lemma}
\begin{proof}
    Given any $\vec\omega$ as in the statement, define $L_{\vec\omega}\vcentcolon=2\pi |\vec\omega|$. By \eqref{eq:conf-class-of-torus-of-rev}, $\Ll(\gamma)=L_{\vec\omega}$ implies $\mathfrak c(f_\gamma)=\vec\omega$. By \eqref{eq:LET_gamma_mr} and continuity, for each $m\in\N$ with $2\pi m > L_{\vec\omega}$, there exists $r=r(m)\in (0,1/\sqrt{2})$ with $\Ll(\gamma_{m,r(m)})=L_{\vec\omega}$. So choosing $m \in \N$
    sufficiently large such that $2\pi m>L_{\vec\omega}$ and $\sqrt{2}\cdot 2\pi m>E$, with $r=r(m)$, \eqref{eq:LET_gamma_mr} yields $
        \Ll(\gamma_{m,r})=L_{\vec\omega} < 2\pi m$ and $\E(\gamma_{m,r})\geq \sqrt{2}\cdot 2\pi m>E. $
\end{proof}

\begin{lemma}\label{lem:circle_stability}
    Let $m\in\N$ and $r\in(0,1)$ such that $\lambda_r=r^{-2}-2\in (-1,(m^2-2m-1)/(1 + 2 m))$. With the notation of \Cref{rem:circular-elastica}, there exists a smooth family of immersions $\gamma\colon(-\varepsilon,\varepsilon)\times\S^1\to\H^2$ with
    \begin{equation}
        \gamma(0,\cdot)=\gamma_{m,r},\quad \partial_t\Ll(\gamma(t,\cdot))\equiv 0,\quad \partial_{t,0}\E(\gamma(t,\cdot))=0\quad\text{and}\quad \partial_{t,0}^2 \E(\gamma(t,\cdot)) < 0.
    \end{equation}
\end{lemma}
\begin{proof}
    Without loss of generality, for $L=2\pi mr/\sqrt{1-r^2}=2\pi m/\sqrt{1+\lambda_r}$, write $\S^1=\R / L\Z$ and suppose that $\gamma_{m,r}$ is parametrized by arc-length. Then define
    \begin{equation}
        \phi(s)\vcentcolon =\cos\big(\sqrt{1+\lambda_r} \cdot \frac{m+1}{m} s\big), \quad \vec\phi(s)=\phi(s)\cdot \vec{n}_{\gamma_{m,r}}(s)\quad\text{for $s\in\S^1$}.
    \end{equation}
    By \Cref{lemma:evo}, we compute for the second variation of $\mathcal L$
    \begin{equation}\label{eq:sec-var-length}
        \mathcal L''(\gamma_{m,r}).(\vec\phi,\vec\phi) = \int_{\S^1} (\partial_s\phi)^2+\phi^2\dd s.
    \end{equation} 
    Since $\int_{\S^1}\langle \curv_{\gamma_{m,r}},\vec\phi\rangle_g\dd s = \frac{1}{r} \int_{\S^1} \phi(s)\dd s=0$, as in the proof of \Cref{prop:conf-will-grad-axisym}, one shows that there exists a length-preserving smooth family of immersions $\gamma\colon(-\varepsilon,\varepsilon)\times\S^1\to\H^2$ with $\gamma(0,\cdot)=\gamma_{m,r}$ and $\partial_{t,0}\gamma(t,\cdot)=\vec\phi$. Using \eqref{eq:1-var-elen}, $\partial_{t,0}\E(\gamma(t,\cdot)) = \lambda_r\sqrt{2+\lambda_r}\int_{\S^1}\phi(s)\dd s=0$.  Differentiating $\Ll(\gamma(t,\cdot))=L$ yields
    \begin{equation}
        0 = \partial_{t,0}^2\Ll(\gamma(t,\cdot)) = \partial_{t,0}\mathcal L'(\gamma(t,\cdot)).\partial_t\gamma = \mathcal L''(\gamma_{m,r})(\vec\phi,\vec\phi) + \mathcal L'(\gamma_{m,r}).\partial_{t,0}^2\gamma.
    \end{equation}
    Therefore, using $\E'(\gamma_{m,r})= -\lambda_r \mathcal L'(\gamma_{m,r})$, we compute
    \begin{align}
        \partial_{t,0}^2 \E(\gamma(t,\cdot)) &= \E''(\gamma_{m,r})(\vec\phi,\vec\phi) + \E'(\gamma_{m,r}).\partial_{t,0}^2\gamma = \mathcal{E}''(\gamma_{m,r})(\vec\phi,\vec\phi)  - \lambda_r \mathcal{L}'(\gamma_{m,r}).\partial_{t,0}^2\gamma \\
        &= \E''(\gamma_{m,r})(    \vec\phi,\vec\phi) + \lambda_r \cdot \mathcal L''(\gamma_{m,r})(\vec\phi,\vec\phi).\label{eq:dt02e-length-constrained}
    \end{align}
    Similarly as \cite[Equation~(1.4)]{langersinger1984}, we have
    \begin{align}
        \E''(\gamma_{m,r})(\vec\phi,\vec\phi) &= \int_{\S^1} 2(\partial_s^2\phi)^2 - (6+5\lambda_r)(\partial_s\phi)^2 + (4+\lambda_r(5+2\lambda_r))\phi^2\dd s.
    \end{align}
    Thus, together with \eqref{eq:sec-var-length} and \eqref{eq:dt02e-length-constrained}, we obtain
    \begin{equation}
        \partial_{t,0}^2 \E(\gamma(t,\cdot)) = \int_{\S^1} 2(\partial_s^2\phi)^2 - (6+4\lambda_r)(\partial_s\phi)^2 + (4+\lambda_r(6+2\lambda_r))\phi^2\dd s.
    \end{equation}
    Using $\int_{\S^1}\cos^2\big(\sqrt{1+\lambda_r} \cdot \frac{m+1}{m} s\big)\dd s=\int_{\S^1}\sin^2\big(\sqrt{1+\lambda_r} \cdot \frac{m+1}{m} s\big)\dd s = \frac{m\pi}{\sqrt{1+\lambda_r}}$, we get
    \begin{align}
        \partial_{t,0}^2 \E(\gamma(t,\cdot)) &= \frac{m\pi}{\sqrt{1+\lambda_r}}\Big(2(1+\lambda_r)^2\frac{(m+1)^4}{m^4}-(6+4\lambda_r)(1+\lambda_r)\frac{(m+1)^2}{m^2}+4+\lambda_r(6+2\lambda_r)\Big)\\
        &= \frac{2\pi}{\sqrt{1+\lambda_r}}\frac{(\lambda_r+1) (2 m+1) \left( \lambda_r (2m+1)-m^2+2 m+1\right)}{m^3} < 0
    \end{align}
    since $-1<\lambda_r < \frac{m^2-2m-1}{2m+1}$.
\end{proof}

As claimed in \Cref{ex:circle}, we now prove that the Willmore flow does not preserve circles.

\begin{lemma}\label{lem:wf-and-round-circles}
    Consider a Willmore flow $\gamma\colon[0,T)\times\S^1\to\H^2$. If there exists a non-trivial interval $I\subseteq [0,T)$ such that $\gamma(t,\cdot)$ parametrizes a round circle for each $t\in I$, then $\gamma$ is stationary.
\end{lemma}
\begin{proof}
    There is a smooth family of diffeomorphisms $\phi\colon[0,T)\times\S^1\to\S^1$ of $\S^1$ such that
    \begin{equation}
        \tilde\gamma(t,x)\vcentcolon=\gamma(t,\phi(t,x))=a(t) + r(t)\begin{pmatrix}
            \cos x\\\sin x
        \end{pmatrix}
    \end{equation}
    where $a(t)=(a^{(1)}(t),a^{(2)}(t))\colon [0,T)\to \H^2$ and $r\colon[0,T)\to (0,\infty)$ are smooth with $r<a^{(2)}$. By direct computations and \cite[Lemma~3.1]{dallacquaspener2018},
    \begin{align}
        \vec n_{\tilde\gamma(t)}(x) &= - (a^{(2)}(t)+r(t)\sin(x)) \begin{pmatrix}
            \cos x\\\sin x
        \end{pmatrix},\quad\curv_{\tilde\gamma(t)}(x) = \frac{a^{(2)}(t)}{r(t)}\vec n_{\tilde\gamma(t)}(x),\\
        \nabla\E(\tilde\gamma(t,\cdot))(x) &= \Big(\frac{(a^{(2)}(t))^2}{r(t)^2}-2\Big)\curv_{\tilde\gamma(t)}.
    \end{align}
    Since $\gamma$ is a Willmore flow, see \Cref{def:will-flow}, we therefore have
    \begin{align}
        - & \frac{\partial_ta^{(1)}(t)\cos x + \partial_ta^{(2)}(t)\sin x + \partial_tr(t)}{a^{(2)}(t)+r(t)\sin x} = \langle \partial_t\tilde\gamma(t,x),\vec n_{\tilde\gamma(t)}(x)\rangle_g \\
        &\quad= - \frac{1}{4(a^{(2)}(t)+r(t)\sin(x))^4} \langle \nabla\E(\tilde\gamma(t,\cdot))(x), n_{\tilde\gamma(t)}(x)\rangle_g \\
        &\quad= -\frac{1}{4(a^{(2)}(t)+r(t)\sin(x))^4}\frac{a^{(2)}(t)}{r(t)}\Big(\frac{(a^{(2)}(t))^2}{r(t)^2}-2\Big)
    \end{align}
    for all $0\leq t<T$ and $x\in\S^1$. Multiplying by $-(a^{(2)}(t)+r(t)\sin x)$ and differentiating in $x$ yields
    \begin{equation}
        -\partial_ta^{(1)}(t)\sin x + \partial_ta^{(2)}(t)\cos x = \frac{-3 a^{(2)}(t) \cos(x) }{4(a^{(2)}(t)+r(t)\sin(x))^4} \Big(\frac{(a^{(2)}(t))^2}{r(t)^2}-2\Big).
    \end{equation}
    With $x=\frac12\pi$, we get $\partial_ta^{(1)}\equiv 0$. So, for $x\in\S^1\setminus\{\pm\frac12\pi\}$,
    \begin{equation}\label{eq:app-circ-1}
        \partial_ta^{(2)}(t) (a^{(2)}(t)+r(t)\sin x)^4  = -\frac{3 a^{(2)}(t)}{4} \Big(\frac{(a^{(2)}(t))^2}{r(t)^2}-2\Big)
    \end{equation}   
    and differentiating in $x$ and evaluating in $x=0$ yields $
        0=\partial_ta^{(2)}\cdot( a^{(2)})^3\cdot r$, i.e., $\partial_ta^{(2)}\equiv 0.$  
    Revisiting \eqref{eq:app-circ-1}, we get $\frac{(a^{(2)})^2}{r^2}\equiv 2$,
    i.e.,  $\tilde\gamma(t,\cdot)$ is a Clifford torus for each time $t$. Thus, $\nabla\E(\gamma(t,\cdot))=0$ for all $t\in[0,T)$ so that $\gamma$ is stationary by \eqref{eq:wf-eq}.
\end{proof}

\section*{Acknowledgements}
The authors would like to thank Reiner Schätzle for helpful discussions.
This research was funded in whole, or in part, by the Austrian Science Fund (FWF), grant number \href{https://doi.org/10.55776/ESP557}{10.55776/ESP557}.

\bibliographystyle{abbrv}
\bibliography{biblio}

\end{document}